\documentclass[10pt]{article}
\usepackage[english]{babel}
\usepackage{natbib}
\usepackage{url}
\usepackage{inputenc}
\usepackage{amsfonts}
\usepackage{amsmath}
\usepackage{empheq}
\usepackage{graphicx}
\usepackage{subcaption}
\graphicspath{{images/}}
\usepackage{parskip}
\usepackage{fancyhdr}
\usepackage{bbold}
\usepackage{amssymb}
\usepackage{vmargin}
\usepackage{array}
\usepackage{amsthm}
\usepackage[ruled,vlined]{algorithm2e}
\usepackage{hyperref}
\usepackage{caption}
\usepackage{float}
\usepackage{fourier-orns}
\usepackage{listings}
\usepackage{epsfig}
\usepackage{amsmath}

\usepackage[shortlabels]{enumitem}
\mathtoolsset{showonlyrefs=true}

\newtheorem{assumption}{Assumption}
\newtheorem{theorem}{Theorem}[section]
\newtheorem{corollary}{Corollary}[section]
\newtheorem{lemma}[theorem]{Lemma}
\newtheorem{remark}{Remark}[section]
\theoremstyle{definition}
\newtheorem{definition}{Definition}[section]

\usepackage{listings}

\newcommand{\tu}{\tilde{u}}

\newcommand{\om}{{\Omega_{0}}}
\newcommand{\R}{{\mathbb{R}}}

\newcommand{\bydef}{\stackrel{\mbox{\tiny\textnormal{\raisebox{0ex}[0ex][0ex]{def}}}}{=}}

\usepackage[dvipsnames]{xcolor}

\makeatletter
\newcommand\subsubsubsection{\@startsection{paragraph}{4}{\z@}{-2.5ex\@plus -1ex \@minus -.25ex}{1.25ex \@plus .25ex}{\normalfont\normalsize\bfseries}}
\newcommand\subsubsubsubsection{\@startsection{subparagraph}{5}{\z@}{-2.5ex\@plus -1ex \@minus -.25ex}{1.25ex \@plus .25ex}{\normalfont\normalsize\bfseries}}
\makeatother
\title{Proving symmetry of localized solutions and application to dihedral patterns in the planar Swift-Hohenberg PDE}
\author{
Dominic Blanco
\footnote{McGill University, Department of Mathematics and Statistics, 805 Sherbrooke Street West, Montreal, QC, H3A 0B9, Canada. {\tt dominic.blanco@mail.mcgill.ca}}
\and
Matthieu Cadiot
\footnote{CMAP, CNRS, Ecole polytechnique, Institut Polytechnique de Paris, 91120 Palaiseau, France. {\tt matthieu.cadiot@polytechnique.edu}}
}
\begin{document}

\maketitle
\begin{abstract}
    In this article, we extend the framework developed in \cite{unbounded_domain_cadiot} to allow for rigorous proofs of existence of smooth, localized solutions in semi-linear partial differential equations possessing both space and non-space group symmetries. We demonstrate our approach on the Swift-Hohenberg model. In particular, for a given symmetry group $\mathcal{G}$, we construct a natural Hilbert space $H^l_{\mathcal{G}}$ containing only functions with $\mathcal{G}$-symmetry. In this space, products and differential operators are well-defined allowing for the study of autonomous semi-linear PDEs. Depending on the properties of $\mathcal{G}$, we derive a Newton-Kantorovich approach based on the construction of an approximate inverse around an approximate solution, $u_0$. More specifically, combining a meticulous analysis and computer-assisted techniques, the Newton-Kantorovich approach is validated thanks to the computation of some explicit bounds. The strategy for constructing $u_0$, the approximate inverse, and the computation of these bounds will depend on the properties of $\mathcal{G}$ and its maximal square lattice space subgroup, $\mathcal{H}$. More specifically, we consider three cases: $\mathcal{G}$ is a space group which can be represented on the square lattice, $\mathcal{G}$ is not a space group which can be represented on the square lattice and the symmetry of $\mathcal{H}$ isolates the solution, and where $\mathcal{G}$ is not a space group which can be represented on the square lattice and the symmetry of $\mathcal{H}$ does not isolate the solution. We demonstrate the methodology on the 2D Swift-Hohenberg PDE by proving the existence of various dihedral localized patterns. The algorithmic details to perform the computer-assisted proofs can be found on Github \cite{julia_cadiot_blanco_symmetry}.
    \end{abstract}

\begin{center}
\end{center}
\section{Introduction}
In this paper, we develop a methodology for constructively proving the existence of symmetric solutions to PDEs. More specifically, we will focus on localized solutions on $\R^m$, that is solutions vanishing at infinity. This will include solutions possessing both space and non-space group symmetries. By space group (or crystallographic group) symmetry, we mean symmetries which combine the translational symmetry of a lattice together with other elemens such as directional flips, rotation, and screw axes (cf. \cite{JB_symmetries_1}). A complete presentation of all 230 space group and their properties is available at \cite{crys}.  We will  illustrate such an approach in the case of localized solutions in the planar Swift-Hohenberg PDE, for which we establish dihedral symmetries. In this work, we  consider  a class of autonomous semilinear PDEs of the form
\begin{align}\label{eq : f(u)=0 on Hl_G}
     &\mathbb{F}(u) = 0, ~~ \text{ where } \mathbb{F}(u) = \mathbb{L}u + \mathbb{G}(u)\\
     \lim_{|x| \to \infty} &u(x) = 0,~\text{and}~u: \mathbb{R}^m \to \mathbb{R} ~ \text{is}~\mathcal{G}\text{-symmetric}.
\end{align}
In the above, $\mathcal{G}$ is a symmetry group, $\mathbb{L}$ is a linear differential operator with constant coefficients, $\mathbb{G}$ is a polynomial operator of order $N_{\mathbb{G}} \in \mathbb{N}$  where $N_{\mathbb{G}} \geq 2$. That is, we can decompose it as a finite sum 
 \vspace{-.3cm}
\begin{equation}\label{def: G and j}
     \mathbb{G}(u) \bydef \displaystyle\sum_{i = 2}^{N_{\mathbb{G}}}\mathbb{G}_i(u) 
     \vspace{-.2cm}
\end{equation}
where $\mathbb{G}_i$ is a sum of monomials of degree $i$ in the components of $u$ and some of its partial derivatives.
In particular, for $i \in \{2,\dots,N_{\mathbb{G}}\}$, $\mathbb{G}_i$ can be decomposed as follows
\vspace{-.2cm}
\[
\mathbb{G}_i(u) \bydef  \mathbb{G}_{i,1}u \dots \mathbb{G}_{i,i}u
\vspace{-.2cm}
\]
 where $\mathbb{G}_{i,j}$ is a linear differential operator with constant coefficients for all  $j \in \{1, \dots, i\}$. $\mathbb{G}(u)$ depends on the derivatives of $u$ but we only write the dependency in $u$ for simplification. 
 \par One model that fits the form of \eqref{eq : f(u)=0 on Hl_G} is the Swift-Hohenberg equation (SH). We wish to prove the existence, local uniqueness, and symmetry of localized stationary patterns in SH. By localized pattern, we mean a stationary solution $u = u(x)$ where $u(x) \to 0$ as $|x| \to \infty$. The SH equation is defined as
\begin{align}\label{eq : swift_hohenberg}
    u_t = -((I_d + \Delta)^2 u + \mu u + \nu_1 u^2 + \nu_2 u^3), ~~ u = u(x,t), ~~x \in \mathbb{R}^2
\end{align}
where $\mu > 0$ and $(\nu_1,\nu_2) \in \mathbb{R}^2$. Note that the sign of $\mu$ is essential in this paper while $\nu_1$ and $\nu_2$ can be chosen freely. Localized patterns in SH have been well documented, particularly the existence of dihedral patterns in 2D (cf. \cite{jason_review_paper,sh_cadiot,jason_spot_paper,jason_ring_paper,sandstede_paper,hexagon2021lloyd,hexagon2008,squareSakaguchi_1997})
By dihedral, we mean the pattern possesses the symmetry of the dihedral group $D_j$ with group presentation 
\begin{align}
    D_j = < r,s ~ | ~ r^n = s^2 = 1, ~ rs = sr^{-1} >,\label{def : Dj presentation}
\end{align}
as defined similarly in \cite{dummit_algebra}. We denote by $D_j$ the symmetry group of the $j$-gon. The group comprises of rotations by $\frac{2\pi k}{n}$ for $k = 0,\dots,n-1$ and reflections. Various numerical and analytical studies of SH have led to a deeper understanding of dihedral planar patterns such as hexagonal \cite{hexagon2021lloyd,hexagon2008} and square \cite{squareSakaguchi_1997} patterns. More generally, the works of \cite{jason_review_paper}, \cite{jason_spot_paper}, and \cite{jason_ring_paper} demonstrate an approach to numerically compute dihedral solutions. By using radial coordinates, the authors provide a method to construct approximate solutions using a Galerkin projection starting with $\mu$ small. By then performing continuation, one can obtain localized planar patterns of various dihedral symmetries. These can form both spot patterns (see \cite{jason_spot_paper}) and ring patterns (see \cite{jason_ring_paper}). Proofs in SH have been obtained before. For instance, under certain hypotheses, \cite{ladder2009Beck} obtains proofs of existence of some patterns. Specifically, in \cite{ladder2009Beck,radial2019Bramburger,radial2009,hexagon2021lloyd,hexagon2008, existence2019sandstede, stability1997Mielke}, for $\mu$ small, several existence results have been obtained. In particular, by using a bifurcation argument, a proof of existence of branches of patterns with $\mu \in (0, \mu^*)$, for some $\mu^*>0$. This argument relies on the implicit function theorem or various fixed-point theorems. The restriction on the value of $\mu$ was removed in \cite{olivier_radial}. In the latter, the authors present a method for rigorously proving radially symmetric solutions that does not depend on the value of $\mu$. Using the radially symmetric ansatz,  the PDE is transformed into an ordinary differential equation. The approach then relies on  a rigorous enclosure of the center stable manifold by using a Lyapunov-Perron operator. This allows one to solve a boundary value problem on $(0,\infty)$. However, this method is restricted to radially symmetric solutions. In this paper, we will explore the rich variety of dihedral patterns which naturally arises in the SH PDE.

 In \cite{sh_cadiot}, the authors obtained rigorous proofs of three localized patterns in SH that are not radially symmetric. Furthermore, this work removed the restrictions on the value $\mu$. Indeed, their approach only requires $\mu > 0$ so that the linear part of the equation is invertible (see Assumption \ref{ass:A(1)}).  The work of the authors of \cite{sh_cadiot} is based on a previous work by the same authors, \cite{unbounded_domain_cadiot}. The authors provide a general method for proving the existence and local uniqueness of localized solutions to semilinear autonomous PDEs on $\mathbb{R}^m$. The approach is based on Fourier series and describes the necessary tools to construct an approximate solution, $u_0$, and an approximate inverse $\mathbb{A}$ of the linearization $D\mathbb{F}(u_0)$ about $u_0$. Then, by using a Newton-Kantorovich approach, one can prove the existence of a true solution to \eqref{eq : f(u)=0 on Hl_G} in a vicinity of $u_0$, using the Banach fixed point theorem. This involves the computation of specific bounds, which are established rigorously in a computer-assisted manner using the arithmetic on intervals  \cite{julia_interval}. An application to the 1D Kawahara equation was provided in the same paper. In \cite{sh_cadiot}, the method of \cite{unbounded_domain_cadiot} was applied to 2D SH. Additionally, the authors were able to construct $D_2$ symmetric solutions using invariant functions under reflections about the $x$ and $y$ axis. \\
 The method was then extended further by the authors of \cite{gs_cadiot_blanco}. In \cite{unbounded_domain_cadiot}, the method was presented for scalar PDEs. In \cite{gs_cadiot_blanco}, the authors generalized the approach to systems of PDEs. The approach was demonstrated on the 2D Gray-Scott system of equations where, along with the proof of the solution, a proof of $D_4$-symmetry was obtained. The results are presented in Theorems 6.1, 6.2, 6.3, and 6.4 of the aforementioned paper. In \cite{whitham_cadiot}, the method of \cite{unbounded_domain_cadiot} was extended to non-local equations. The approach was used to treat the existence and stability of solitary waves in the capillary-gravity Whitham equation. The approach needed to be modified to accommodate the Fourier multiplier operator, which is non local in this case. This allowed for the proof of multiple even solitary waves, as well as their spectral stability.
 
The use of symmetry in computer-assisted proofs has already been well-documented in previous works. Two of the works we would like to mention are \cite{jp_saddle_node} and \cite{cyclic_sym}. In \cite{jp_saddle_node}, the authors developed rigorous approaches for verifying different bifurcations. The approach is computer assisted, and relies on translating the desired bifurcation into the zeros of an augmented system. One of their applications was for proving $\mathbb{Z}_2$ symmetry breaking pitchfork bifurcations where $\mathbb{Z}_2$ is the symmetry group
\begin{align}
    \mathbb{Z}_2 \bydef <g ~ | ~ g^2 = 1>\label{def : Z2group}
\end{align}
defined similarly as in \cite{dummit_algebra}. This work was then extended in \cite{cyclic_sym}, where the symmetry breaking pitchfork bifurcation was generated by a cyclic group symmetry. In particular, the authors showed that the solution $u$ had the symmetry $u(x) = -u(x + \frac{1}{n})$ for all $x \in \mathbb{R}$ and some $n \in \mathbb{N}$ whereas the eigenfunction, $\varphi$, did not possess this symmetry. The authors needed to develop the theoretical foundation for rigorous verification of a pitchfork bifurcation generated by a cyclic symmetry group. The authors then reformulate this result as the zeros of an augmented system, similar to what was done in \cite{jp_saddle_node}. Then, by using a computer-assisted approach, the authors verify a zero of this map providing the result. While the methods were presented for the diblock copolymer model, the authors remark that their approach can likely be generalized to other parabolic PDEs.

\par The authors of \cite{whitham_cadiot,gs_cadiot_blanco,unbounded_domain_cadiot} and \cite{sh_cadiot} were often able to validate some symmetry of the localized patterns they proved. More specifically, this was achieved by constructing an approximate solution $u_0$ thanks to a symmetric Fourier series. Translating the symmetry on the Fourier coefficients then provides a natural computer-assisted method to create symmetry. In fact, such techniques on symmetric Fourier coefficients have been deeply developed in \cite{JB_symmetries_1} and \cite{JB_symmetries_2}. More specifically, the authors describe a way to build a Fourier series that possesses the desired symmetry.  The methods developed in \cite{JB_symmetries_1} and \cite{JB_symmetries_2} are applicable to any space group in any dimension. The authors use an algorithm to compute a reduced set of Fourier indices under the symmetry of a space group. Furthermore, the method allows one to determine the specific relations amongst the coefficients. In particular, it relies on computing a set which contains one element from each orbit (see \cite{gallianalgebra} for a definition) of the given space group. Such a set is called a fundamental domain, denoted $J_{\mathrm{dom}}(\mathcal{G})$. To demonstrate the useful of $J_{\mathrm{dom}}(\mathcal{G})$, suppose $u$ is a Fourier series with Fourier coefficients $(u_n)_{n \in \mathbb{Z}^m}$. If an index, $n$, is in the fundamental domain, then the other elements in its orbit (in the sense of group action) will have the same value. That is, for any $g \in \mathcal{G}$, $u_n = u_{g \cdot n}$ where $\cdot$ denotes the group action. This relation already reduces the number of Fourier indices one must store. Additionally, the authors of \cite{JB_symmetries_1} and \cite{JB_symmetries_2} compute what they refer to as the trivial set. The trivial set, denoted $J_{\mathrm{triv}}(\mathcal{G})$ contains all Fourier indices which have values of $0$. That is, if $u_n = 0$ for some $n$, then this $n$ is in the trivial set. Then, the authors define the set 
\begin{align}
    J_{\mathrm{sym}}(\mathcal{G}) \bydef \mathbb{Z}^m \setminus J_{\mathrm{triv}}(\mathcal{G}).
\end{align}
The set $J_{\mathrm{sym}}(\mathcal{G})$ contains all coefficients not in the trivial set. Since the trivial set contains all the coefficients which are $0$, $J_{\mathrm{sym}}(\mathcal{G})$ contains all those coefficients which are non-zero. Once $J_{\mathrm{dom}}(\mathcal{G})$ and $J_{\mathrm{sym}}(\mathcal{G})$ are obtained, one can define the reduced set of Fourier coefficients
 \begin{align}
     J_{\mathrm{red}}(\mathcal{G}) \bydef J_{\mathrm{dom}}(\mathcal{G}) \cap J_{\mathrm{sym}}(\mathcal{G}).\label{def : reduced set}
 \end{align}
 Indeed, the relations determined by $J_{\mathrm{dom}}(\mathcal{G})$ and $J_{\mathrm{triv}}(\mathcal{G})$ allow one to write a Fourier series only featuring indices in $J_{\mathrm{red}}(\mathcal{G})$. Such a Fourier series will enforce the needed relations numerically. This allows one to build an approximate solution that is necessarily of the symmetry enforced. The authors then demonstrate their approach on the Ohta-Kawasaki Problem in 3D. In \cite{JB_symmetries_1}, the authors obtain rigorous proofs of existence and local uniqueness of periodic solutions to 3D Ohta-Kawasaki with a SG229 and SG230 symmetry. We refer the interested reader to \cite{crys,JB_symmetries_1} for a definition of these groups and their properties. In \cite{JB_symmetries_2}, they generalize the approach and allow for one to attempt rigorous proofs for any space group symmetry. Since the rigorous approach used relies on a contraction argument, it can be shown that the true solution also has the symmetry of the approximate solution. The method was also applied in \cite{van2021spontaneous}, where symmetries on the range of the function were also considered. The same algorithm can be applied and the computational details are outlined in Sections 5 and 6.

 % In fact,  the ideas developed by the authors of \cite{JB_symmetries_1} and \cite{JB_symmetries_2} were already partially used in \cite{unbounded_domain_cadiot}, \cite{sh_cadiot}, \cite{gs_cadiot_blanco}, and \cite{whitham_cadiot} to obtain the proof of (simple) symmetries. Indeed, the proof of the symmetry of the solution to the 1D Kawahara equation in \cite{unbounded_domain_cadiot} was obtained by enforcing even ($\mathbb{Z}_2$) symmetry in the Fourier coefficients of $u_0$. Similarly, the proof of $D_2$ symmetry in \cite{sh_cadiot} was obtained by enforcing $D_2$ symmetry in the Fourier coefficients (that is cosine-cosine series).  The proof of $D_4$-symmetry in \cite{gs_cadiot_blanco} was possible due to the use of Fourier series with the symmetries of the square. Since squares tile the plane, the algorithm of \cite{JB_symmetries_1} and \cite{JB_symmetries_2} allows for one to construct an approximate solution and inverse of the linearization with $D_4$ symmetry enforced.  In this paper, we wish to push the usage of \cite{JB_symmetries_1} and \cite{JB_symmetries_2} further, and verify more complex symmetries, such as non space group symmetries.
 
\begin{figure}[H]
\centering
 \begin{minipage}{.33\linewidth}
  %\centering\epsfig{figure=gs_leaf_top_good.eps,width=\linewidth}
  \centering\epsfig{figure=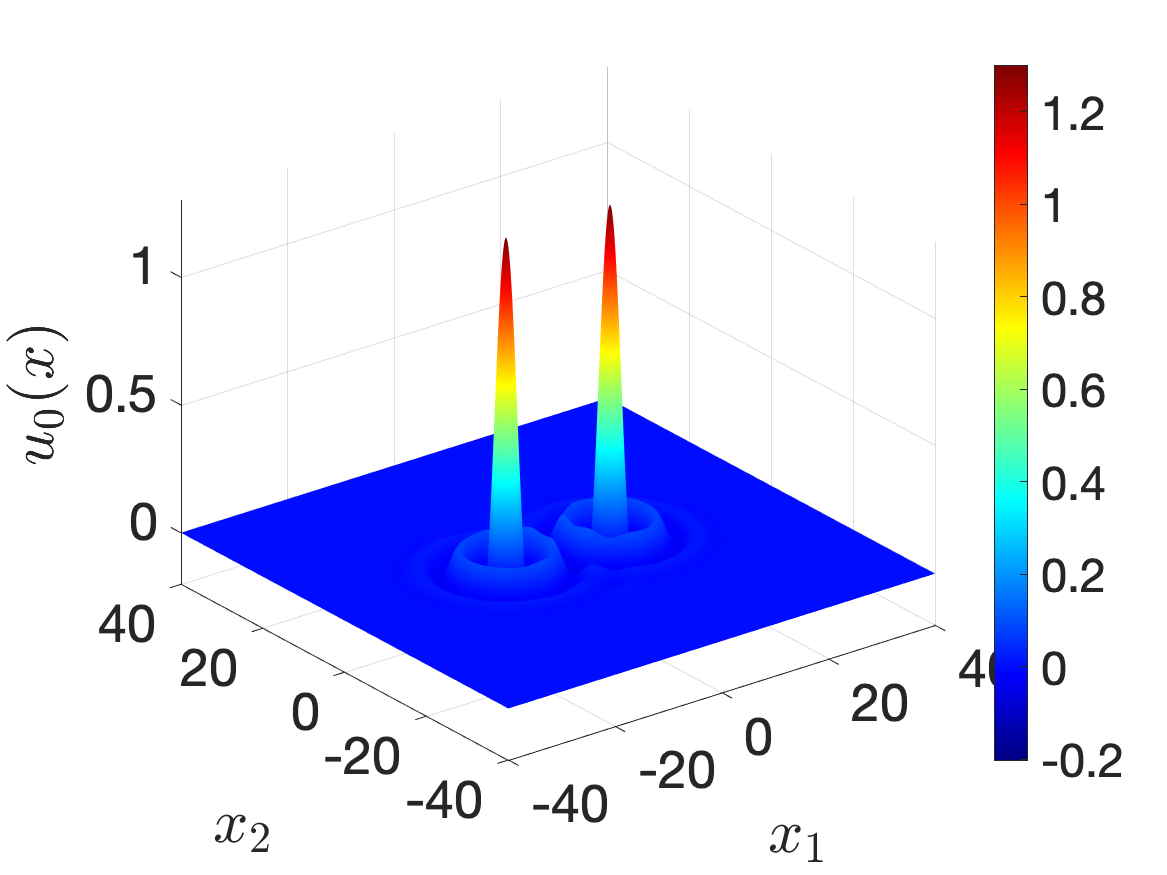,width=\linewidth}
  \end{minipage}%
 \begin{minipage}{.33\linewidth}
  \centering\epsfig{figure=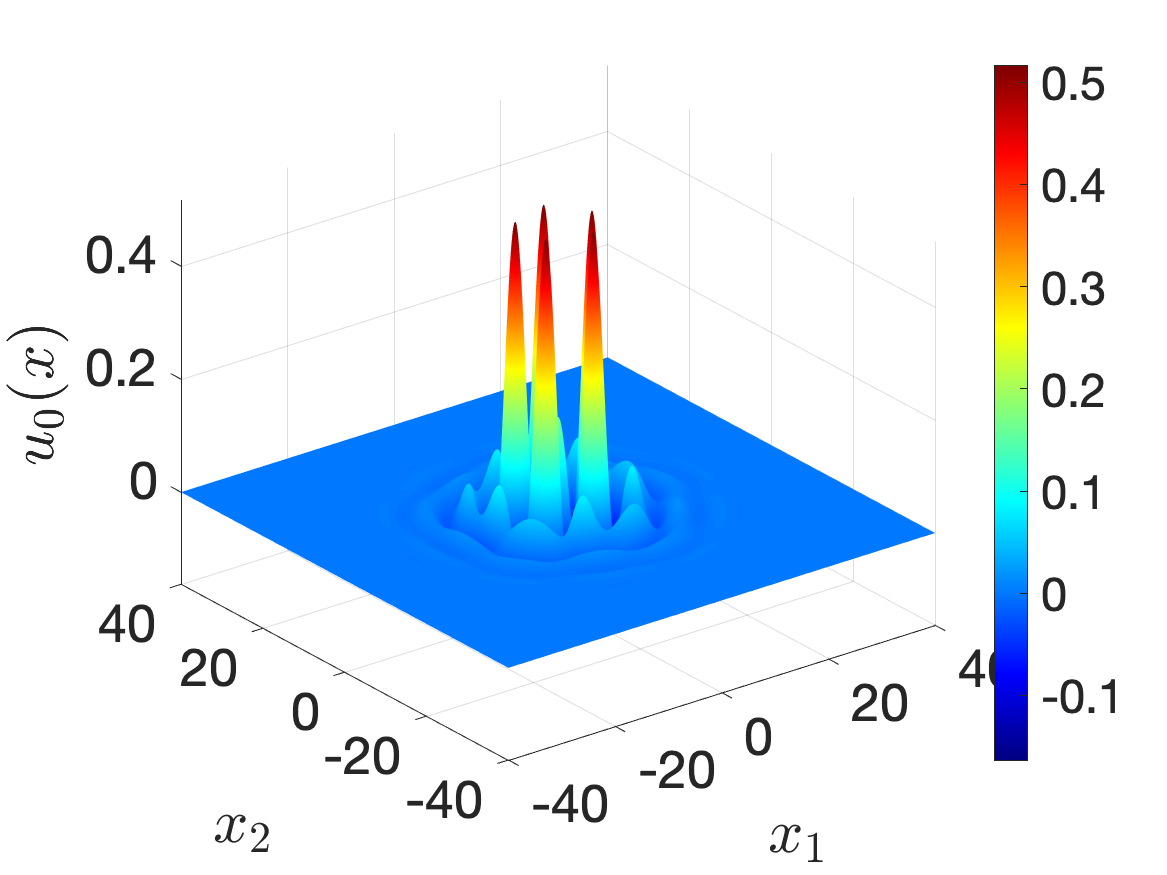,width=\linewidth}
 \end{minipage} %
 \begin{minipage}{.33\linewidth}
  \centering\epsfig{figure=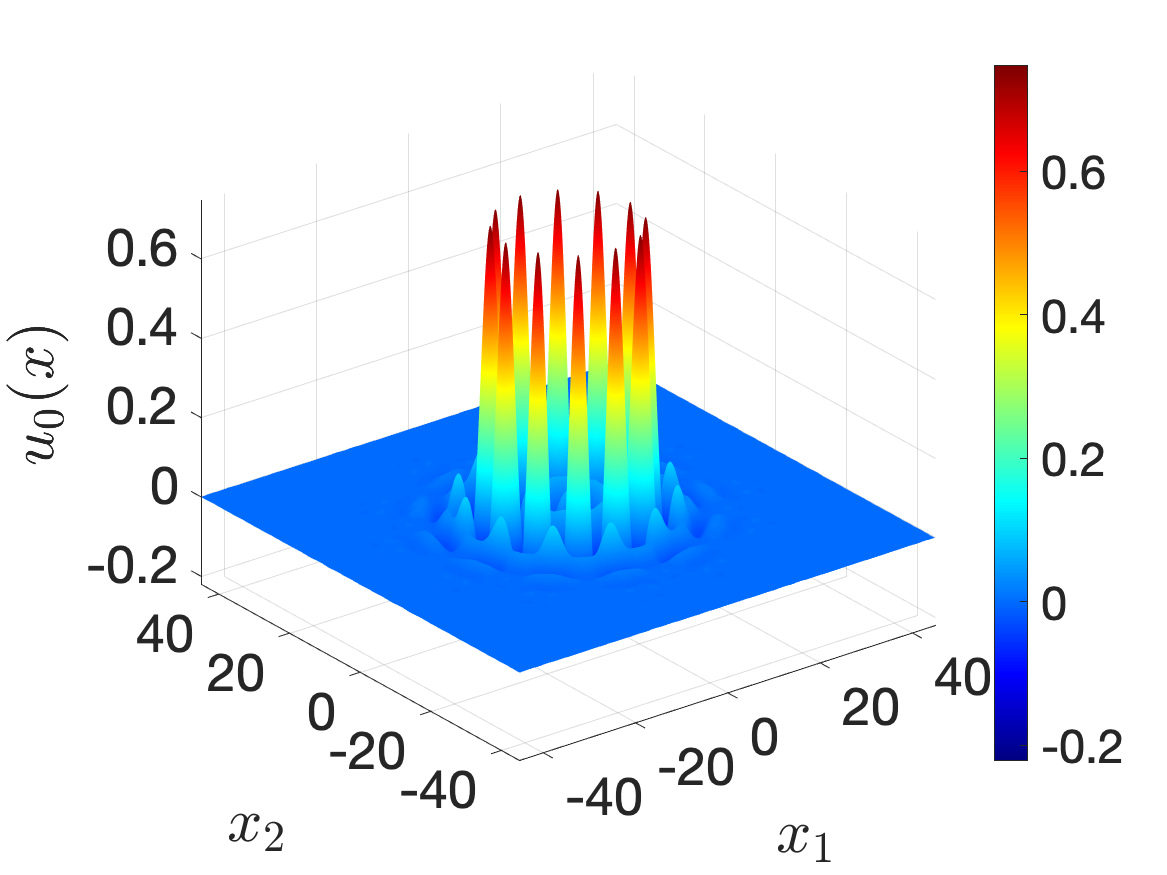,width=\linewidth}
 \end{minipage} 
 \caption{Plot of $u_0$ of an approximation of a $D_2$ Pattern (L), a $D_5$ Pattern (C) and a $D_{12}$ Pattern (R) in the planar Swift Hohenberg Equation.}
 \end{figure}\label{eq : fig_intro}%

As the approach of \cite{unbounded_domain_cadiot} for proving the existence of localized solutions on $\R^m$ relies on approximate objects defined thanks to Fourier series, our goal is to use the work of \cite{JB_symmetries_1} and \cite{JB_symmetries_2} (developed in the case of periodic solutions) in order to validate symmetries of localized solutions. Indeed, we  construct a $\mathcal{G}$-symmetric approximate solution $u_0$ and prove the existence of an actual solution to \eqref{eq : f(u)=0 on Hl_G} in a vicinity of $u_0$ thanks to a Newton-Kantorovich approach (similarly as what was achieved in \cite{unbounded_domain_cadiot}). Such an approach also requires the construction of an approximate inverse $\mathbb{A} : L^2_{\mathcal{G}} \to H^l_{\mathcal{G}}$ for $D\mathbb{F}(u_0)$ (where the symmetric spaces $L^2_{\mathcal{G}} $ and $ H^l_{\mathcal{G}}$ are defined in \eqref{def : L2 and Hl G}). Doing so, we prove the existence of multiple localized solutions, satisfying dihedral symmetries (cf. Figure \ref{eq : fig_intro} for instance).

The usage of the work developed in \cite{JB_symmetries_1} and \cite{JB_symmetries_2} is not direct and requires a meticulous adaptation. First, recall that \cite{JB_symmetries_1} and \cite{JB_symmetries_2} were developed for the study of periodic solutions, whereas we are interested in localized solutions on $\R^m$. However, as pointed out above, we aim at constructing an approximate solution $u_0$ thanks to Fourier series, which will rely on the aforementioned papers. More specifically, given an hypercube $\om = (-d,d)^m$
\begin{equation}
   u_0(x) = \mathbb{1}_{\om}(x) \sum_{n \in J_{\mathrm{red}}(\mathcal{G})} u_n \sum_{k \in \mathrm{orb}_{\mathcal{G}}(n)} e^{2\pi i \tilde{k} \cdot x}
\end{equation}
where $J_{\mathrm{red}}(\mathcal{G})$ is defined as in \eqref{def : reduced set} and $\mathrm{orb}_{\mathcal{G}}(n)$
is the orbit of $n \in \mathbb{Z}^m$ in $\mathcal{G}$ (cf. \cite{gallianalgebra}). More specifically, we define the orbit as \begin{align}
    \mathrm{orb}_{\mathcal{G}}(n) = \{g \cdot n, \ \mathrm{for \ all} \ g \in \mathcal{G}\}.\label{def : orbit}
\end{align}
Now, we must draw a distinction between the symmetric functions on $\R^m$ which we can construct thanks to Fourier series, and the one that cannot. Firstly, since our goal will be to build a compactly supported function on $\mathbb{R}^m$, we must exclude symmetries which do not respect this, such as translations. In particular, the symmetry must be compatible with a compactly supported localized state. Focusing now on 2D, the distinction {will also rely} on the tiling of the plane. Indeed, since regular squares tile the plane, we are able to construct $D_2$ and $D_4$ symmetric and compactly supported functions  on $\R^2$ thanks to Fourier series (cf. Section \ref{sec : HisG}). This is achieved using the map $\gamma^\dagger_{\mathcal{G}}$ defined in \eqref{def : gamma dagger}). For symmetries which we cannot express as Fourier series on the square lattice along with those which do not lead to a tiling of the plane (i.e. non-space group symmetries), we need to consider two cases. First, if $\mathcal{G}$ possesses $D_2$ or $D_4$ as a subgroup, we can use such a subgroup to represent the solution, and add in a second time the additional symmetries using a well chosen average (cf. Section \ref{sec : H isolate}). On the other hand, if   $\mathcal{G}$ does not possess $D_2$ or $D_4$ as a subgroup, we still use an average over the elements of the group in order to ensure symmetry, but now require the use of \emph{unfolding parameters} to isolate the solution (cf. Section \ref{sec : H nonisolate}). A variant of Lagrange multipliers, unfolding parameters are a useful tool for computer assisted proofs when one needs to isolate a solution. We refer the interested reader to \cite{jay_unfold1,jay_unfold4,jay_unfold3,jay_unfold2} for previous works using unfolding parameters and a computer-assisted approach. For each case, we describe the associated computer-assisted strategy for validating the  symmetry under study. As an application, we prove the existence of various dihedral localized solutions in the 2D Swift-Hohenberg PDE (cf. Section \ref{sec : proof of solutions}). In particular, we prove the existence of novel symmetric patterns, such as a $D_5$ symmetric solution. Although being intensively  observed numerically (cf.  \cite{jason_review_paper},\cite{jason_spot_paper}, \cite{jason_ring_paper} and \cite{hexagon2008} for instance), this is the first proof of such symmetric solutions away from perturbative regions.

At this point, we provide some notations and assumptions characterizing \eqref{eq : f(u)=0 on Hl_G}, which will allow to describe our approach.

\subsection{Notations and assumptions}

 Our goal is to apply a Newton-Kantorovich approach to rigorously prove the existence, local uniqueness, and symmetry of localized solutions to \eqref{eq : f(u)=0 on Hl_G}. Our approach relies on an application of the Banach fixed point theorem. In particular, it requires a meticulous choice of function spaces for the definition of $\mathbb{F}$, given in \eqref{eq : f(u)=0 on Hl_G}. For that purpose, we follow the set-up introduced in \cite{unbounded_domain_cadiot}, and we recall the required associated assumptions. First, we introduce relevant notation for the study of localized solutions.

We use the Lebesgue notation $L^2 = L^2(\R^m)$ or $L^2(\Omega_0)$ for a bounded domain $\om$.  More generally, $L^p = L^p(\R)$ is the usual $p$ Lebesgue space associated to its  norm $\| \cdot \|_{p}$, and $L^2$ is associated to its inner product $(\cdot,\cdot)_2$. Given a Banach space $X$,  denote by $\mathcal{B}(X)$ the space of bounded linear operators on $X$ and given $\mathbb{B} \in \mathcal{B}(L^2)$, denote by $\mathbb{B}^*$ the adjoint of $\mathbb{B}$ in $L^2.$

Denote by  $\mathcal{F} :L^2 \to L^2$ the \textit{Fourier transform} operator and
write $\mathcal{F}(f) \bydef \hat{f}$, where $\hat{f}$ is defined as 
\begin{align}\label{def : Fourier transform}
    \hat{f}(\xi) \bydef \int_{\R^m}f(x)e^{-i2\pi x\cdot \xi}dx
\end{align}
for all $\xi \in \R^m$. Similarly, the \textit{inverse Fourier transform} operator is expressed as $\mathcal{F}^{-1}$ and defined as $\mathcal{F}^{-1}(f)(x) \bydef \int_{\R^m}f(\xi)e^{i2\pi x\cdot \xi}d\xi$. In particular, recall the classical Plancherel's identity
\begin{equation}\label{eq : plancherel definition}
    \|f\|_2 = \|\hat{f}\|_2, \quad \text{for all }f \in L^2.
\end{equation}
Finally, for $f_1,f_2$, the continuous convolution of $f_1$ and $f_2$ is represented by $f_1*f_2$, and defined as 
\begin{equation*}
    (f_1*f_2)(x) = \int_{\mathbb{R}^m} f_1(x-y) f_2(y) dy, ~  \text{ for all } x \in \R^m.
\end{equation*}

Now, we introduce the required assumption characterizing \eqref{eq : f(u)=0 on Hl_G}. Additional context for such assumptions is detailed in \cite{unbounded_domain_cadiot}.

\begin{assumption}\label{ass:A(1)}
Assume that the Fourier transform of the linear operator $\mathbb{L}$ is given by 
\begin{equation} \label{eq:assumption1_on_L}
\mathcal{F}\big(\mathbb{L}u\big)(\xi) = l(\xi) \hat{u}(\xi), \quad \text{for all } u\in L^2,
\end{equation}
where $l(\cdot)$ is a polynomial in $\xi$. Moreover, assume that
\begin{equation} \label{eq:l>0}
|l(\xi)| >0, \qquad \text{for all } \xi \in \mathbb{R}^m.
\end{equation}
\end{assumption}

\begin{assumption}\label{ass : LinvG in L1}
For all $2 \leq i \leq N_{\mathbb{G}}$, $k \in J_i$ and $1 \leq p \leq i$, assume that there exists a polynomial $g_{i,p} : \R^m \to \mathbb{C}$ such that
\vspace{-.2cm}
\[
\mathcal{F}\bigg(\mathbb{G}_{i,p} u\bigg)(\xi) = g_{i,p}(\xi)\hat{u}(\xi),
\quad \text{for all } \xi \in \mathbb{R}^m, \text{ and  assume that } ~\frac{g_{i,p}}{l} \in L^1.
\vspace{-.2cm}
\]
\end{assumption}
Assumption \ref{ass:A(1)} is equivalent to assuming $\mathbb{L}$ is invertible whereas Assumption \ref{ass : LinvG in L1} is stronger than semilinearity as it assumes additional decay for each $g_{i,p}.$ Using Assumption \ref{ass:A(1)}, we define the Hilbert space  of function $H^l$ as 
\begin{align}
    H^l \bydef \biggl\{ u \in L^2 \ : \ \int_{\mathbb{R}^m} |\hat{u}(\xi)|^2 |l(\xi)|^2 d\xi < \infty\biggl\}\label{def : Hl}
\end{align}associated to its natural inner product $(\cdot,\cdot)_l$ and norm defined as
\begin{align}
    (u,v)_l \bydef \int_{\mathbb{R}^m} \hat{u}(\xi) \overline{\hat{v}(\xi)} |l(\xi)|^2 d\xi ~ \text{and} ~ \|u\|_{l} \bydef \|\mathbb{L} u\|_{2}.
\end{align}
In particular, note that $\mathbb{L} : H^l \to L^2(\R^m)$ is an isometric isomorphism and \cite{unbounded_domain_cadiot} provides that that $\mathbb{G} : H^l \to  H^1(\R^m)$ is smooth (under Assumption \ref{ass : LinvG in L1}). Consequently, we transform \eqref{eq : f(u)=0 on Hl_G} as a zero finding problem $\mathbb{F}(u) = 0$ with $u \in H^l$. We refine  $\mathbb{F}$ based on the properties of the group $\mathcal{G}$.

In fact, since we are interested about $\mathcal{G}$-symmetric functions, we define the following subspaces
\begin{align}
    &L^2_{\mathcal{G}} \bydef \{u \in L^2, ~\mathrm{such}~\mathrm{that}~ u(gx) = u(x), \mathrm{for}~\mathrm{all}~ g \in \mathcal{G},x \in \mathbb{R}^2\} \\
    &H^l_{\mathcal{G}} \bydef \{u \in H^l,~\mathrm{such}~ \mathrm{that}~ u(gx) = u(x), \mathrm{for}~\mathrm{all}~ g \in \mathcal{G}, x \in \mathbb{R}^2\}.\label{def : L2 and Hl G}
\end{align}
To study $\mathcal{G}$-symmetric solutions, we require our  zero finding problem $\mathbb{F}$ to be invariant under such a symmetry (so that it can naturally be defined on $H^l_{\mathcal{G}} \to L^2_{\mathcal{G}}$). This leads to the following.
\begin{assumption}\label{ass : L_G invariant}
 Assume that the restrictions $\mathbb{L} : H^l_{\mathcal{G}} \to L^2_{\mathcal{G}}$ and $\mathbb{G} : H^l_{\mathcal{G}} \to L^2_{\mathcal{G}}$ are well-defined.   
\end{assumption}
Assumption \eqref{ass : L_G invariant} ensures that the range of $\mathbb{L}$ and $\mathbb{G}$ are also of $\mathcal{G}$-symmetry. In particular, we have that $\mathbb{F} : H^l_{\mathcal{G}} \to L^2_{\mathcal{G}}$ is well-defined and  we transform \eqref{eq : f(u)=0 on Hl_G} as
\begin{align}\label{eq : zero finding on H l G}
    \mathbb{F}(u) = 0, ~~\text{ where } u \in H^l_{\mathcal{G}}.
\end{align}
\begin{remark}
Assumption \ref{ass : L_G invariant} is presented partially for necessity and partially for simplicity. In the case that $\mathbb{L}$ and $\mathbb{G}$ map to different symmetries, the approach would not work except for trivial solutions. Hence, there is nothing to consider. It is possible that $\mathbb{L}$ and $\mathbb{G}$ could map to the same different symmetry. For example, suppose $\mathbb{L} u, \mathbb{G}(u) \in L^2_{\mathcal{K}}$ for some other symmetry group $\mathcal{K} \neq \mathcal{G}$. One could possibly generalize to this case, but it would require additional analysis. We choose not present this for simplicity.
\end{remark}
The rest of the paper is organized as follows. First, we recall in Section \ref{sec : symmetric fourier coefficients} some result from \cite{JB_symmetries_1} and \cite{JB_symmetries_2} about symmetric Fourier coefficients. In particular, we describe how they can be applied in the case of localized functions. In Section \ref{sec : strategies for computer assisted proofs}, we present adapted Newton-Kantorovich approaches depending on the symmetry $\mathcal{G}$ to establish. For each case, we present how to construct the required symmetric approximate objects. In Section \ref{sec : computing the bounds}, we provide the technical details for the application of Section \ref{sec : strategies for computer assisted proofs} in the case of the 2D Swift-Hohenberg PDE and the proof of dihedral localized solutions. Such details are illustrated in Section \ref{sec : proof of solutions}, in which we present constructive proofs of existence of multiple dihedral localized patterns. 
\section{Symmetric Fourier coefficients}\label{sec : symmetric fourier coefficients}
As mentioned in the introduction, our approach heavily relies on symmetric Fourier series. From that perspective, we introduce related notations and results from \cite{JB_symmetries_1} and \cite{JB_symmetries_2}, which we present in the context of \eqref{eq : f(u)=0 on Hl_G}.

 To begin, we define $\Omega_0 \bydef (-d,d)^m$  where $0<d <\infty$.  Then, we define  $$\tilde{n} = (\tilde{n}_1,\tilde{n}_2,\tilde{n}_3,\dots,\tilde{n}_m)\bydef \left( \frac{n_1}{2d},\frac{n_2}{2d} ,\frac{n_3}{2d},\dots,\frac{n_m}{2d}\right) \in \mathbb{R}^m$$ for all $(n_1,\dots,n_m) \in \mathbb{Z}^m$. Similarly as in the continuous case, we want to restrict to Fourier series representing $\mathcal{G}$-symmetric functions. If $\mathcal{G}$ is a space group, we can use the following lemma.
\begin{lemma}\label{lem : G-fourier series}
Let $u$ be a Fourier series of the form 
\begin{equation}\label{usual_fourier_series}
    u(x) = \sum_{n \in \mathbb{Z}^m} u_n e^{2\pi i  \tilde{n} \cdot x}.
\end{equation}
Let $\mathcal{G}$ be a space group on the square lattice. Therefore, $\mathcal{G}$ has a unitary representation of the form
\begin{align}
    gx = \mathcal{A}x + b
\end{align}
where $\mathcal{A} \in M_{m \times m}(\mathbb{R})$ and $b \in [0,1]^m$. Next, define
\begin{align}
    \beta_g(n) \bydef \mathcal{A} n,~~\alpha_g(n) \bydef \exp(2\pi i \beta_g(n) \cdot b).
\end{align}
Then, it follows that $u(x) = u(gx)$ for all $g \in \mathcal{G}$ and $x \in \mathbb{R}^m$ if and only if $ \alpha_g(n) u_{\beta_g(n)} = u_n$. In this case, we say that $u$ has a $\mathcal{G}$-Fourier series representation.
\end{lemma}
\begin{proof}
The proof can be found in \cite{JB_symmetries_2}. 
\end{proof}
Lemma \ref{lem : G-fourier series} provides a correspondence between symmetry of functions and symmetry of the sequence coefficients. This lemma also has an immediate corollary.
\begin{corollary}\label{corr : reduced_set}
Let $u$ be a $\mathcal{G}$-Fourier series where $\mathcal{G}$ is a space group on the square lattice. 
Let $\mathrm{orb}_{\mathcal{G}}(n)$ be defined as in \eqref{def : orbit} 
Let $J_{\mathrm{red}}(\mathcal{G})$ be defined as in \eqref{def : reduced set}. Then, the $\mathcal{G}$-Fourier series of $u$ can be written with indices in $J_{\mathrm{red}}(\mathcal{G})$. More specifically,
\begin{align}
    u(x) = \sum_{n \in J_{\mathrm{red}}(\mathcal{G})} u_n \sum_{k \in \mathrm{orb}_{\mathcal{G}}(n)} e^{2\pi i \tilde{k} \cdot x}.\label{def : H_fourier series}
\end{align}
\end{corollary}
\begin{proof}
The proof can be found in \cite{JB_symmetries_2}.
\end{proof}
Essentially, $J_{\mathrm{red}}(\mathcal{G})$ contains all coefficients necessary to compute a $\mathcal{G}$-Fourier series, which, in the case of invariance under a symmetry group, is often a strict subset of $\mathbb{Z}^m$. Therefore, we  restrict the indexing of $\mathcal{G}$-symmetric functions to $J_{\mathrm{red}}(\mathcal{G})$
and construct the full series by symmetry if needed. Let $\ell^p_{\mathcal{G}}$ denote the following Banach space
{\small\begin{align}    \ell^p_{\mathcal{G}} \bydef \left\{U = (u_n)_{n \in J_{\mathrm{red}}(\mathcal{G})}: ~ \|U\|_p \bydef \left( \sum_{n \in J_{\mathrm{red}}(\mathcal{G})} \alpha_n|u_n|^p\right)^\frac{1}{p} < \infty \right\}, \text{ where }(\alpha_n)_{n \in J_{\mathrm{red}}(\mathcal{G})} \bydef |\mathrm{orb}_{\mathcal{G}}(n)|.
\end{align}}
Note that $\ell^p_{\mathcal{G}}$ possesses the same sequences as the usual $p$ Lebesgue space for sequences indexed on $J_{\mathrm{red}}(\mathcal{G})$. For the special case $p=2$, $\ell^2_{\mathcal{G}}$ is an Hilbert space on sequences indexed on $J_{\mathrm{red}}(\mathcal{G})$ and we denote $(\cdot,\cdot)_2$ its inner product given by
\[
(U,V)_2 \bydef \sum_{n \in J_{\mathrm{red}}(\mathcal{G})} \alpha_n u_n \overline{v_n}
\]
for all $U = (u_n)_{n \in J_{\mathrm{red}}(\mathcal{G})}, V = (v_n)_{n \in J_{\mathrm{red}}(\mathcal{G})} \in \ell^2_{\mathcal{G}}.$ Moreover, for a bounded operator $K : \ell^2_{\mathcal{G}} \to \ell^2_{\mathcal{G}}$, $K^*$ denotes the adjoint of $K$ in $\ell^2_{\mathcal{G}}$.  

Now, the differential operators $\mathbb{L}$ and $\mathbb{G}$ have Fourier coefficients equivalents when defined on $\mathcal{G}$-symmetric Fourier series.  The linear differential operator $\mathbb{L}$ has a Fourier coefficients representation $L$, which is an infinite diagonal matrix with coefficients $\left(l(\tilde{n})\right)_{n\in J_{\mathrm{red}}(\mathcal{G})}$ on the diagonal where 
$l$ is the symbol of $\mathbb{L}$ given in \eqref{ass:A(1)}. Similarly, $\mathbb{G}$ has a Fourier coefficients representation $G$, for which products of functions are turned into discrete convolutions for sequences. Given Fourier coefficients $U = (U_n)_{n \in \mathbb{Z}^m},  V= (V_n)_{n \in \mathbb{Z}^m}$ corresponding to the usual exponential Fourier series expansion, we define the discrete convolution as 
\[
(U*V)_n = \sum_{k \in \mathbb{Z}^m} U_{n-k}V_k.
\]
In the case of symmetric sequences $\ell^p_\mathcal{G}$, we still denote $U*V$ the discrete convolution representing the product of two functions for consistency.  We also recall that Young's convolution inequality is applicable and 
\begin{align}\label{eq : youngs inequality}
    \|U*V\|_2 \leq \|U\|_2\|V\|_1
\end{align}
for all $U \in \ell^2_{\mathcal{G}}$ and all $V \in \ell^1_{\mathcal{G}}$. In fact, given $U \in \ell^1(J_{\mathrm{red}}(\mathcal{G}))$, this above allows to  define
\begin{align}\label{def : discrete conv operator}
    \mathbb{U} : \ell^2(J_{\mathrm{red}}(\mathcal{G})) &\to \ell^2(J_{\mathrm{red}}(\mathcal{G})) \\
    V &\mapsto  U * V
\end{align}
 the discrete convolution operator associated to $U$.

\subsection{Transition between Fourier coefficients and symmetric functions}
Now that we introduced the required notations to characterize symmetric Fourier coefficients, we describe how to construct functions in $L^2_\mathcal{G}$. For this purpose, we introduce the map
$\gamma_{\mathcal{G}} : L^2_{\mathcal{G}} \to \ell^2_{\mathcal{G}}$, given as
\begin{equation}\label{def : gamma}
    \left(\gamma_{\mathcal{G}}(u)\right)_n \bydef  \frac{1}{|\om|}\int_{\om} u(x) e^{-2\pi i \tilde{n}\cdot x}dx
\end{equation}
for all $n \in J_{\mathrm{red}}(\mathcal{G})$. Similarly, we define $\gamma_{\mathcal{G}}^\dagger : \ell^2_{\mathcal{G}} \to L^2_{\mathcal{G}}$ as 
\begin{equation}\label{def : gamma dagger}
    \gamma_{\mathcal{G}}^\dagger(U)(x) \bydef \mathbb{1}_{\om}(x) \sum_{n \in J_{\mathrm{red}}(\mathcal{G})} u_n \sum_{k \in \mathrm{orb}_{\mathcal{G}}(n)} e^{2\pi i \tilde{k} \cdot x}
\end{equation}
for all $x = (x_1,x_2) \in \R^2$ and all $U =\left(u_n\right)_{n \in J_{\mathrm{red}}(\mathcal{G})} \in \ell^2_{\mathcal{G}}$, where $\mathbb{1}_{\om}$ is the characteristic function on $\om$. Given $u \in L^2_{\mathcal{G}}$, $\gamma_{\mathcal{G}}(u)$ represents the Fourier coefficients indexed on $J_{\mathrm{red}}(\mathcal{G})$ of the restriction of $u$ on $\om$. Conversely, given a sequence $U\in \ell^2_{\mathcal{G}}$, $\gamma_{\mathcal{G}}^{\dagger}\left(U\right)$  is the function representation of $U$ in $L^2_{\mathcal{G}}.$ In particular, notice that $\gamma_{\mathcal{G}}^\dagger\left(U\right)(x) =0$ for all $x \notin \om.$  In particular, the map $\gamma_{\mathcal{G}}^\dagger$ provides a direct path for constructing functions in $L^2_\mathcal{G}$ thanks to sequences in  $\ell^2_{\mathcal{G}}.$

Then, recalling similar notations from \cite{unbounded_domain_cadiot}
\begin{align}
    L^2_{\mathcal{G},\om} \bydef \left\{u \in L^2_{\mathcal{G}} : \text{supp}(u) \subset \overline{\om} \right\},~
   H_{\mathcal{G},\om}^l \bydef \left\{u \in H^l_{\mathcal{G}} : \text{supp}(u) \subset \overline{\om} \right\}.
\end{align}
Moreover, recall that $\mathcal{B}(L^2_{\mathcal{G}})$ (respectively $\mathcal{B}(\ell^2_{\mathcal{G}})$) denotes the space of bounded linear operators on $L^2_{\mathcal{G}}$ (respectively $\ell^2_{\mathcal{G}}$) and denote by $\mathcal{B}_{\om}(L^2_{\mathcal{G}})$ the following subspace of $\mathcal{B}(L^2_{\mathcal{G}})$
\begin{equation}\label{def : Bomega}
    \mathcal{B}_{\om}(L^2_{\mathcal{G}}) \bydef \{\mathbb{K}_{\om} \in \mathcal{B}(L^2_{\mathcal{G}}) :  \mathbb{K}_{\om} = \mathbb{1}_{\om}\mathbb{K}_{\om} \mathbb{1}_{\om}\}.
\end{equation}
Finally, define $\Gamma_{\mathcal{G}} : \mathcal{B}(L^2_{\mathcal{G}}) \to \mathcal{B}(\ell^2_{\mathcal{G}})$ and $\Gamma_{\mathcal{G}}^\dagger : \mathcal{B}(\ell^2_{\mathcal{G}}) \to \mathcal{B}(L^2_{\mathcal{G}})$ as follows
\begin{equation}\label{def : Gamma and Gamma dagger}
    \Gamma_{\mathcal{G}}(\mathbb{K}) \bydef \gamma_{\mathcal{G}} \mathbb{K} \gamma_{\mathcal{G}}^\dagger ~~ \text{ and } ~~  \Gamma_{\mathcal{G}}^\dagger(K) \bydef \gamma_{\mathcal{G}}^\dagger {K} \gamma_{\mathcal{G}} 
\end{equation}
for all $\mathbb{K} \in \mathcal{B}(L^2_{\mathcal{G}})$ and all $K \in \mathcal{B}(\ell^2_{\mathcal{G}}).$

The maps defined above in \eqref{def : gamma}, \eqref{def : gamma dagger} and \eqref{def : Gamma and Gamma dagger} are fundamental in our analysis as they allow to pass from the problem on $\R^m$ to the one in $\ell^2_{\mathcal{G}}$ and vice-versa. Furthermore, using Parseval's identity, we obtain natural isometric isomorphisms, when restricted to the relevant spaces.
\begin{lemma}\label{lem : gamma and Gamma properties}
    The map $\sqrt{|\om|} \gamma_{\mathcal{G}} : L^2_{\mathcal{G},\om} \to \ell^2_{\mathcal{G}}$ (respectively $\Gamma_{\mathcal{G}} : \mathcal{B}_\om(L^2_{\mathcal{G}}) \to \mathcal{B}(\ell^2_{\mathcal{G}})$) is an isometric isomorphism whose inverse is given by $\frac{1}{\sqrt{|\om|}} \gamma_{\mathcal{G}}^\dagger : \ell^2_{\mathcal{G}} \to L^2_{\mathcal{G},\om}$ (respectively $\Gamma_{\mathcal{G}}^\dagger :   \mathcal{B}(\ell^2_{\mathcal{G}}) \to \mathcal{B}_\om(L^2_{\mathcal{G}})$). In particular,
    \begin{align}\label{eq : parseval's identity}
        \|u\|_2 = \sqrt{\om}\|U\|_2 \text{ and } \|\mathbb{K}\|_2 = \|K\|_2
    \end{align}
    for all $u \in L^2_{\mathcal{G},\om}$ and $\mathbb{K}\in \mathcal{B}_\om(L^2_{\mathcal{G}})$, and where $U \bydef \gamma_{\mathcal{G}}(u)$ and ${K} \bydef \Gamma_{\mathcal{G}}(\mathbb{K})$.
\end{lemma}
\begin{proof}
    The proof relies on Parseval's identity and Lemma 3.2 from \cite{unbounded_domain_cadiot}.
\end{proof}

The above lemma not only provides a one-to-one correspondence between the elements in $L^2_{\mathcal{G},\om}$ (respectively $\mathcal{B}_\om(L^2_{\mathcal{G}})$) and the ones in $\ell^2_{\mathcal{G}}$ (respectively $\mathcal{B}(\ell^2_{G})$) but it also provides an identity on norms. 
\par Now all of the theory we have introduced in this section relies on $\mathcal{G}$ being a space group. In general, this is not always the case. As we wish to use the theory of this section, we introduce a definition.
\begin{definition}[\bf Maximal Square Lattice Space Subgroup]\label{def : maximal square lattice space subgroup}
Let $\mathcal{H}$ be a symmetry group. We say that $\mathcal{H}$ is a \emph{maximal square lattice space subgroup} of $\mathcal{G}$ if
\begin{itemize}
    \item $\mathcal{H}$ is a subgroup of $\mathcal{G}$
    \item $\mathcal{H}$ is a space group
    \item  $\mathcal{H}$-symmetric functions can be represented on the square lattice
    \item $|\mathcal{K}| \leq |\mathcal{H}|$ where $\mathcal{K}$ is any square lattice space subgroup of $\mathcal{G}$.
\end{itemize}
\end{definition}
For example, suppose $\mathcal{G} = D_8$. Some space subgroups of $\mathcal{G}$ include $ \mathbb{Z}_2 \times \mathbb{Z}_1, D_2,D_4$, and more. Recall that $D_j$ is defined as in \eqref{def : Dj presentation} and we define $\mathbb{Z}_2 \times \mathbb{Z}_1$ as the direct product (see \cite{dummit_algebra,gallianalgebra} for a definition) of the groups $\mathbb{Z}_2$ (see \eqref{def : Z2group} for a definition) and the trivial group $\mathbb{Z}_1$. This group comprises of even symmetry in the first component and no symmetry in the second. It is clear that the square lattice space subgroup of maximal order is $\mathcal{H} = D_4$. As another example, consider $\mathcal{G} = D_{10}$. The largest subgroup of $D_{10}$ is $D_5$; however, $D_5$ is not a space group as pentagons do not tile the plane. Therefore, we must choose a smaller subgroup, in this case, $\mathcal{H} = D_2$ is the largest in order. Finally, consider $\mathcal{H} = D_{12}$. Note that $D_6$ is a space subgroup of $D_{12}$; however, we cannot represent $D_6$-symmetric functions on the square lattice. Hence, we must choose a smaller subgroup, in this case $\mathcal{H} = D_4$.   
\par For the remainder of this paper, we will refer to $\mathcal{H}$ as the maximal square lattice space subgroup of $\mathcal{G}$. This means that the results of this section will always apply to $\mathcal{H}$ since it is a space group, and hence why we will use $\mathcal{H}$ frequently in the rest of this paper. Why we chose $\mathcal{H}$ to be the square lattice space subgroup of maximal order will become clear in Section \ref{sec : H isolate}. 
\section{Three cases and their strategies for performing computer assisted proofs}\label{sec : strategies for computer assisted proofs}
In this section, we will discuss three separate cases based on the properties of $\mathcal{G}$ and its maximal square lattice space subgroup $\mathcal{H}$. For each of these cases, we will first introduce the main theorem we wish to apply. Following this, our goal will be to construct the necessary objects to apply this theorem. In fact, a part of the construction will be numerical and require rigorous numerics. For that matter, given $\mathcal{N} \in \mathbb{N}$, we introduce projection operators  $\pi^{\mathcal{N}}$ and $\pi_{\mathcal{N}}$ defined as
 \begin{align}
 \nonumber
    (\pi^{\mathcal{N}}(V))_n  =  \begin{cases}
          v_n,  & n \in I^{\mathcal{N}} \\
              0, &n \notin I^{\mathcal{N}}
    \end{cases} ~~ \text{ and } ~~
     (\pi_{\mathcal{N}}(V))_n  =  \begin{cases}
          0,  & n \in I^{\mathcal{N}} \\
              v_n, &n \notin I^{\mathcal{N}}
    \end{cases}
 \end{align}
    where $I^{\mathcal{N}} \bydef \{n \in J_{\mathrm{red}}(\mathcal{G}), ~ n_1,n_2 \leq \mathcal{N}\}$
    % \begin{align} 
    % I^{\mathcal{N}} \bydef \{n \in J_{\mathrm{red}}(\mathcal{G}), ~ n_1,n_2 \leq \mathcal{N}\}\label{def : IN}
    % \end{align}
    for all $V = (v_n)_{n \in  J_{\mathrm{red}}(\mathcal{G})} \in \ell^2(J_{\mathrm{red}}(\mathcal{G})).$ In particular, given $V \in \ell^2(J_{\mathrm{red}}(\mathcal{G}))$, $\pi^\mathcal{N} V$ can be represented by a vector as it contains $\mathcal{N}$ nonzero coefficients. Similarly, given an operator $B \in \mathcal{B}(\ell^2(J_{\mathrm{red}}(\mathcal{G})))$, $\pi^\mathcal{N}B\pi^\mathcal{N}$ can be represented by a matrix.
\subsection{\texorpdfstring{$\mathcal{H} = \mathcal{G}$}{HisG}}\label{sec : HisG}
In the case that $\mathcal{H} = \mathcal{G}$, then $\mathcal{G}$ is a space group. Now, as mentioned in the introduction, our approach relies on the construction of two things. The first is an approximate solution to \eqref{eq : zero finding on H l G}, which we will denote $u_0 \in H^l_{\mathcal{G}}$ with compact support on $\om$. The second is an approximate inverse to $D\mathbb{F}(u_0)$, which we will denote $\mathbb{A} : L^2_{\mathcal{G}} \to H^l_{\mathcal{G}}$. Once we have these subjects, we can derive a Newton-Kantorovich type Theorem in order to prove our results. We will begin with the construction of the approximate solution, $u_0$. Following this, we will build the operator $\mathbb{A}$, the approximate inverse. Finally, we will state the Theorem we wish to apply.
\subsubsection{Construction of \texorpdfstring{$u_0$}{u0}}\label{sec : construction of u0 gen}
Since $\mathcal{G}$ is a space group, we can build $u_0$ using Fourier series. Indeed, given a numerical size of truncation $N_0 \in \mathbb{N}$ for our sequences, we suppose that we have access to a numerical candidate $U_0 \in \ell^2(J_{\mathrm{red}}(\mathcal{G}))$ such that $U_0 = \pi^{N_0} U_0$ (that is $U_0$ is a trigonometric polynomial of order $N_0$). Then, we define $u_0 \bydef \gamma^\dagger_\mathcal{G}(U_0) \in L^2_\mathcal{G}$ its function representation on $\R^m$. However, we do not readily have that $u_0$ is smooth on $\R^m$. That is, $u_0 \in H^l_\mathcal{G}$ is not guaranteed. In order to tackle this issue, we ensure that the trace of $u_0$ on $ \Omega_0$ is null using a projection of its Fourier coefficients $U_0$. The resulting function becomes smooth by construction. Such an analysis has been developed  in Section 4 of \cite{unbounded_domain_cadiot} and we illustrate it in Section \ref{sec : computing the bounds} below. In particular, we obtain that $u_0$ is defined via its Fourier coefficients $U_0$ as  
\begin{align}\label{def : construction of u0}
 u_0 \bydef \gamma_{\mathcal{G}}^\dagger(U_0) \in H^l_\mathcal{G}, ~~~~ \text{ where } U_0 = \pi^{N_0} U_0.
 \end{align}
With $u_0$ now constructed, let us discuss the operator $\mathbb{A}$.
\subsubsection{The Operator \texorpdfstring{$\mathbb{A}$}{A}}\label{sec : operator A first}
In this section, we focus our attention on the construction of $\mathbb{A} : L^2_{\mathcal{G}} \to H^l_{\mathcal{G}}$. Specifically, we recall the construction exposed in Section 3 from \cite{unbounded_domain_cadiot}.
Let $N \in \mathbb{N}$ be the numerical truncation size for our matrices. 
We begin by numerically computing an approximate inverse for $\pi^N DF(U_0)L^{-1}\pi^N$ using floating point arithmetic. We denote this approximation by $B^N$. $B^N$ is built as a matrix that we naturally extend to a bounded linear operator on $\ell^2_{\mathcal{G}}$ such that $B^N = \pi^N B^N \pi^N$. With $B^N$ defined, we are now able to define the bounded linear operator $\mathbb{B} : L^2_{\mathcal{G}} \to L^2_{\mathcal{G}}$ as
\begin{equation}\label{def : the operator B}
    \mathbb{B} \bydef \mathbb{1}_{\mathbb{R}^m \setminus \om} + \Gamma_{\mathcal{G}}^\dagger(\pi_N + B^N).
\end{equation}
Then, by using $\mathbb{B}$, we can  define the operator $\mathbb{A} : L^2_{\mathcal{G}} \to H^l_{\mathcal{G}}$ as 
\begin{equation}\label{def : the operator A}
    \mathbb{A} \bydef \mathbb{L}^{-1} \mathbb{B}.
\end{equation}
We refer the interested reader to the Section 3 of \cite{unbounded_domain_cadiot} for the justification of such a construction. 
The well definedness of $\mathbb{A}$ is obtained by using the fact that  $\mathbb{L}$ is an isometric isomorphism between $H^l_{\mathcal{G}}$ and $L^2_{\mathcal{G}}$ (cf. \eqref{def : L2 and Hl G}). Moreover, $\mathbb{A}$ is completely determined by $B^N$, which is chosen numerically. This implies that we can use interval arithmetic to perform rigorous computations involving $\mathbb{A}$. In particular, using Lemma \ref{lem : gamma and Gamma properties}, we have 
\begin{equation}\label{eq : equality norm A and BN}
\|\mathbb{A}\|_{2,l} = \|\mathbb{B}\|_2 = \max\left\{1,\|B^N\|_2\right\}.
\end{equation}
In practice, if  $\mathcal{Z}_1$ defined in Theorem \ref{th: radii polynomial} satisfies $\mathcal{Z}_1 <1$, then $\|I_d - \mathbb{A}D\mathbb{F}(u_0)\|_{l}<1$.
From there, Theorem 3.5 in \cite{unbounded_domain_cadiot} provides that both $\mathbb{A} : L^2_{\mathcal{G}} \to H^l_{\mathcal{G}}$ and $D\mathbb{F}(u_0) : H^l_{\mathcal{G}} \to L^2_{\mathcal{G}}$ have a bounded inverse, which, in such a case, justifies that $\mathbb{A}$ can be considered as an approximate inverse of $D\mathbb{F}(u_0)$. With $u_0$ and $\mathbb{A}$ now available, let us state the main theorem we wish to apply.
\subsubsection{Newton-Kantorovich Approach in the case \texorpdfstring{$\mathcal{H} = \mathcal{G}$}{HisG}}
In this section, we introduce the main theorem we wish to apply for the case $\mathcal{H} = \mathcal{G}$. We introduce the following fixed point operator $\mathbb{T} : \overline{B_r(u_0)} \to \overline{B_r(u_0)}$ given by
\begin{align}
\mathbb{T}(u) \bydef u - \mathbb{A}\mathbb{F}(u).\label{def : T approx}
\end{align}
We want to prove there exists an $r > 0$ such that $\mathbb{T}$ is well-defined and a contraction. In order to determine a possible value for $r>0$ that would provide the contraction, we wish to use a Radii-Polynomial theorem. In particular, we build $u_0 \in H^l_{\mathcal{G}}$ as in Section \ref{sec : construction of u0 gen}, $\mathbb{A} : L_{\mathcal{G}}^2 \to H^l_{\mathcal{G}}$ as in Section \ref{def : the operator A}, and the bounds $\mathcal{Y}_0, \mathcal{Z}_1 >0$ and  $\mathcal{Z}_2 : (0, \infty) \to [0,\infty)$ in such a way that the hypotheses of the following theorem are satisfied.
\begin{theorem}\label{th: radii polynomial}
Let $\mathbb{A} : L_{\mathcal{G}}^2 \to H_{\mathcal{G}}^l$ be a bounded linear operator. Moreover, let $\mathcal{Y}_0, \mathcal{Z}_1$ be non-negative constants and let  $\mathcal{Z}_2 : (0, \infty) \to [0,\infty)$ be a non-negative function  such that
  \begin{align}\label{eq: definition Y0 Z1 Z2}
    \|\mathbb{A}\mathbb{F}(u_0)\|_l \leq &\mathcal{Y}_0\\
    \|I_d - \mathbb{A}D\mathbb{F}(u_0)\|_{l} \leq &\mathcal{Z}_1\\
    \|\mathbb{A}\left({D}\mathbb{F}(h) - D\mathbb{F}(u_0)\right)\|_l \leq &\mathcal{Z}_2(r)\|h - u_0\|_{l}, ~~ \text{for all } h \in B_r(u_0)
\end{align}  
If there exists $r>0$ such that
\begin{equation}\label{condition radii polynomial}
    \frac{1}{2}\mathcal{Z}_2(r)r^2 - (1-\mathcal{Z}_1)r + \mathcal{Y}_0 <0, \ and \ \mathcal{Z}_1 + \mathcal{Z}_2(r)r < 1 
 \end{equation}
then there exists a unique $\tilde{u} \in \overline{B_r(u_0)} \subset H^l_{\mathcal{G}}$ such that $\mathbb{F}(\tilde{u})=0$, where $B_r(u_0)$ is the open ball of $H^l_{\mathcal{G}}$ centered at $u_0$. 
\end{theorem}
\begin{proof}
The proof can be found in \cite{van2021spontaneous}.
\end{proof}
By using Theorem \ref{th: radii polynomial}, we are able to perform computer assisted proofs in the case $\mathcal{H} = \mathcal{G}$. Let us now move to the next case.
\subsection{When \texorpdfstring{$\mathcal{H}< \mathcal{G}$}{HleqG} but \texorpdfstring{$\mathcal{H}$}{H} isolates the solution}\label{sec : H isolate}
In this section, we consider the case where $\mathcal{H} < \mathcal{G}$, but we assume that the symmetry of $\mathcal{H}$ allows to isolate the solution. 
Indeed, the set of solutions of equations on $\R^m$ might possess natural translation and rotation invariances (as it is the case in the planar Swift-Hohenberg PDE presented in Section \ref{sec : computing the bounds}). In that sense, solutions are not isolated; however, when constraining the solutions to symmetries, the solutions might become isolated since the translations and rotations are moduled out.  In this section, we assume that $\mathcal{H}$ possesses enough symmetries to isolate a given solution, from the rest of the solution set. 
\par For the approximate solution, we cannot use the construction described in Section \ref{sec : construction of u0 gen} readily since we would get $u_0 \in H^l_\mathcal{H}$ but not necessarily $u_0 \in H^l_\mathcal{G}$. We present in the next section our strategy to tackle this difficulty along with the choice for the operator $\mathbb{A}$.
\subsubsection{Construction of \texorpdfstring{$w_0$}{w0} and \texorpdfstring{$\mathbb{A}$}{A}}\label{sec : construction of w_0}
\par 
Assume that we have access to an approximate solution $u_0 \in H^l_{\mathcal{H}}$ of the form \eqref{def : construction of u0}. We want to construct an approximate solution $w_0 \in H^l_\mathcal{G}$ such that $w_0$ remains close to $u_0$ (in $H^l$-norm say) and so that the construction of $w_0$ is compatible with the use of rigorous numerics. That is, the construction has to involve quantifiable objects. Under such objectives, we make use of the bijectivity of the group action and define $w_0$ as
\begin{equation}\label{v_D_m_fourier}
    w_0 \bydef \frac{1}{|\mathcal{G}|} \sum_{g \in \mathcal{G}} g \cdot u_0,
\end{equation}
then $w_0$ is necessarily $\mathcal{G}$-symmetric. This simple construction is actually quite useful in our computer-assisted analysis since it provides explicit formulas for our required computations. We illustrate its usage in Section \ref{sec : computing the bounds}.

Concerning the construction of $\mathbb{A}$, recall that $\mathbb{A}$ is chosen as an approximate inverse for $D\mathbb{F}(u_0) : H^l_{\mathcal{H}} \to L^2_{\mathcal{H}}$, and it is obtained using the strategy described in Section \ref{sec : operator A first}.  In particular, if $u_0$ is a good approximate solution, and is close to be $\mathcal{G}$-invariant (in the sense that $\|u_0-w_0\|_l$ is small), then $\mathbb{A}$ will be a good approximate inverse for $D\mathbb{F}(w_0) : H^l_{\mathcal{G}} \to L^2_{\mathcal{G}}$ as well. Therefore, we still build $\mathbb{A}$ using the strategy of Section \ref{sec : operator A first}, and we quantify such an approximation in Section \ref{sec : HleqG H isolates bounds}.
\subsubsection{Newton-Kantorovich Approach in the case \texorpdfstring{$\mathcal{H} < \mathcal{G}$}{HleqG} and \texorpdfstring{$\mathcal{H}$}{H} isolates the solution}
We now wish to derive a a Radii-Polynomial Theorem for the case $\mathcal{H} < \mathcal{G}$, but $\mathcal{H}$ isolates the solution. Before stating this theorem, we define \begin{equation}\label{def : G_tilde}
    \tilde{\mathbb{G}}(u_0) \bydef \mathbb{G}(w_0) - \frac{1}{|\mathcal{G}|} \sum_{g \in \mathcal{G}} g \mathbb{G}(u_0).
    \end{equation}
In particular, $\mathbb{G}(w_0)$ can be written as a sum of the elements of $g$ applied to $\mathbb{G}(u_0)$ plus an extra term we denote by $\tilde{\mathbb{G}}(u_0)$. This term represents the fact that $u_0$ is only $\mathcal{H}$ (and not $\mathcal{G}$) symmetric.
\par We want to prove, as in \cite{unbounded_domain_cadiot}, that there exists $r>0$ such that $\mathbb{T} : \overline{B_r(w_0)} \to \overline{B_r(w_0)}$ defined as
\begin{align}
\mathbb{T}(w) := w - D\mathbb{F}(w_0)^{-1}\mathbb{F}(w)\label{T_exact}
\end{align}
is well defined and is a contraction. Note that we use the exact inverse, $D\mathbb{F}(w_0)^{-1}$ (assuming it exists) in the definition of \eqref{T_exact}. This is done so that $\mathbb{T}(w)$ remains $\mathcal{G}$-symmetric. Indeed, since the construction of $\mathbb{A}$ performed in Section \ref{sec : operator A first} does not satisfy the $\mathcal{G}$-symmetry, we cannot define $\mathbb{T}$ as in \eqref{def : T approx} as it will not map $\overline{B_r(w_0)}$ into itself. By using the exact inverse, we will prove that $\mathbb{T}$ does map $\overline{B_r(w_0)}$ into itself along with the fact that it is a contraction in the following theorem.
\begin{theorem}\label{th: radii polynomial s}
Let $\mathcal{H} \leq \mathcal{G}$ be the square lattice space subgroup of $\mathcal{G}$ of maximal order. Let $\mathbb{A} : L^2_{\mathcal{H}} \to H^l_{\mathcal{H}}$ and let $\mathcal{Y}_0,\mathcal{Y}_s, \mathcal{Z}_1,\mathcal{Z}_s$ be non-negative constants and let $\mathcal{Z}_2 : [0,\infty) \to (0,\infty)$ be a non-negative function such that for all $r>0$
  \begin{align}\label{eq: definition Y0 Z1 Z2 s}
    \|\mathbb{A}\mathbb{F}(u_0)\|_l \leq &\mathcal{Y}_0\\
    \|\mathbb{A}\tilde{\mathbb{G}}(u_0)\|_{l} \leq & \mathcal{Y}_s \\
    \|I_d - \mathbb{A}D\mathbb{F}(u_0)\|_{l} \leq &\mathcal{Z}_1\\
    \|\mathbb{A}(D\mathbb{G}(w_0) - D\mathbb{G}(u_0))\|_{l} \leq & \mathcal{Z}_s \\
    \|\mathbb{A}\left({D}\mathbb{F}(w_0) - D\mathbb{F}(h)\right)\|_l \leq &\mathcal{Z}_2(r)\|h-w_0\|_l, ~~ \text{for all } h \in B_r(w_0)
\end{align}  
If there exists $r>0$ such that
\begin{equation}\label{condition radii polynomial s}
   \frac{1}{2}\mathcal{Z}_2(r)r^2 - (1-\mathcal{Z}_1 - \mathcal{Z}_s)r + \mathcal{Y}_0 + \mathcal{Y}_s <0, ~\text{and}~ \mathcal{Z}_2(r)r < 1-\mathcal{Z}_1 - \mathcal{Z}_s,
 \end{equation}
then there exists a unique $\tilde{w} \in \overline{B_r(w_0)} \subset H^l_{\mathcal{G}}$ such that $\mathbb{F}(\tilde{w})=0$, where $\overline{B_r(w_0)}$ is the closed ball of $H^l_{\mathcal{G}}$ centered at $w_0$ and of radius $r$.
\end{theorem}

\begin{proof}
Our proof is based on that of Theorem 2.15 in \cite{van2021spontaneous}. Firstly, observe that
\begin{align}
    \|I_d - \mathbb{A}D\mathbb{F}(w_0)\|_{l} \leq \|I_d - \mathbb{A}D\mathbb{F}(u_0)\|_{l} + \|\mathbb{A}(D\mathbb{G}(w_0) - D\mathbb{G}(u_0))\|_{l} \leq \mathcal{Z}_1 + \mathcal{Z}_s.
\end{align}
Since we assume that $\mathcal{Z}_1 + \mathcal{Z}_s < 1$ and that $\mathbb{A}$ is invertible, a Neumann series argument implies that and $D\mathbb{F}(w_0) : H^l_{\mathcal{H}} \to L^2_{\mathcal{H}}$ has a bounded inverse. We denote by $D\mathbb{F}(w_0)^{-1} : L^2_{\mathcal{H}} \to H^l_{\mathcal{H}}$ the inverse of $D\mathbb{F}(w_0)$. Using Assumption \ref{ass : L_G invariant}, recall that we can also restrict $\mathbb{F}$ as  $\mathbb{F}: H^l_{\mathcal{G}} \to L^2_{\mathcal{G}}$. Then, it follows by definition that $D\mathbb{F}(w_0)^{-1} : L^2_{\mathcal{G}} \to H^l_{\mathcal{G}}$ is well-defined.
Next, observe that
{\small\begin{align}
    \|D\mathbb{F}(w_0)^{-1}\|_{\mathcal{B}(L^2_{\mathcal{G}},H^l_{\mathcal{G}})} = \sup_{y \in L^2_{\mathcal{G}}} \frac{\|D\mathbb{F}(w_0)^{-1} y\|_{H^l_{\mathcal{G}}}}{\|y\|_{L^2_{\mathcal{G}}}} \leq \sup_{y \in L^2_{\gamma_{\mathcal{H}}}} \frac{\|D\mathbb{F}(w_0)^{-1} y\|_{H^l_{\mathcal{H}}}}{\|y\|_{L^2_{\mathcal{H}}}} = \|D\mathbb{F}(w_0)^{-1}\|_{\mathcal{B}(L^2_{\mathcal{H}},H^l_\mathcal{H})}
\end{align}}
where the second to last step follows from the fact that $L^2_{\mathcal{G}} \subset L^2_{\mathcal{H}}$ since $\mathcal{H} < \mathcal{G}$.
\par We now estimate for $r > 0$ and $w \in \overline{B_r(w_0)}$,
{\small\begin{align}
    \|\mathbb{T}(w) - w_0\|_{l} \leq \|\mathbb{T}(w) - \mathbb{T}(w_0)\|_{l} + \|\mathbb{T}(w_0) - w_0\|_{l} \leq \frac{1}{1-\mathcal{Z}_1 - \mathcal{Z}_s} \frac{1}{2} \mathcal{Z}_2(r) r^2 + \|D\mathbb{F}(w_0)^{-1} \mathbb{F}(w_0)\|_{l}\label{first_breakage}
\end{align}}
where the last step followed from the proof of Theorem 2.15 in \cite{van2021spontaneous}. Next, using Assumption \ref{ass : L_G invariant}, we have that $g\mathbb{L}w = \mathbb{L} gw$ for all $w \in H^l_{\mathcal{G}}$. In particular, combining \eqref{def : G_tilde} and the fact that $w_0 = \frac{1}{|\mathcal{G}|}\sum_{g \in \mathcal{G}} u_0$, observe that
\begin{align}
    \|D\mathbb{F}(w_0)^{-1}\mathbb{F}(w_0)\|_{l} &= \left\|D\mathbb{F}(w_0)^{-1} \left(\frac{1}{|\mathcal{G}|}\sum_{g \in \mathcal{G}} g\mathbb{F}(u_0) + \tilde{\mathbb{G}}(u_0)\right)\right\|_{l} \\
    &\leq \frac{1}{|\mathcal{G}|}\left\|D\mathbb{F}(w_0)^{-1}\sum_{g \in \mathcal{G}} g\mathbb{F}(u_0)\right\|_{l} + \|D\mathbb{F}(w_0)^{-1} \tilde{\mathbb{G}}(u_0)\|_{l}\label{Y_0_step1}
\end{align}
Now, since $D\mathbb{F}(w_0)^{-1}$ is invariant under the action of $\mathcal{G}$, we obtain
\begin{align}
    \frac{1}{|\mathcal{G}|}\left\|D\mathbb{F}(w_0)^{-1}\sum_{g \in \mathcal{G}} g\mathbb{F}(u_0)\right\|_{l} &= \frac{1}{|\mathcal{G}|}\left\|\sum_{g \in \mathcal{G}} gD\mathbb{F}(w_0)^{-1}\mathbb{F}(u_0)\right\|_{l},
\end{align}
and since the $l$-norm is invariant under the action of $\mathcal{G}$ thanks to Assumption \ref{ass : L_G invariant}, we obtain
\begin{align}
    \frac{1}{|\mathcal{G}|}\left\|\sum_{g \in \mathcal{G}} gD\mathbb{F}(w_0)^{-1}\mathbb{F}(u_0)\right\|_{l}  \leq \frac{1}{|\mathcal{G}|}\sum_{g \in \mathcal{G}} \left\| gD\mathbb{F}(w_0)^{-1}\mathbb{F}(u_0)\right\|_{l} =   \|D\mathbb{F}(w_0)^{-1} \mathbb{F}(u_0)\|_{l}.\label{Y_0_step2}
\end{align}
Returning to \eqref{Y_0_step1}, we use \eqref{Y_0_step2}  to obtain
{\small\begin{align}
    \|D\mathbb{F}(w_0)^{-1} \mathbb{F}(w_0)\|_{l} \leq \|D\mathbb{F}(w_0)^{-1} \mathbb{F}(u_0)\|_{l} + \|D\mathbb{F}(w_0)^{-1} \tilde{\mathbb{G}}(u_0)\|_{l} 
    &\leq \frac{1}{1-\mathcal{Z}_1 - \mathcal{Z}_s}\left(\|\mathbb{A}\mathbb{F}(u_0)\|_{l} + \|\mathbb{A}\tilde{\mathbb{G}}(u_0)\|_{l}\right) \\
    &\leq \frac{1}{1-\mathcal{Z}_1 - \mathcal{Z}_s}(\mathcal{Y}_0 + \mathcal{Y}_s).
\end{align}}
Returning to \eqref{first_breakage}, we obtain
\begin{align}
    \|\mathbb{T}(w) - w_0\|_{l} \leq \frac{1}{1-\mathcal{Z}_1 - \mathcal{Z}_s} \left(\frac{1}{2} \mathcal{Z}_2(r) r^2 + \mathcal{Y}_0 + \mathcal{Y}_s\right).
\end{align}
Hence, $\mathbb{T}$ maps $\overline{B_r(w_0)}$ to itself. Next, we show that $\mathbb{T}$ is a contraction. In particular, we have
\begin{align}
    \|D\mathbb{T}(w)\|_{\mathcal{B}(H^l_{\mathcal{G}},H^l_{\mathcal{G}})} = \|I_d - D\mathbb{F}(w_0)^{-1}\mathbb{F}(w)\|_{\mathcal{B}(H^l_{\mathcal{G}},H^l_{\mathcal{G}})} 
    &= \|D\mathbb{F}(w_0)^{-1} (D\mathbb{F}(w_0) - D\mathbb{F}(w))\|_{{\mathcal{B}(H^l_{\mathcal{G}},H^l_{\mathcal{G}})}} \\
    &\hspace{-0.5cm}\leq \frac{1}{1-\mathcal{Z}_1 - \mathcal{Z}_s} \|\mathbb{A}(D\mathbb{F}(w_0) - D\mathbb{F}(w))\|_{\mathcal{B}(H^l_{\mathcal{H}},H^l_{\mathcal{H}})} \\
    &\hspace{-0.5cm}\leq \frac{\mathcal{Z}_2(r) r}{1-\mathcal{Z}_1 - \mathcal{Z}_s}.
\end{align}
Since \eqref{condition radii polynomial s} is satisfied, we obtain
\begin{align}
    \|D\mathbb{T}(w)\|_{\mathcal{B}(H^l_{\mathcal{G}},H^l_{\mathcal{G}})} \leq \frac{\mathcal{Z}_2(r) r}{1-\mathcal{Z}_1 - \mathcal{Z}_s} < 1
\end{align}
as desired. Using the Banach fixed point theorem, we obtain the existence of a unique fixed pointed $\tilde{w}$ of $\mathbb{T}$ in $\overline{B_r(w_0)} \subset H^l_\mathcal{G}$.
\end{proof} 
With Theorem \ref{th: radii polynomial s}, we are able to perform computer assisted proofs in the case $\mathcal{H} < \mathcal{G}$, but $\mathcal{H}$ isolates the solution. Let us now move to the final case.
\subsection{When \texorpdfstring{$\mathcal{H}$}{H} does not isolate the solution}\label{sec : oddprimes}
In this section, we assume that the symmetries of $\mathcal{H}$ are not sufficient to isolate the solution (in comparison with the previous section). In such a case, we expect $D\mathbb{F}(u_0) : H^l_\mathcal{H} \to L^2_\mathcal{H}$ to have a kernel (due to translation or rotation invariance) and the framework proposed with the use of Theorem \ref{th: radii polynomial} cannot apply. Indeed, a contraction argument relies of the invertibility of the Jacobian at the solution. In order to eliminate this natural kernel, we modify our zero finding map $\mathbb{F}$.

One possible choice is to append extra equations in $\mathbb{F}$ to isolate the solution. This results in an unbalanced system, which can be overcome by the addition of an \emph{unfolding parameters}. In general, choosing extra equations with adequate unfolding parameters is a non-trivial task. Some applications in computer assisted proofs using unfolding parameters can be found in \cite{jay_unfold1}, \cite{jay_unfold4}, \cite{jay_unfold3}, and \cite{jay_unfold2}.  In particular, this choice is extremely problem-dependent and has to be handled in the case to case scenario. For that matter, we restrict ourselves to the case $m=2$ (that is we consider solutions on the plane $\R^2$) where $\mathcal{G} = D_j$ is a dihedral group and where $j$ is an odd prime. Moreover, we assume that the solution set is invariant under $SE(2)$, the group of Euclidean symmetries on the plane. That is solutions are invariant under all possible combinations of  translations and rotations.
In this case, the maximal square lattice space subgroup is $\mathcal{H} = \mathbb{Z}_2 \times \mathbb{Z}_1$. In fact,  the lack of isolation stems from the translation invariance of the set of solutions. In other terms, $\mathcal{H}$ does not allow to mode out the translation invariance in the $x_2$-direction. 
\subsubsection{The map \texorpdfstring{$\mathbb{F}$}{Fbar}}
\par Suppose that $\tu \in H^l_{\mathbb{Z}_2 \times \mathbb{Z}_1}$ solves \eqref{eq : f(u)=0 on Hl_G}. Then, using Proposition 2.5 of \cite{unbounded_domain_cadiot}, we have that $\tilde{u}$ is infinitely differentiable. Moreover,  note that $\partial_{x_2} \tu \in H^l_{\mathbb{Z}_2 \times \mathbb{Z}_1}$ and 
\begin{align*}
   0 =  \partial_{x_2}\mathbb{F}(\tu) = D\mathbb{F}(\tu) \partial_{x_2} \tu,
\end{align*}
so $\partial_{x_2}\tu \in Ker(D\mathbb{F}(\tu))$. In particular, $\partial_{x_2} \tu \in H^l_{\mathbb{Z}_2 \times \mathbb{Z}_1}$ arises from the translation invariance in $x_2$ which remains in $H^l_{\mathbb{Z}_2 \times \mathbb{Z}_1}$.  Our goal is to use this information in order to construct a Lagrange multiplier, which will leave the set of solutions unchanged.

We denote $u_0 = \gamma^\dagger_{\mathbb{Z}_2 \times \mathbb{Z}_1}(U_0) \in H^l_{\mathbb{Z}_2 \times \mathbb{Z}_1}$ our approximate solution constructed as in Section \ref{sec : construction of u0 gen}. Recalling that $\mathcal{G} = D_j$ in this section, we denote $w_0 = \frac{1}{j} \sum_{g \in D_j} g u_0$ our ${D}_j$-invariant approximate solution. Now, let us fix $\rho>0$ to be a scaling factor. $\rho$ is chosen numerically in order to optimize the computation of the bounds that will be defined later (see Theorem \ref{th: radii polynomial augmented}). Then, we add an unfolding parameter, which we denote by $\beta$, to \eqref{eq : f(u)=0 on Hl_G} resulting in
\begin{align}
    \mathbb{f}(\beta,w) \bydef \mathbb{L}w + \mathbb{G}(w) + \beta \frac{\partial_{x_2} w_0}{\rho}\label{f_unfold}
\end{align}
 Similarly as a Lagrange multiplier, the term $\beta \frac{\partial_{x_2} w_0}{\rho}$ will allow to desingularize the Jacobian of $\mathbb{f}$ at a solution.
 In order to balance \eqref{f_unfold}, we need to add an extra equation for fixing a value of $\beta$. In particular, such an equation should help ups conclude that $\beta = 0$ whenever we also have $\mathbb{f}(\beta,w)=0$. 
Then, we add the equation $\frac{1}{\rho}(\mathbb{L}(w - w_0),\partial_{x_2}w_0)_2$, resulting in the augmented system
\begin{align}
    \mathbb{F}(\beta,w) \bydef \begin{bmatrix}
        \frac{1}{\rho}(\mathbb{L}(w - w_0),\partial_{x_2}w_0)_2 \\
        \mathbb{f}(\beta,w)
    \end{bmatrix}.\label{f_unfold_augmented}
\end{align}
In particular, the added equation will allow to project in an orthogonal direction with respect to $\partial_{x_2}w_0$. We still denote $\mathbb{F}$ our zero finding problem of interest, for consistency of the notation in our Newton-Kantorovich approaches.

Now, we define the spaces on which we define $\mathbb{F}$. We introduce the Banach space $X$ as
\begin{align}
    X \bydef \left\{w \in L^2, ~ w = v + \sigma \partial_{x_2} w_0, ~\text{for}~ \sigma \in \mathbb{R}, v \in L^2_{D_j}\right\}. 
\end{align}
Next, motivated by the definitions in Section 5 of \cite{unbounded_domain_cadiot}, we define the spaces $H_1 \bydef \mathbb{R} \times H^l_{D_j}$ and  $H_2 \bydef \mathbb{R} \times X$ with their associated norms
\begin{align}
     \|(\beta,w)\|_{H_1} \bydef (|\beta|^2 + \|w\|_{l}^2)^{\frac{1}{2}}~\text{and}~\|(\beta,w)\|_{H_2} \bydef (|\beta|^2 + \|w\|_{2}^2)^{\frac{1}{2}}.
\end{align}
Furthermore, define $X_1 \bydef \mathbb{R} \times X^l_{\mathbb{Z}_2 \times \mathbb{Z}_1}$ and $X_2 \bydef \mathbb{R} \times \ell^2_{\mathbb{Z}_2 \times \mathbb{Z}_1}$ with their associated norms
\begin{align}
    \|(\beta,W)\|_{X_1} \bydef (|\beta|^2 + |\om| \|W\|_{l}^2)^{\frac{1}{2}}~~\text{and}~~\|(\beta,W)\|_{X_2} \bydef (|\beta|^2 + |\om| \|W\|_{2}^2)^{\frac{1}{2}}.
\end{align}
We prove in the next lemma that $H_1$ and $H_2$ are suitable spaces for the definition of $\mathbb{F}$.
\begin{lemma}\label{lem : space_F_bar}
Let $\mathbb{F}$ be given as in \eqref{f_unfold_augmented}. Then, $\mathbb{F}: H_1 \to H_2$ is well-defined and smooth.
\end{lemma}
\begin{proof}
Let $w \in H^l_{D_j}$, then it follows that $\mathbb{L}w + \mathbb{G}(w) \in L^2_{D_j}$. Therefore, $\mathbb{L}w + \mathbb{G}(w) + \beta \partial_{x_2} w_0 \in X$. The smoothness of $\mathbb{F}$ is a direct consequence of Lemma 2.4 in \cite{unbounded_domain_cadiot}.
\end{proof}
The aforementioned lemma shows that we can consider $\mathbb{F}$ as a map from $H_1$ to $H_2$. In particular, since $(0,w_0) \in H_1$, we obtain that $D\mathbb{F}(0,w_0) : H_1 \to H_2$ is a bounded linear operator. \par Now, we need to justify the introduction of the unfolding parameter $\beta$ and our specific choice for the extra equation. In particular, we need to prove that if $\tilde{x} = (\tilde{\beta},\tilde{w})$ is a solution to \eqref{f_unfold}, then $\tilde{\beta} = 0$ necessarily holds. This is necessary to ensure that the solution we prove is a solution to \eqref{eq : swift_hohenberg}. Before stating this result, we establish a preliminary lemma. Given $\theta \in \R$, define $R_\theta$ as the rotation operator by an angle $\theta.$ For simplicity, we denote $R_\theta x$, the rotation of $x \in \R^2$ by $\theta$ and $R_\theta u : x \to u(R_\theta x)$ the rotation operator acting on functions in $L^2$. Then, we obtain the following.
\begin{lemma}\label{lem : average 0 partial_y}
Recall that $w_0 = \frac{1}{j} \sum_{g \in D_j} u_0 \in H^l_{D_j}$ and that $u_0 = \gamma^\dagger_{\mathbb{Z}_2 \times \mathbb{Z}_1}(U_0) \in H^l_{\mathbb{Z}_2 \times \mathbb{Z}_1}$. Then, 
\begin{align}
    \frac{1}{j} \sum_{k = 0}^{j-1} R_{\frac{2\pi k}{j}} \partial_{x_2} w_0 = 0.
\end{align}
Moreover, if $w_0 \neq 0$, then $\partial_{x_2} w_0$ is not $D_j$-invariant.
\end{lemma}
\begin{proof}
Let $x \in \bigcup_{k = 0}^{j-1} R_{\frac{2\pi k}{k}} \Omega_0$. Then, by definition of $w_0$ we have 
\begin{align}
    \partial_{x_2} w_0(x) =  \frac{i}{j} \sum_{k = 0}^{j-1} \sum_{n \in \mathbb{Z}^2} \left(\tilde{n}_2 \cos\left(\frac{2\pi k}{j}\right) -\tilde{n}_1 \sin\left(\frac{2\pi k}{j}\right)\right) u_n e^{i \tilde{n} \cdot R_{\frac{2\pi k}{j}} x}.
\end{align}
Let us now consider the $D_j$-average of $\partial_{x_2} w_0$. In particular, using the change of index $s = k + p$, we get
\begin{align}
    \frac{1}{j} \sum_{p = 0}^{j-1} \partial_{x_2} w_0(R_{\frac{2\pi p}{j}}x) 
    &= \frac{i}{j^2} \sum_{k = 0}^{j-1} \sum_{p = 0}^{j-1} \sum_{n \in \mathbb{Z}^2} \left(\tilde{n}_2 \cos\left(\frac{2\pi k}{j}\right) -\tilde{n}_1 \sin\left(\frac{2\pi k}{j}\right)\right) u_n e^{i \tilde{n} \cdot R_{\frac{2\pi (k+p)}{j}} x} \\
    &=\frac{i}{j^2} \sum_{k = 0}^{j-1} \sum_{s = k}^{j - 1 + k} \sum_{n \in \mathbb{Z}^2} \left(\tilde{n}_2 \cos\left(\frac{2\pi k}{j}\right) -\tilde{n}_1 \sin\left(\frac{2\pi k}{j}\right)\right) u_n e^{i \tilde{n} \cdot R_{\frac{2\pi s}{j}} x} \\
    &=\frac{i}{j^2} \sum_{k = 0}^{j-1} \sum_{s = 0}^{j - 1} \sum_{n \in \mathbb{Z}^2} \left(\tilde{n}_2 \cos\left(\frac{2\pi k}{j}\right) -\tilde{n}_1 \sin\left(\frac{2\pi k}{j}\right)\right) u_n e^{i \tilde{n} \cdot R_{\frac{2\pi s}{j}} x} \\
    &= \frac{i}{j^2}  \sum_{s = 0}^{j - 1} \sum_{n \in \mathbb{Z}^2} u_n e^{i \tilde{n} \cdot R_{\frac{2\pi s}{j}} x} \left[\sum_{k = 0}^{j-1} \left(\tilde{n}_2 \cos\left(\frac{2\pi k}{j}\right) -\tilde{n}_1 \sin\left(\frac{2\pi k}{j}\right)\right) \right] \\
    &= \frac{i}{j^2}  \sum_{s = 0}^{j - 1} \sum_{n \in \mathbb{Z}^2} u_n e^{i \tilde{n} \cdot R_{\frac{2\pi s}{j}} x} \left[\tilde{n}_2\sum_{k = 0}^{j-1}  \cos\left(\frac{2\pi k}{j}\right) -\tilde{n}_1 \sum_{k = 0}^{j-1} \sin\left(\frac{2\pi k}{j}\right) \right].
\end{align}
By definition of the roots of unity, it follows that 
\begin{align} 
\sum_{k = 0}^{j-1} \cos\left(\frac{2\pi k}{j}\right) = \sum_{k = 0}^{j-1} \sin\left(\frac{2\pi k}{j}\right) = 0.
\end{align}
Hence, we obtain that
\begin{align}
    \frac{1}{j} \sum_{p = 0}^{j-1} \partial_{x_2} w_0(R_{\frac{2\pi p}{j}}) = 0.\label{showing average partial is 0}
\end{align}
In order to prove the second conclusion, suppose by way of contradiction that $\partial_{x_2} w_0$ is $D_j$-invariant. It follows that
\begin{align}
    \frac{1}{j}\sum_{p = 0}^{j-1} R_{\frac{2\pi p}{j}} \partial_{x_2} w_0  = \partial_{x_2} w_0.
\end{align}
But \eqref{showing average partial is 0} yields that $\frac{1}{j}\sum_{p = 0}^{j-1} R_{\frac{2\pi p}{j}} \partial_{x_2} w_0  = 0$. Hence, we obtain that $\partial_{x_2} w_0 = 0$, and since $w_0 \in H^l_{D_j}$, it implies that $w_0=0$. This is a contradiction and allows to conclude the proof.
\end{proof}
In order to be able to use Lemma \ref{lem : average 0 partial_y} to our case, we simply need to verify that $ w_0 \neq 0$, which can be done by verifying that $w_0(b) = \frac{1}{j}\sum_{k = 0}^{j-1}u_0(R_{\frac{2\pi k}{j}} b) \neq 0$ for some $b \in \bigcup_{k = 0}^{j-1} R_{\frac{2\pi k}{j}} \om$. We can do this by rigorously evaluating the Fourier series thanks to rigorous numerics (cf. \cite{julia_cadiot_blanco_symmetry}). Supposing that $w_0 \neq 0$, we obtain that following result.
\begin{lemma}\label{beta_is_0}
Suppose that $w_0 \neq 0$ and suppose  $(\tilde{\beta},\tilde{w}) \in H_1$ is a zero of $\mathbb{F}$ given in \eqref{f_unfold}. Then, $\tilde{\beta} = 0$. In particular $\mathbb{L}\tilde{w} + \mathbb{G}(\tilde{w}) = 0$.
\end{lemma}
\begin{proof}
Since $(\tilde{\beta},\tilde{w})$ is a zero of $\mathbb{F}$, we have
\begin{align}
    \mathbb{L}\tilde{w} + \mathbb{G}(\tilde{w}) = -\tilde{\beta} \frac{\partial_{x_2} w_0}{\rho}.\label{step_in_proof_beta_0}
\end{align}
Using that $\tilde{w} \in H^l_{D_j}$, it follows that $\mathbb{L}\tilde{w} + \mathbb{G}(\tilde{w}) \in L^2_{D_j}$. 
Hence,  we obtain that $\tilde{\beta} \frac{\partial_{x_2} w_0}{\rho}$ is $D_j$-invariant. Since $w_0 \neq 0$ and $w_0 \in H^l_{D_j}$ by construction, Lemma \ref{beta_is_0} implies that we have $\tilde{\beta} =0.$
\end{proof}
With Lemma \ref{beta_is_0}  available, we have shown that the zeros of \eqref{f_unfold} provide zeros of  \eqref{eq : f(u)=0 on Hl_G} in $H^l_{D_j}$. In particular, we will consider $x_0 \bydef (0,w_0)$ our approximate solution. At this step, we have chosen $\mathbb{F},$ and $x_0 = (0,w_0)$. It remains to choose the operator $\mathbb{A}$, which will be the focus on the next section.
\subsubsection{Construction of \texorpdfstring{$\mathbb{A}$}{Abar}}
We wish to build $\mathbb{A} : \mathbb{R} \times L^2_{\mathbb{Z}_2\times \mathbb{Z}_1} \to \mathbb{R} \times H^l_{\mathbb{Z}_2 \times \mathbb{Z}_1}$ approximating the inverse of $D\mathbb{F}(0,w_0)$. To begin, given $u \in L^2$ and $U \in \ell^2$, define $u^*$ and $U^*$ as the duals in $L^2$ and $\ell^2$ respectively of $u$ and $U$. More specifically, we have 
\begin{align}
    u^*(v) \bydef (u,v)_2, ~~ U^*(V) \bydef (U,V)_2 ~~~~ \text{ for all } v \in L^2 \text{ and } V \in \ell^2.
    \label{dual definition}
\end{align}
 Moreover, we slightly abuse notation and, given $u \in L^2$ and $U \in \ell^2$, we denote $u^* \mathbb{L} : H^l \to \mathbb{R}$ and $U^* L : X^l \to \mathbb{R}$ the following linear operators
\begin{align}\label{def : dual compose L}
    u^* \mathbb{L}v \bydef (u, \mathbb{L}v)_2, ~~ U^* LV \bydef (U, LV)_2  ~~~~ \text{ for all } v \in L^2 \text{ and } V \in \ell^2.
\end{align}
Using the above notation, we notice that 
\begin{align}
    D\mathbb{F}(0,w_0) = \begin{bmatrix}
        0 & \left(\frac{\partial_{x_2}w_0}{\rho}\right)^* \mathbb{L} \\
        \frac{\partial_{x_2} w_0}{\rho} & D\mathbb{f}(w)
    \end{bmatrix}
\end{align}
where $w_0$ is defined  as $w_0 = \frac{1}{|D_j|}\sum_{g \in D_j} gu_0$. In practice, we expect $u_0$ to be a good approximation of $w_0$, in the sense that $u_0$ is "almost" $D_j$-invariant and $\|u_0-w_0\|_l$ is small. In fact, $u_0$ is simplest to work with since it possesses a Fourier coefficients representation $U_0 = \gamma_{\mathbb{Z}_2\times \mathbb{Z}_1}(u_0)$. Motivated by this, we define the auxiliary operator $\mathbb{Q}: \mathbb{R} \times H^l_{\mathbb{Z}_2 \times \mathbb{Z}_1} \to \mathbb{R} \times L^2_{\mathbb{Z}_2 \times \mathbb{Z}_1}$  as follows
\begin{align}
    \mathbb{Q}(0,u) \bydef \begin{bmatrix}
        0 & \left(\frac{\partial_{x_2}u_0}{\rho}\right)^* \mathbb{L} \\
        \frac{\partial_{x_2} u_0}{\rho} & D\mathbb{f}(u)
    \end{bmatrix}~~\text{and}~~Q(0,U) \bydef \begin{bmatrix}
        0 & \left(\frac{\partial_{x_2} U_0}{\rho}\right)^* \mathbb{L} \\
        \frac{\partial_{x_2} U_0}{\rho} & L + DG(U)
    \end{bmatrix}.\label{def : Q}
\end{align}
Intuitively, $\mathbb{Q}(0,u_0)$ approximates $D\mathbb{F}(0,w_0)$ and we can control the difference $\mathbb{Q}(0,u_0) - D\mathbb{F}(0,w_0)$ using $\|w_0 -u_0\|_{H^l}$ (cf Lemma \ref{lem : Z_1_extra}).

We now define $B^N: \mathbb{R} \times \ell^2_{\mathbb{Z}_2 \times \mathbb{Z}_1} \to \mathbb{R} \times \ell^2_{\mathbb{Z}_2 \times \mathbb{Z}_1}$ as a finite dimensional operator satisfying
\begin{align}
    B^N \bydef \begin{bmatrix}
        1 & 0 \\
        0 & \pi^N
    \end{bmatrix} B^N \begin{bmatrix}
        1 & 0 \\
        0 & \pi^N
    \end{bmatrix}.
\end{align}
Note that the above yields that $B^N$ can be represented by a matrix.
In fact, we choose $B^N$ to be a numerical approximate inverse to $\pi^N Q(0,U_0)^{-1} L^{-1}\pi^N$ and then follow the construction performed in Section 5 of \cite{unbounded_domain_cadiot}. More specifically, we define $\mathbb{B}: \mathbb{R} \times L^2_{\mathbb{Z}_2 \times \mathbb{Z}_1} \to \mathbb{R} \times L^2_{\mathbb{Z}_2 \times \mathbb{Z}_1}$ as
\begin{align}\label{eq : operator B extra equation}
    \mathbb{B} \bydef \begin{bmatrix}
        0 & 0 \\
        0 & \mathbb{1}_{\mathbb{R}^2 \setminus \om} \end{bmatrix}
        +
        \begin{bmatrix} 1 & 0 \\
        0 & \gamma_{\mathbb{Z}_2 \times \mathbb{Z}_1}^\dagger \end{bmatrix}\left(\begin{bmatrix}
        0 & 0 \\
        0 & \pi_N\end{bmatrix} + B^N\right)\begin{bmatrix}
            1 & 0 \\
            0 & \gamma_{\mathbb{Z}_2 \times \mathbb{Z}_1}
        \end{bmatrix}.
\end{align}
Then, by using $\mathbb{B}$, we can define the operator $\mathbb{A} : \mathbb{R} \times L^2_{\mathbb{Z}_2 \times \mathbb{Z}_1} \to \mathbb{R} \times H^l_{\mathbb{Z}_2 \times \mathbb{Z}_1}$ as
\begin{align}
    \mathbb{A} \bydef \begin{bmatrix}
        1 & 0 \\
        0 & \mathbb{L}^{-1}
    \end{bmatrix} \mathbb{B}.
\end{align}
Using that $\begin{bmatrix}
        1 & 0 \\
        0 & \mathbb{L}^{-1}
    \end{bmatrix}$ acts as an isometric isomorphism between $\R \times L^2_{\mathbb{Z}_2 \times \mathbb{Z}_1} \to \R \times H^l_{\mathbb{Z}_2 \times \mathbb{Z}_1}$, we obtain that $\mathbb{A} : \mathbb{R} \times L^2_{\mathbb{Z}_2 \times \mathbb{Z}_1} \to \mathbb{R} \times H^l_{\mathbb{Z}_2 \times \mathbb{Z}_1}$ is a well-defined bounded linear operator. Such a construction will be justified in the next section. In particular, we will provide the Radii-Polynomial Theorem for this case, of which we can compute the explicit bounds as the justification (see Section \ref{sec : Z1 nonisolate} for instance).
\subsubsection{Newton-Kantorovich Approach in the case \texorpdfstring{$\mathcal{H} < \mathcal{G}$}{HleqG} and \texorpdfstring{$\mathcal{H}$}{H} does not isolate the solution}
Before stating the Radii Polynomial Theorem in this case, we need to take care of a small technicality. In order to derive the $\mathcal{Y}_s$ bound, as done in Corollary \ref{th: radii polynomial s}, we need to verify that $D\mathbb{F}(0,w_0)^{-1}$ is invariant under the rotation operator. In general, this is false for the spaces under which $D\mathbb{F}(0,w_0)^{-1}$ is defined. In the specific case we need though, we can recover a sufficient result. Before stating this result, define the operator 
\begin{align}\overline{R}_{\theta} \bydef \begin{bmatrix}
    0 & 0 \\
    0 & R_{\theta}
\end{bmatrix}\label{def : R_theta_bar}
\end{align}
We now state an  essential lemma in our analysis. 
\begin{lemma}\label{lem : DF_bar_commutes}
Let $u \in  L^2_{\mathbb{Z}_2 \times \mathbb{Z}_1}$. Let $\overline{R}_{\theta}$ be defined as in \eqref{def : R_theta_bar} and consider the restriction of $D\mathbb{F}(0,w_0)^{-1} : \{0\} \times L^2_{\mathbb{Z}_2 \times \mathbb{Z}_1} \to H_1$. Then, 
\begin{align}
    \frac{1}{m}\sum_{k = 0}^{m-1} \overline{R}_{\theta}D\mathbb{F}(0,w_0)^{-1} \begin{bmatrix}
        0 \\ u
    \end{bmatrix} = D\mathbb{F}(0,w_0)^{-1} \sum_{k = 0}^{m-1}\overline{R}_{\frac{2\pi k}{m}} \begin{bmatrix}
        0 \\ u
    \end{bmatrix}. 
\end{align}
\end{lemma}
\begin{proof}
Suppose that
\begin{align}
    D\mathbb{F}(0,w_0)^{-1} \begin{bmatrix}
        0 \\ u
    \end{bmatrix} = \begin{bmatrix}
        \sigma \\ v
    \end{bmatrix}
\end{align}
for some $(\sigma,v) \in H_1$. Then, observe that
\begin{align}
    \begin{bmatrix}
        0 \\ u
    \end{bmatrix} = D\mathbb{F}(0,w_0) \begin{bmatrix}
        \sigma \\ v
    \end{bmatrix} = \begin{bmatrix}
        \left(\frac{\partial_{x_2}w_0}{\rho}\right)^* \mathbb{L} v \\
        \sigma \frac{\partial_{x_2} w_0}{\rho} + (\mathbb{L} + D\mathbb{G}(w_0)) v
    \end{bmatrix}.\label{DF_inverse_multiplication}
\end{align}
Now, consider 
\begin{align}
    D\mathbb{F}(0,w_0) \sum_{k = 0}^{j-1}\overline{R}_{\frac{2\pi k}{j}}\begin{bmatrix}
        \sigma \\ v
    \end{bmatrix} &= \begin{bmatrix}
        \left(\frac{\partial_{x_2}w_0}{\rho}\right)^* \mathbb{L} \sum_{k = 0}^{j-1} R_{\frac{2\pi k}{j}} v \\
        (\mathbb{L} + D\mathbb{G}(w_0))\sum_{k = 0}^{j-1} R_{\frac{2\pi k}{j}}v
    \end{bmatrix}.
\end{align}
Since $\mathbb{L}$ and $\mathbb{L} + D\mathbb{G}(w_0)$ are rotationally invariant, we obtain
\begin{align}
    \begin{bmatrix}
        \left(\frac{\partial_{x_2}w_0}{\rho}\right)^*  \sum_{k = 0}^{j-1} R_{\frac{2\pi k}{j}} \mathbb{L}v \\
        (\mathbb{L} + D\mathbb{G}(w_0))\sum_{k = 0}^{j-1} R_{\frac{2\pi k}{j}}v
    \end{bmatrix} = \begin{bmatrix}
        \frac{1}{\rho}\sum_{k = 0}^{j-1} \left(\partial_{x_2}w_0\right)^* R_{\frac{2\pi k}{j}}\mathbb{L}  v \\
        \sum_{k = 0}^{j-1} R_{\frac{2\pi k}{j}}(\mathbb{L} + D\mathbb{G}(w_0))v
    \end{bmatrix} 
    &= \begin{bmatrix}
        \frac{1}{\rho}\sum_{k = 0}^{j-1} (R_{\frac{2\pi k}{j}}^*\partial_{x_2} w_0)^* \mathbb{L}  v \\
        \sum_{k = 0}^{j-1} R_{\frac{2\pi k}{j}}(\mathbb{L} + D\mathbb{G}(w_0))v
    \end{bmatrix} \\
    &= \begin{bmatrix}
        \frac{1}{\rho}\sum_{k = 0}^{j-1} (R_{\frac{2\pi k}{j}}^*\partial_{x_2} w_0)^* \mathbb{L}  v \\
        \sum_{k = 0}^{j-1} R_{\frac{2\pi k}{j}}(\mathbb{L} + D\mathbb{G}(w_0))v
    \end{bmatrix}. 
    \end{align}
Since we are averaging all the rotations, it follows that
\begin{align}
    \begin{bmatrix}
        \frac{1}{\rho}\sum_{k = 0}^{j-1} (R_{\frac{2\pi k}{j}}^*\partial_{x_2} w_0)^* \mathbb{L}  v \\
        \sum_{k = 0}^{j-1} R_{\frac{2\pi k}{j}}(\mathbb{L} + D\mathbb{G}(w_0))v
    \end{bmatrix}  &= \begin{bmatrix}
        \frac{1}{\rho}\sum_{k = 0}^{j-1} (R_{\frac{2\pi k}{j}}\partial_{x_2} w_0)^* \mathbb{L}  v \\
        \sum_{k = 0}^{j-1} R_{\frac{2\pi k}{j}}(\mathbb{L} + D\mathbb{G}(w_0))v
    \end{bmatrix}=\begin{bmatrix}
        \frac{1}{\rho}\left(\sum_{k = 0}^{j-1} R_{\frac{2\pi k}{j}}\partial_{x_2} w_0\right)^* \mathbb{L}  v \\
        \sum_{k = 0}^{j-1} R_{\frac{2\pi k}{j}}(\mathbb{L} + D\mathbb{G}(w_0))v
    \end{bmatrix}. 
\end{align}
Now, by Lemma \ref{lem : average 0 partial_y}, specifically using \eqref{showing average partial is 0}, we obtain
\begin{align}
    \begin{bmatrix}
        \frac{1}{\rho}\left(\sum_{k = 0}^{j-1} R_{\frac{2\pi k}{j}}\partial_{x_2} w_0\right)^* \mathbb{L}  v \\
        \sum_{k = 0}^{j-1} R_{\frac{2\pi k}{j}}(\mathbb{L} + D\mathbb{G}(w_0))v
    \end{bmatrix} &= \begin{bmatrix}
        0\\
        \sum_{k = 0}^{j-1} R_{\frac{2\pi k}{j}}(\mathbb{L} + D\mathbb{G}(w_0))v
    \end{bmatrix}.\label{first_way}
\end{align}
 Additionally, observe that
{\footnotesize\begin{align}
    \sum_{k = 0}^{j-1} \overline{R}_{\frac{2\pi k}{j}} D\mathbb{F}(0,w_0) \begin{bmatrix}
        \sigma \\ v
    \end{bmatrix} &= \begin{bmatrix}
        0 \\
        \frac{\sigma}{\rho} \sum_{k = 0}^{j-1} R_{\frac{2\pi k}{j}}\partial_{x_2} w_0 + R_{\frac{2\pi k}{j}} (\mathbb{L} + D\mathbb{G}(w_0)) v
    \end{bmatrix}= \begin{bmatrix}
        0 \\
     \sum_{k = 0}^{j-1}R_{\frac{2\pi k}{j}} (\mathbb{L} + D\mathbb{G}(w_0)) v
    \end{bmatrix}\label{second_way}
\end{align}}
where the last step followed from Lemma \ref{lem : average 0 partial_y}. Noticing that \eqref{first_way} and \eqref{second_way} are equivalent, we have obtained the desired result.
\end{proof}
With Theorem  \ref{lem : DF_bar_commutes} available, we have shown that we can commute $\overline{R}_{\theta}$ with $D\mathbb{F}(0,w_0)^{-1}$ provided that $D\mathbb{F}(0,w_0)^{-1}$ is applied to an element of $\{0\} \times L^2_{\mathbb{Z}_2 \times \mathbb{Z}_1}$. Recall the definition of $\mathbb{Q}$ in \eqref{def : Q}. Choose $\mathbb{F}$ as in \eqref{f_unfold_augmented}$, x_0 \bydef (0,w_0),\mathbb{A} \approx D\mathbb{F}(0,w_0)^{-1}$ (assuming $D\mathbb{F}(0,w_0)^{-1}$ exists). We now define the fixed point operator, $\mathbb{T}$ as
\begin{align}
    \mathbb{T}(x) = x - D\mathbb{F}(0,w_0)^{-1} \mathbb{F}(x).
\end{align}
Therefore, we are now ready to state the Radii Polynomial Theorem for this case. 
\begin{theorem}\label{th: radii polynomial augmented}
Let $m > 0$ be an odd integer other than $3$. Let $\mathbb{A} : \mathbb{R} \times L^2_{\mathbb{Z}_2 \times \mathbb{Z}_1} \to \mathbb{R} \times H^l_{\mathbb{Z}_2 \times \mathbb{Z}_1}$ and let $\mathcal{Y}_0,\mathcal{Y}_s, \mathcal{Z}_1,\mathcal{Z}_s$ be non-negative constants and let $\mathcal{Z}_2(r): [0,\infty) \to (0,\infty)$ be a non-negative function such that for all $r>0$
  \begin{align}\label{eq: definition Y0 Z1 Z2 augmented}
    \left\|\mathbb{A}\begin{bmatrix}
        0 \\ \mathbb{f}(0,u_0)
    \end{bmatrix}\right\|_{H_1} \leq &\mathcal{Y}_0\\
    \left\|\mathbb{A}\begin{bmatrix}
        0 \\
        \tilde{\mathbb{G}}(u_0)
    \end{bmatrix}\right\|_{H_1} \leq &\mathcal{Y}_s\\
    \|I_d - \mathbb{A}\mathbb{Q}(0,u_0)\|_{H_1} \leq &\mathcal{Z}_1\\
    \|\mathbb{A}(\mathbb{Q}(0,u_0) - D\mathbb{F}(0,w_0))\|_{H_1} \leq &\mathcal{Z}_s \\
    \|\mathbb{A}\left(D\mathbb{F}(0,w_0) - D\mathbb{F}(h)\right)\|_{H_1} \leq &\mathcal{Z}_2(r)\|h-x_0\|_{H_1}, ~~ \text{for all } h \in B_r(0,w_0)
\end{align}  
If there exists $r>0$ such that \eqref{condition radii polynomial s} is satisfied, then there exists a unique $(0,\tilde{w}) \in \overline{B_r(0,w_0)} \subset H_1$ such that $\mathbb{F}(0,\tilde{w})=0$, where $\overline{B_r(0,w_0)}$ is the closed ball of $H_1$ centered at $(0,w_0)$ and of radius $r$.
\end{theorem}
\begin{proof}
 Proceeding similarly to how we did in Theorem \ref{th: radii polynomial s}. Observe that
\begin{align}
    \|I_d - \mathbb{A} D\mathbb{F}(0,w_0)\|_{H_1} \leq \|I_d - \mathbb{A}\mathbb{Q}(0,u_0)\|_{H_1} + \|\mathbb{A}(\mathbb{Q}(0,u_0) - D\mathbb{F}(0,w_0))\|_{H_1} \leq \mathcal{Z}_1 + \mathcal{Z}_s.
\end{align}
Since we assumed $ \mathcal{Z}_1 + \mathcal{Z}_s < 1$, this shows that $D\mathbb{F}(0,w_0)^{-1}$ exists. The derivation of $\mathcal{Z}_2$ is unchanged from Theorem \ref{th: radii polynomial s}. Next, examine
\begin{align}
    \|D\mathbb{F}(0,w_0)^{-1} \mathbb{F}(0,w_0)\|_{H_1} &= \left\|D\mathbb{F}(0,w_0)^{-1} \begin{bmatrix}
        0 \\
        \frac{1}{j}\sum_{k = 0}^{j-1} R_{\frac{2\pi k}{j}} \mathbb{f}(0,u_0) + \tilde{\mathbb{G}}(u_0)
    \end{bmatrix}\right\|_{H_1} \\
    &\hspace{-1cm}\leq \frac{1}{j}\left\|D\mathbb{F}(0,w_0)^{-1} \sum_{k = 0}^{j-1} \overline{R}_{\frac{2\pi k}{j}} \begin{bmatrix}
        0 \\
         \mathbb{f}(0,u_0)
    \end{bmatrix}\right\|_{H_1} + \left\|D\mathbb{F}(0,w_0)^{-1} \begin{bmatrix}
        0 \\
        \tilde{\mathbb{G}}(u_0)
    \end{bmatrix}\right\|_{H_1}.\label{Step_Y_0_Y_s_bar}
\end{align}
Now, using Lemma \ref{lem : DF_bar_commutes}, we have
{\footnotesize\begin{align}
    \frac{1}{j}\left\|D\mathbb{F}(0,w_0)^{-1} \sum_{k = 0}^{j-1} \overline{R}_{\frac{2\pi k}{j}} \begin{bmatrix}
        0 \\
         \mathbb{f}(0,u_0)
    \end{bmatrix}\right\|_{H_1} = \frac{1}{j}\left\| \sum_{k = 0}^{j-1} \overline{R}_{\frac{2\pi k}{j}}D\mathbb{F}(0,w_0)^{-1} \begin{bmatrix}
        0 \\
         \mathbb{f}(0,u_0)
    \end{bmatrix}\right\|_{H_1} \leq \left\|D\mathbb{F}(0,w_0)^{-1} \begin{bmatrix}
        0 \\ \mathbb{f}(0,u_0)
    \end{bmatrix}\right\|_{H_1}
\end{align}}
which allows us to return to \eqref{Step_Y_0_Y_s_bar} to get
\begin{align}
    \|D\mathbb{F}(0,w_0)^{-1} \mathbb{F}(0,w_0)\|_{H_1} &\leq \left\|D\mathbb{F}(0,w_0)^{-1} \begin{bmatrix}
        0 \\ \mathbb{f}(0,u_0)
    \end{bmatrix}\right\|_{H_1} + \left\|D\mathbb{F}(0,w_0)^{-1} \begin{bmatrix}
        0 \\
        \tilde{\mathbb{G}}(u_0)
    \end{bmatrix}\right\|_{H_1} \\
    &\leq \frac{1}{1-\mathcal{Z}_1 - \mathcal{Z}_s}\left(\left\|\mathbb{A} \begin{bmatrix}
        0 \\ \mathbb{f}(0,u_0)
    \end{bmatrix}\right\|_{H_1} + \left\|\mathbb{A} \begin{bmatrix}
        0 \\
        \tilde{\mathbb{G}}(u_0)
    \end{bmatrix}\right\|_{H_1}\right) \\
    &\leq \frac{1}{1-\mathcal{Z}_1 - \mathcal{Z}_s} (\mathcal{Y}_0 + \mathcal{Y}_s).
\end{align}
We now conclude that $\mathbb{T}$ maps $\overline{B_r(0,w_0)}$ to itself. The fact that $\mathbb{T}$ is a contraction follows by using the same steps as those of Theorem \ref{th: radii polynomial s}. This concludes the proof.
\end{proof}
\section{Computing the Bounds}\label{sec : computing the bounds}
In the previous section, we provided details on the construction of the objects required for each Radii Polynomial Theorem based on the properties of $\mathcal{G}$ and $\mathcal{H}$. Now, we demonstrate that such choices are justified by the fact that the bounds of Theorems \ref{th: radii polynomial}, \ref{th: radii polynomial s}, and \ref{th: radii polynomial augmented} can be computed explicitly, combining a meticulous analysis and rigorous numerics.
 We illustrate  the computations of these bounds in the case of the 2D Swift-Hohenberg (SH) equation, \eqref{eq : swift_hohenberg} and in the case of  dihedral symmetries. By dihedral symmetry, we mean the symmetry group $D_j$. We denote by $D_j$ the symmetry group of the $j$-gon. 
 
 Recall the planar Swift-Hohenberg (SH) equation defined in \eqref{eq : swift_hohenberg}. Since we focus on stationary localized solutions, the problem we wish to solve is
\begin{align}\label{eq : sh localized 2D}
    \mathbb{F}(u) = (I_d + \Delta)^2 u + \mu u + \nu_1 u^2 + \nu_2 u^3 = 0, ~~ u = u(x), ~~ x \in \mathbb{R}^2 \text{ and } \lim_{|x| \to \infty} u(x) = 0,
\end{align}
where $\mu > 0$ and $(\nu_1,\nu_2) \in \mathbb{R}^2$. In this case, we have
\begin{align}
    \mathbb{L} \bydef (I_d + \Delta)^2 + \mu, ~~ \mathbb{G}(u) \bydef \nu_1 u^2 + \nu_2 u^3.\label{def : L and G for SH}
\end{align}
As stated in the introduction to this manuscript, SH is known to exhibit a wide variety of localized patterns. In particular, \cite{olivier_radial} provided a rigorous proof of existence and local uniqueness of a radially symmetric solution to \eqref{eq : sh localized 2D}. The authors of \cite{jason_spot_paper}, \cite{jason_ring_paper}, and \cite{jason_review_paper} show 2D patterns with dihedral symmetry. Some of these patterns were proven by the authors of \cite{sh_cadiot} using the approach we wish to build upon in this paper. In \cite{sh_cadiot}, only $D_2$-symmetry was proven; however, the authors conjecture that the patterns proven in Theorems 4.1, 4.2, and 4.3 have $D_4, D_6,$ and $D_8$-symmetry respectively. In this section, we provide the necessary tools to verify this conjecture along with the ability to prove other localized solutions to \eqref{eq : sh localized 2D} with different dihedral symmetries. 

Specifically, given $j \in \mathbb{N}$ and $j \geq 2$, our goal is to investigate the existence of $D_j$ symmetric solutions to \eqref{eq : sh localized 2D}. In particular, using the approach presented in Section \ref{sec : strategies for computer assisted proofs}, we aim at proving the existence of solutions in $H^l_{\mathcal{G}}$, where $\mathcal{G} \bydef D_j$ in this section.  For a given value of $j$, we construct a  maximal square lattice space subgroup  $\mathcal{H}$, as defined in Definition \ref{def : maximal square lattice space subgroup}. The following lemma provides the classification of $\mathcal{H}$ for a given $\mathcal{G} = D_j$.
\begin{lemma}\label{lem : H classification}
Let $\mathcal{G} = D_j$. Let $j_0$ be the largest number in the set $\{2,4\}$ which divides $j$. Moreover, we define $\mathcal{H}$ as follows
\begin{align}
    \mathcal{H} = \begin{cases}
        D_j & j = 2,4 \\
        D_{j_0} & j ~\text{is not a prime number nor in } ~ \{2,4\} \\
        \mathbb{Z}_2 \times \mathbb{Z}_1 & j ~\text{is a prime number other than} ~ \{2\}.
    \end{cases}
\end{align}
Then, $\mathcal{H}$ is a maximal square lattice space group of $\mathcal{G}$, in the sense of Definition \ref{def : maximal square lattice space subgroup}.
\end{lemma}
\begin{proof}
The proof follows from fundamental properties of dihedral groups (cf. \cite{gallianalgebra}).
\end{proof}
With the above result available, and following Section \ref{sec : strategies for computer assisted proofs}, we adapt our approach depending on the relationship between $\mathcal{H}$ and $\mathcal{G}$. Before presenting our analysis in each case, we expose the construction of the approximate solution $u_0 \in H^l_{\mathcal{H}}$.

When following the steps of Section \ref{sec : construction of u0 gen}, one comes to the point where they need to obtain the approximate solution thanks to its Fourier coefficients representation $u_0 = \gamma^\dagger_\mathcal{H}(U_0)$, where $U_0 = \pi^{N_0}U_0$ as in \eqref{def : construction of u0}.
We construct an initial guess of the form
\begin{equation}\label{initial_guess_for_newton}
    \bar{u}_0(x) = \alpha \sum_{k = 0}^{j-1} \mathrm{sech}(\beta |R_{\frac{2\pi k}{j}}x|)
\end{equation}
for all $x \in \R^2$, where $|x| = \sqrt{|x_1|^2 + |x_2|^2} $ and $(\alpha,\beta) \in \R^2$ are parameters. In particular, we place $j$ hyperbolic secant functions in a ring to obtain a $D_j$-symmetric guess. Then, we compute  a finite truncation of the  Fourier coefficients of $\bar{u}_0$ on $\om$, which we denote as $\bar{U}_0$. In particular $\bar{U}_0 = \pi^{N_0} \bar{U}_0$, that is $\bar{U}_0$ is a vector. Playing with the parameters $\alpha$ and $\beta$, $\bar{U}_0$ serves as an initial guess for a Newton method. Indeed, we refine the sequence $\bar{U}_0$ thanks to a Newton method on the Galerkin projection of size $N_0$ for the zero finding problem $F(U)= 0$.  Finally, we still denote the refined vector as $\bar{U}_0$, and  we represent $\bar{U}_0$ in $\mathcal{H}$ symmetric sequences, that is $\bar{U}_0 \in \ell^2_\mathcal{H}$. Our approximate solution $\bar{u}_0$ becomes $\bar{u}_0  = \gamma^\dagger_\mathcal{H}(\bar{U}_0) \in L^2_\mathcal{H}$.

Now, note that $\overline{u}_0$ is not necessarily smooth, that is  $\overline{u}_0 \notin H^l$ in general. This comes from the fact that $\bar{u}_0$ is not necessarily smooth on  the boundary of $\om$, which we will denote as $\partial \om$.  In order to tackle this problem, we use the finite-dimensional trace projection given in \cite{unbounded_domain_cadiot}. In particular, such a projection allows to project $\bar{U}_0$ into a subset of Fourier coefficients which represent functions with zero trace on  $\partial \om$.  This will allow to recover smoothness on  $\partial \om$, and hence obtain a smooth extension by zero on $\R^2$.

Recall from Lemma \ref{lem : H classification} that the choices of $\mathcal{H}$ we consider are $\mathbb{Z}_2\times\mathbb{Z}_1, D_2,$ and $ D_4,$, we need to construct a trace operator valid for each of these groups (following \cite{unbounded_domain_cadiot}). To make this less tedious, we proceed in a more general fashion. Recall that we can use the orbit to unfold a sequence $U = (u_n)_{n \in J_{\mathrm{red}}(\mathcal{H})}$ to a full Fourier sequence.  Let us define a function, $\mathcal{K}: \mathbb{C}^{|J_{\mathrm{red}}(\mathcal{H})|} \to \mathbb{C}^{(2N_0+1)^2}$ which unfolds  vectors in $\pi^{N_0} \ell^2_{\mathcal{H}}$ into  vectors in $\pi^{N_0} \ell^2$ representing a sequence in the full Fourier series expansion.
In particular, the input $U = (u_n)_{n \in J_{\mathrm{red}}(\mathcal{H})}$ is represented on the computer as a vector of length $|J_{\mathrm{red}}(\mathcal{H})|$. The output $\mathcal{K}(U) = (u_n)_{n \in \mathbb{Z}^2}$ is represented on the computer as a vector of length $(2N_0 + 1)^2$. 

Now, the trace zero conditions read as follows
\begin{align}
    \partial_{n}^j u|_{ \partial  \om} = 0 \text{ for all}~ j \in \{0,1,2,3\}\label{trace condition : general}
\end{align}
where $\partial_{n}$ is the normal derivative on $\partial  \om$. In fact, since we need $u_0 \in H^l \subset H^4(\R^2)$, we only require the derivatives up to order $3$ to vanish on $\partial \om$, hence the trace conditions \eqref{trace condition : general}.
 These conditions have an equivalent in terms of Fourier coefficients, which we translate as the kernel of a finite-dimensional periodic trace operator $\mathcal{T}$ on Fourier coefficients (cf. Section 4.1 in \cite{unbounded_domain_cadiot}). The projection into the kernel of $\mathcal{T}$ can easily be achieved using Theorem 4.1 from \cite{unbounded_domain_cadiot} for instance. We denote by $U_0^{N_0,full}$ the obtained vector of size $(2N_0+1)^2$ in the kernel of $\mathcal{T}$. Then, we denote $U_0^{full}$ its padding with zeros, that is
\begin{align}
    (U_0^{full})_n \bydef \begin{cases}
        (U_0^{N_0,full})_n & \mathrm{if} ~ |n| \leq N_0 \\
        0 & \mathrm{else}
    \end{cases}.
\end{align}
Moreover, since $U_0^{N_0,full}$ is in the kernel of each $\mathcal{T}$ by construction, we define
\begin{align}
    u_0^{full} \bydef \gamma^\dagger(U_0^{full})
\end{align}
and obtain that $u_0^{full}$ satisfies \eqref{trace condition : general}. In particular, we obtain that $u_0^{full} \in H^l$, but not necessarily in $H^l_{\mathcal{H}}$. In order to obtain this, we use \eqref{v_D_m_fourier} and construct
\begin{align}
    u_0 \bydef \frac{1}{|\mathcal{H}|} \sum_{h \in \mathcal{H}} h \cdot u_0^{full}.\label{u_0 average after trace}
\end{align}
Since $\mathcal{H}$ is a space group, the above average will still yield a Fourier series on $ \om$. Therefore, we have successfully built an approximate solution $u_0 \in H^l_{\mathcal{H}}$ such that $\mathrm{supp}(u_0) \subset \overline{\om}$, associated to its Fourier coefficients representation $U_0 = \pi^{N_0} U_0 \in X^l_{\mathcal{H}}$.
\subsection{When \texorpdfstring{$j = 2,4$}{j24}}
We wish to apply Theorem \ref{th: radii polynomial} as we have $\mathcal{H} = \mathcal{G}$. This means $j = 2$ or $4$ so we have $\mathcal{H} = D_2$ or $D_4$. We now introduce some notation. \begin{align} v_0 \bydef 2\nu_1 u_0 + 3\nu_2 u_0^2,~V_0 \bydef \gamma_{\mathcal{G}}(v_0), ~V_0^N \bydef \pi^{2N} V_0, ~v_0^N = \gamma^\dagger_{\mathcal{G}}(V_0^N).\label{def : v_0}
\end{align}
Moreover, notice that $\mathbb{V}_0 = D\mathbb{G}(u_0)$, the multiplication operator for $v_0$. We define $\mathbb{F}$ as in \eqref{eq : sh localized 2D}. We choose $u_0 \in H^l_{\mathcal{G}}$ an approximate solution to \eqref{eq : sh localized 2D} constructed as in Section \ref{sec : construction of u0 gen}. We will choose $\mathbb{A} : L^2_{D_j} \to H^l_{D_j}$ to be an approximate inverse to $D\mathbb{F}(u_0)$ built as in Section \ref{sec : operator A first}. We now state the values of $\mathcal{Y}_0, \mathcal{Z}_1 >0$ and  $\mathcal{Z}_2 : (0, \infty) \to [0,\infty)$ in this case in the following lemma.
\begin{lemma}\label{th: radii polynomial comp}
Let $\mathbb{A} : L_{D_j}^2 \to H_{D_j}^l$ be a bounded linear operator. Moreover, let $\kappa, Z_1$, and $\mathcal{Z}_u$ be non-negative constants such that
{\footnotesize\begin{align}
    \kappa \bydef \frac{2\sqrt{\mu} + (1 + \mu)(2\pi - 2\arctan(\sqrt{\mu})}{8\mu^{\frac{3}{2}}(1+\mu)},~
    \|I_d - B(I_d + \mathbb{V}_0^NL^{-1})\|_{2} \leq Z_1,~
    \|(\Gamma_{\mathcal{G}}^\dagger(L^{-1}) - \mathbb{L}^{-1})\mathbb{v}_0^N\|_{2} \leq \mathcal{Z}_u\label{def : kappa}
\end{align}}
where $\mathbb{V}_0^N$ and $\mathbb{v}_0^N$ are the multiplication operators for $V_0^N$ and $v_0^N$ defined in \eqref{def : v_0} respectively. Additionally, let $\mathcal{Y}_0, \mathcal{Z}_1$ be non-negative constants and let  $\mathcal{Z}_2 : (0, \infty) \to [0,\infty)$ be a non-negative function  such that
  \begin{align}\label{eq: definition Y0 Z1 Z2 comp}
     \mathcal{Y}_0 &\bydef \sqrt{ |\om|} (\|B^N F(U_0)\|_{2}^2 + \|(\pi^{N_0} - \pi^N) L U_0 + (\pi^{3N_0} - \pi^N) G(U_0)\|_{2}^2)^{\frac{1}{2}} \\
     \mathcal{Z}_1 &\bydef Z_1 + \max\{1,\|B^N\|_{2}\}\left(\mathcal{Z}_u + \frac{1}{\mu} \|V_0 - V_0^N\|_{1}\right)
    \\
    \mathcal{Z}_2(r) &\bydef 3|\nu_2| \frac{\kappa^2}{\mu} \max\{1,\|B^N\|_{2}\}r + \frac{\kappa}{\mu} \max\{2|\nu_1|, (\|\mathbb{W}(B^N)^*\|_{2}^2 + \|W\|_{1}^2)^{\frac{1}{2}}\}
\end{align}  
 where $W \bydef (w_k)_{k \in J_{\mathrm{red}}(D_j)} = (2\nu_1\delta_{k}+6\nu_2u_k)_{k \in J_{\mathrm{red}}(D_j)}$ and $\delta_k$ is the Kronecker symbol. Moreover, $\mathbb{W}$ is the discrete convolution operator associated to $W$. Then, $\mathcal{Y}_0, \mathcal{Z}_1, $ and $\mathcal{Z}_2(r)$ satisfy the estimates stated in Theorem \ref{th: radii polynomial}.
\end{lemma}
\begin{proof}
To compute $\mathcal{Y}_0$ and $\mathcal{Z}_2$, the result follows those presented in \cite{unbounded_domain_cadiot}. The details for $\mathcal{Z}_1$ will be provided in Lemmas \ref{lem : Z_1 periodic} and \ref{lem : Zu}.
\end{proof}
We will now provide the additional details necessary to compute the bounds $Z_1$ and $\mathcal{Z}_u$. We will begin with $Z_1$, which we state in the following lemma.
\begin{lemma}\label{lem : Z_1 periodic}
Let $M^N = \pi^N + \mathbb{V}_0^NL^{-1}$ and $M = \pi_N + M^N$. Let $Z_1 > 0$ be such that
\begin{align}
    Z_1 \bydef \varphi(Z_{1,1},Z_{1,2},Z_{1,3},Z_{1,4}) 
\end{align}
where $\varphi$ is defined in Lemma 4.2 of \cite{gs_cadiot_blanco} and
\begin{align}
    &Z_{1,1} \bydef \sqrt{\|(\pi^N - B^NM^N) (\pi^N - (M^N)^{*} (B^N)^{*})\|_{2}} ,~Z_{1,3} \bydef \sqrt{\|\pi^N L^{-1} \mathbb{V}_0^N \pi_N \mathbb{V}_0^N L^{-1} \pi^N\|_{2}}\\
    &Z_{1,2} \bydef \max_{n \in J_{\mathrm{red}(\mathcal{G})} \setminus I^N} \frac{1}{|l(\tilde{n})|} \sqrt{\|B^N\mathbb{V}_0^N \pi_N \mathbb{V}_0^N B^N\|_{2}},~Z_{1,4} \bydef \max_{n \in J_{\mathrm{red}}(\mathcal{G}) \setminus I^N} \frac{1}{|l(\tilde{n})|} \|V_0^N\|_{1}.
\end{align}
Then, $\|I_d - B(I_d + \mathbb{V}_0^NL^{-1})\|_{2} \leq Z_1$ and $\|I_d - ADF(U_0)\|_{l} \leq Z_1 + \frac{1}{\mu} \|B\|_{2} \|V_0^N - V_0\|_{1}.$
\end{lemma}
\begin{proof}
We begin by examining $I - BM$.
\begin{align}
    I - BM = \begin{bmatrix}
        \pi^N (I - BM)\pi^N & \pi^N (I - BM)\pi_N \\
        \pi_N (I - BM)\pi^N & \pi_N (I - BM)\pi_N
    \end{bmatrix} 
    &= \begin{bmatrix}
         \pi^N - B^NM^N & -B^N \mathbb{V}_0^N L^{-1} \pi_N \\ 
        -\pi_N \mathbb{V}_0^N L^{-1}\pi^N & -\pi_N \mathbb{V}_0^N L^{-1} \pi_N
    \end{bmatrix}.\label{step_in_Z1_isolated}
\end{align}
Now, we will use Lemma 4.2 from \cite{gs_cadiot_blanco} on \eqref{step_in_Z1_isolated}. This means we must compute the norm of each of the blocks of \eqref{step_in_Z1_isolated}. We begin with $\pi^N - B^NM^N$.
\begin{align}
     \|\pi^N - B^NM^N\|_{2}^2 &= \|(\pi^N - B^NM^N) (\pi^N - (M^N)^{*} (B^N)^{*})\|_{2} \bydef Z_{1,1}^2.\label{definition_of_Z11_isolated}
\end{align}
Next, we examine $- B^N \mathbb{V}_0^N L^{-1} \pi_N$. Using the properties of the adjoint, we get
\begin{align}
    \nonumber \|-B^N \mathbb{V}_0^N L^{-1} \pi_N\|_{2}^2    = \|\pi_N L^{-1} \mathbb{V}_0^N (B^N)^{*}\|_{2}^2 &\leq \|\pi_N L^{-1}\|_{2}^2 \|\pi_N \mathbb{V}_0^N (B^N)^{*}\|_{2}^2 \\
    &\hspace{-1.5cm}\leq \max_{n \in J_{\mathrm{red}(\mathcal{G})} \setminus I^N} \frac{1}{|l(\tilde{n})|^2} \|B^N\mathbb{V}_0^N \pi_N \mathbb{V}_0^N B^N\|_{2} \bydef Z_{1,2}^2.\label{step_in_Z1_2_isolated}
\end{align}
\par Next, we consider the term $-\pi_N \mathbb{V}_0^N L^{-1} \pi^N$.
\begin{align}
    \|-\pi_N \mathbb{V}_0^N L^{-1} \pi^N\|_{2}^2 
    &= \|\pi^N L^{-1} \mathbb{V}_0^N \pi_N \mathbb{V}_0^N L^{-1} \pi^N\|_{2} \bydef Z_{1,3}^2\label{step_in_Z1_3_isolated}
\end{align}
Lastly, we treat the block $-\pi_N \mathbb{V}_0^N L^{-1} \pi_N$.
\begin{align}
    \|-\pi_N \mathbb{V}_0^N L^{-1} \pi_N\|_{2}^2     = \|\pi_N L^{-1} \mathbb{V}_0^N \pi_N\|_{2}^2 \leq \|\pi_N L^{-1}\|_{2} \|\pi_N \mathbb{V}_0^N\pi_N\|_{2} &\leq \max_{n \in J_{\mathrm{red}}(\mathcal{G}) \setminus I^N} \frac{1}{|l(\tilde{n})|^2} \|\mathbb{V}_0^N\|_{2} \\
    &\hspace{-1.5cm}\leq \max_{n \in J_{\mathrm{red}}(\mathcal{G}) \setminus I^N} \frac{1}{|l(\tilde{n})|} \|V_0^N\|_{1} \bydef Z_{1,4}.\label{step_in_Z1_4_isolated}
\end{align}
With \eqref{step_in_Z1_isolated}, \eqref{step_in_Z1_2_isolated}, \eqref{step_in_Z1_3_isolated}, and \eqref{step_in_Z1_4_isolated} computed, we are able to compute the $Z_1$ bound. The extra term follows from the proof of Lemma 3.4 in \cite{sh_cadiot}.
\end{proof}
We now state the lemma for the $\mathcal{Z}_u$ bound.
\begin{lemma}\label{lem : Zu}
Let $a,b,C_0,C_1(d),C_{12}(d),$ and $C_2(d) > 0$ be  constants defined as follows
\begin{align} 
a &\bydef \frac{\sqrt{-1+\sqrt{1+\mu}}}{2}, ~~ b \bydef \sqrt{2}a - i \frac{\sqrt{\mu}}{2\sqrt{2}a} \\
C_0 &\bydef \sup_{r \in [0,\infty)}  e^{\sqrt{2} ar}\left|\frac{1}{2i\sqrt{\mu}}\bigg(K_0(br) - K_0(\overline{b}r) \bigg)\right| \\
C_1(d) &\bydef 4\left(\frac{2ad+1+e^{-2ad}}{a^2} + e^{-2ad}\left(4d + \frac{e^{-2ad}}{a}\right) + \left(\frac{1+e^{-2ad}}{a} + 2d
    \right)\frac{2e^{-1}+1}{a(1-e^{-ad})}\right) \\
     & \qquad + ~   \frac{4(2e^{-1}+1)^2}{a^2(1-e^{-ad})^2} +\frac{2}{a}\left( \frac{1+e^{-2ad}}{a} + 2d + e^{-2ad}(4d+\frac{e^{-2ad}}{a})+ \frac{2e^{-1}+1}{a(1-e^{-ad})}\right)\\
     C_{12}(d) &\bydef  8\left(2d + \frac{1}{2a}\right) \left(2d + \frac{1+e^{-2ad}}{2a} + \left(2d + \frac{3+e^{-2ad}}{2a}\right)\frac{1}{1-e^{-ad}} +  \frac{4e^{-1} + 1 +e^{-2ad}}{2a(1-e^{-ad})^2}\right)\\
      C_2(d) &\bydef \frac{2}{a} \left[\frac{1+e^{-2ad}}{a} + 2d + e^{-2ad}(4d + \frac{e^{-2ad}}{a}) + \frac{(2e^{-1}+e^{-2ad})}{a(1-e^{-ad})} \right].
\end{align} 
Next, let $E_1, E_2 \in \ell^2_{\mathcal{G}}$ be defined as 
\begin{align}
    E_1 \bydef \gamma_{\mathcal{G}}(\mathbb{1}_{\om}(x) \cosh(2ax_1))~~\text{and}~~ E_2 \bydef \gamma_{\mathcal{G}}(\mathbb{1}_{\om}(x) \cosh(2ax_2)).
\end{align}
Let $v_0^N$ be defined as in \eqref{def : v_0}. Let $v^N_{0,e},$ and $v^N_{0,o}$ be the even and odd parts in the $x_2$ direction of $v_0^N$ respectively. Then, let $V^N_{0,s} \bydef \gamma_{\mathcal{H}}(v^N_{0,s})$ for $s \in \{e,o\}$.
Finally, define $\mathcal{Z}_{u,2,s} > 0$ for $s \in \{e,o\}$
\begin{align}
    \mathcal{Z}_{u,2,s}^2 &\bydef \biggl(\frac{C_0^2 e^{-2ad} |\om|}{a^2} (V_{0,s}^N,V_{0,s}^N * \pi^{4N}(E_1 + E_2))_2  \\
    &+e^{-4ad}C_0^2|\om|(V_{0,s}^N,V^N_{0,s} * [C_1(d) \pi^{4N} E_1+ C_{1,2}(d)\pi^{4N} E_{1,2}+  C_{2}(d)\pi^{4N}E_2])_2 \biggr) .
\end{align}
Following this, let $\mathcal{Z}_{u,1}, \mathcal{Z}_{u,2},\mathcal{Z}_u > 0$ be such that
{\small\begin{align}
    \mathcal{Z}_{u,1} \bydef \sqrt{\frac{C_0^2 e^{-2ad} |\om|}{a^2} (V_0^N,V_0^N * \pi^{4N}(E_1 + E_2))_2},~\mathcal{Z}_{u,2} \bydef \sum_{s \in \{e,o\}}  \mathcal{Z}_{u,2,s},~\mathcal{Z}_u \bydef \sqrt{\mathcal{Z}_{u,1}^2 + \mathcal{Z}_{u,2}^2}.
\end{align}}
Then, it follows that $ \mathcal{Z}_u$ satisfies the estimate in \eqref{def : kappa}.
\end{lemma}
\begin{proof}
The proof follows using similar steps to Lemma 3.12 of \cite{sh_cadiot} with some additional factors due to the cases where we do not have even symmetry.
\end{proof}
With Lemma \ref{lem : Zu}. available, we can compute the $\mathcal{Z}_u$ bound. The computation is very similar to that performed in \cite{sh_cadiot}. For simplicity, we refer the interested reader to the aforementioned work. We have included details on its computation in the code for the proofs found at \cite{julia_cadiot_blanco_symmetry}.
\subsection{When \texorpdfstring{$j$}{j} is not a prime number nor \texorpdfstring{$4$}{24}}
\label{sec : HleqG H isolates bounds}
In this section, $\mathcal{H}$ isolates the solution. Hence, it follows that $\mathcal{H} = D_{j_0}$ where $j_0 \in \{2,4\}$ and $j_0$ divides $j$. Following the strategy presented in Section \ref{sec : H isolate}, we present the technical computations of the bounds of Theorem \ref{th: radii polynomial s}. One can readily notice that the bounds $\mathcal{Y}_0$ and  $\mathcal{Z}_1$ are exactly the same as their counterpart in Theorem \ref{th: radii polynomial}. Consequently, it remains to describe the computation of the bounds $\mathcal{Y}_s$, $\mathcal{Z}_s$ and $\mathcal{Z}_2$. 
\subsubsection{The Bound \texorpdfstring{$\mathcal{Y}_s$}{Ys}}\label{sec : Y_s isolate}
Before presenting our analysis for the bound $\mathcal{Y}_s$, we introduce notation and preliminary results. Recall that, given $\theta \in \R$, we slightly abuse notation and denote $R_\theta$ as the rotation operator by the angle $\theta$, whether it applies to a vector in $\R^2$, a function, or a sequence of Fourier coefficients. We will define $w_0$ in this section to be. 
\begin{align}\label{def : w0 for dihedral}
    w_0 = \frac{1}{j} \sum_{k = 0}^{j-1}  R_{\frac{2\pi k}{j}} u_0
\end{align}
in the case of $D_j$ symmetry. Notice that we have only added the rotations of $u_0$. It is clear that this is sufficient given that $u_0 \in H^l_{\mathcal{H}}$  which already has the necessary reflections included in the $\mathcal{H}$-symmetry since every case has at least $D_2$ symmetry enforced. 
\par In practice, we want to be able to rigorously control the difference $u_0 - w_0$, where $u_0$ is given in \eqref{def : construction of u0} and $w_0$ is given by \eqref{def : w0 for dihedral}. Using \eqref{def : w0 for dihedral}, quantifying $u_0 - w_0$ amounts to quantifying the difference between $u_0$ and its rotations. For this purpose, we introduce the following result.
\begin{lemma}\label{lem : w_minus_rotation}
Let $W \in \ell^2_{D_{j_0}}$ and define $w \bydef \gamma^{\dagger}_{D_{j_0}}(W)$. Let $\theta \in [0,2\pi)$, then
\begin{align}
    \|w - R_{\theta} w\|_{2} \leq \phi(w,\theta)
\end{align}
where
\begin{align}
    &\phi(w,\theta) \bydef  \sqrt{2|\om| \|W\|_{2}^2 - 2\int_{\om \cap R_{\theta} \om}  {w(y)} w(R_{\theta} y) dy}.
\end{align}
\end{lemma}
\begin{proof}
The proof can be found in Appendix \ref{apen : w minus rotation proof}.
\end{proof}
With Lemma \ref{lem : w_minus_rotation} available to us, we provide an explicit formula for the $\mathcal{Y}_s$ bound. 
\begin{lemma}\label{lem : Ys Bound}
Let $\mathbb{L}_0 \bydef \Delta - I_d$ and let $l_0$ be its symbol (Fourier transform). Let $\mathcal{Y}_{s,1},\mathcal{Y}_{s,2} > 0$ be such that
\small{\begin{align}
    &\mathcal{Y}_{s,1} \bydef \left\|\frac{1}{l_0}\right\|_{2}\frac{|\nu_1|}{2j^2 }\phi\left(u_0,R_{\frac{2\pi}{j}}\right)\phi\left(\mathbb{L}_0u_0,R_{\frac{2\pi}{j}}\right)\sum_{k,p = 0}^{j-1} \mathcal{R}_{p-k,j}^2  \\
    &\mathcal{Y}_{s,2} \bydef \left\|\frac{1}{l_0}\right\|_{2}\frac{|\nu_2|}{j^3}\|U_0\|_{1} \phi\left(u_0,R_{\frac{2\pi}{j}}\right)\sum_{k,c,p = 0}^{j-1} \mathcal{R}_{c-k,j}\left(\mathcal{R}_{k-p,j}\phi\left(\mathbb{L}_0u_0,R_{\frac{2\pi}{j}}\right) + \mathcal{R}_{k-c,j}\phi\left(\mathbb{L}_0u_0,R_{\frac{2\pi }{j}}\right)\right)
\end{align}}
where $\phi$ is defined in Lemma \ref{lem : w_minus_rotation}, and
\begin{align}
    \mathcal{R}_{k,j} \bydef \sqrt{\left(\sum_{p = 0}^{k-1} \cos(\frac{2\pi p}{j})\right)^2 + \left(\sum_{p = 1}^{k-1} \sin(\frac{2\pi p}{j})\right)^2}. 
\end{align}
Then, defining
\begin{align}
    \mathcal{Y}_s \bydef \max\{1,\|B^N\|_{2}\}(\mathcal{Y}_{s,1} + \mathcal{Y}_{s,2}),
\end{align}
it follows that $\|\mathbb{A}\tilde{\mathbb{G}}(u_0)\|_{2} \leq \mathcal{Y}_s$. 
\end{lemma}
\begin{proof}
The proof can be found in Appendix \ref{apen : Y_s isolate proof}.
\end{proof}
\subsubsection{The Bound \texorpdfstring{$\mathcal{Z}_2$}{Z2}}\label{sec : Z2 HleqG H isolates}
Before stating our estimation for the $\mathcal{Z}_2$ bound, we state the two following lemmas.
\begin{lemma}\label{lem : kappa_one}
Let $u,v \in L^2$. Let $\kappa> 0$ be defined as in \eqref{def : kappa}. Then,
\begin{align}
    \|uv\|_{2} \leq \kappa \|u\|_{2} \|v\|_{l} 
\end{align}
\end{lemma}
\begin{proof}
We begin by following the same steps as Lemma 2.1 of \cite{unbounded_domain_cadiot}. 
\begin{align}
    \nonumber \|uv\|_{2} &\leq \|u\|_{2} \|v\|_{\infty} \leq \|u\|_{2}\|\hat{v}\|_{1} \leq \left\| \frac{1}{l}\right\|_{2} \|u\|_{2} \|v\|_{l}
\end{align}
where the final step follows from the proof of Lemma 2.1 in \cite{unbounded_domain_cadiot}. 
\end{proof}
\begin{lemma}\label{lem : w_0 minus u_0}
Let $\mathbb{K}$ be a linear operator invariant under the symmetry of $D_j$. Let $u_0$ be defined as in Section \ref{sec : construction of u0 gen}. Let $w_0$ be defined by \eqref{def : w0 for dihedral}. Then, it follows that
\begin{align}
    \|\mathbb{K}(w_0 - u_0)\|_{2} \leq \frac{1}{j}\phi\left(\mathbb{K}u_0,R_{\frac{2\pi}{j}}\right) \sum_{k=0}^{j-1} \mathcal{R}_{k,j}.
\end{align}
\end{lemma}
\begin{proof}
{\small\begin{align}
    \|\mathbb{K}(w_0 - u_0)\|_{2} = \left\| \frac{1}{j} \sum_{k = 0}^{j-1} R_{\frac{2\pi k}{j}} \mathbb{K}u_0 - \mathbb{K}u_0 \right\|_{2} 
    &\leq \frac{1}{j} \sum_{k = 0}^{j-1} \left\| R_{\frac{2\pi k}{j}} \mathbb{K}u_0 - \mathbb{K}u_0\right\|_{2} \leq \frac{1}{j}\phi\left(\mathbb{K}u_0,R_{\frac{2\pi}{j}}\right) \sum_{k=0}^{j-1} \mathcal{R}_{k,j}.\label{eq : w_0 minus u_0}
\end{align}}
This concludes the proof.
\end{proof}
The result of Lemma \ref{lem : w_0 minus u_0} will be useful in the technical analysis of this paper. It is required to compute both the $\mathcal{Z}_2$ and $\mathcal{Z}_s$ bounds not only when $\mathcal{H}$ isolates the solution, but also in the case where it doesn't. We now state the result for the $\mathcal{Z}_2$ bound.
\begin{lemma}\label{lem : bound Z_2 s}
Let $\kappa>0$ be the constant defined in \eqref{def : kappa}. Moreover, let $r>0$ and let $\mathcal{Z}_2(r) >0$ be such that 
\begin{equation}\label{def : Z2 in lemma}
    \mathcal{Z}_2(r) \geq  \mathcal{Z}_{2,1}(r) + \mathcal{Z}_{2,2}(r) 
\end{equation}
where
\begin{align}
    &\mathcal{Z}_{2,1}(r) \bydef 3|\nu_2|\frac{\kappa^2}{\mu} \max \left\{1,\|B^N\|_{2}\right\} r +  \frac{\kappa}{\mu} \max\left\{ 2|\nu_1|, ~\left(\|\mathbb{W}(B^N)^{\star}\|_{2}^2+\|W\|_1^2\right)^{\frac{1}{2}}\right\} \\ 
    &\mathcal{Z}_{2,2}(r) \bydef \frac{6\kappa^2 |\nu_2|}{j} \max\{1,\|B^N\|_{2}\} \phi\left(u_0, R_{\frac{2\pi }{j}}\right)\sum_{k = 0}^{j-1} \mathcal{R}_{k,j}
\end{align}
 where $W \bydef (w_k)_{k \in J_{\mathrm{red}}(D_{j_0})} = (2\nu_1\delta_{k}+6\nu_2u_k)_{k \in J_{\mathrm{red}}(D_{j_0})}$ and $\delta_k$ is the Kronecker symbol. Moreover, $\mathbb{W}$ is the discrete convolution operator associated to $W$.
Then $ \|\mathbb{A}\left({D}\mathbb{F}(s) - D\mathbb{F}(w_0)\right)\|_l \leq \mathcal{Z}_2(r)r$  for all $s \in \overline{B_r(w_0)} \subset H^l_{D_j}.$
\end{lemma}

\begin{proof}
The proof can be found in Appendix \ref{apen : Z2 isolate}.
\end{proof}
\subsubsection{The Bound \texorpdfstring{$\mathcal{Z}_s$}{Zs}}\label{sec : Zs isolate}
We already have everything we need to compute the $\mathcal{Z}_s$ bound. We state its result as a lemma.
\begin{lemma}\label{lem : Zs bound}
Let $\mathcal{Z}_s > 0$ be such that
\begin{align}
    \mathcal{Z}_s \geq 2\kappa \max\{1,\|B^N\|_{2}\} (|\nu_1| + 3|\nu_2| \|U_0\|_{1})\frac{1}{j} \phi\left(u_0, R_{\frac{2\pi }{j}}\right)\sum_{k = 0}^{j-1} \mathcal{R}_{k,j}
\end{align}
where $\phi$ is defined in Lemma \ref{lem : Ys Bound}.
Then, $\|\mathbb{A}(D\mathbb{G}(w_0) - D\mathbb{G}(u_0))\|_{l} \leq \mathcal{Z}_s$.
\end{lemma}
\begin{proof}
The proof can be found in Appendix \ref{apen : Zs isolate}.
\end{proof}
With the bounds $\mathcal{Y}_0$ and $\mathcal{Z}_1$ computed thanks to Corollary \ref{th: radii polynomial} along with $\mathcal{Y}_s, \mathcal{Z}_2,$ and $\mathcal{Z}_s$ due to lemmas \ref{lem : Ys Bound}, \ref{lem : bound Z_2 s}, and \ref{lem : Zs bound} respectively, we can now perform proofs of localized patterns in the case that $\mathcal{H}$ is a space subgroup of $\mathcal{G}$ that isolates the solution. 
\subsection{When \texorpdfstring{$j$}{j} is a prime number other than \texorpdfstring{$2$}{2}}\label{sec : H nonisolate}
In this section, $\mathcal{H}$ does not isolate the solution meaning that $\mathcal{H}$ is $\mathbb{Z}_2 \times \mathbb{Z}_1$ according to Lemma \ref{lem : H classification}. Providing bounds which satisfy Theorem \ref{th: radii polynomial augmented} will provide us with the solution and symmetry we desire. The bound $\mathcal{Y}_0$ can easily be computed thanks to the theory developed in Section 5 of \cite{unbounded_domain_cadiot} and Lemma \ref{th: radii polynomial comp}. Now, recall the definition of the operator $\mathbb{B} : \mathbb{R} \times L^2_{\mathbb{Z}_2 \times \mathbb{Z}_1} \to \mathbb{R} \times L^2_{\mathbb{Z}_2 \times \mathbb{Z}_1}$ and $B^N : \mathbb{R} \times \ell^2_{\mathbb{Z}_2 \times \mathbb{Z}_1} \to \mathbb{R} \times \ell^2_{\mathbb{Z}_2 \times \mathbb{Z}_1}$ in \eqref{def : the operator B}. Decomposing along the tensor products, $\mathbb{B}$ and $B^N$ can be written as follows
\begin{align}
    \mathbb{B} \bydef \begin{bmatrix}
        \mathbb{b}_{11} & \mathbb{b}_{12}^* \\
        \mathbb{b}_{21} & \mathbb{b}_{22}
    \end{bmatrix},~ B^N \bydef \begin{bmatrix}
        b_{11}^N & (b_{12}^N)^* \\
        b_{21}^N & b_{22}^N
    \end{bmatrix}
\end{align}

where, for instance, $\mathbb{b}_{22} : L^2_{\mathbb{Z}_2 \times \mathbb{Z}_1} \to L^2_{\mathbb{Z}_2 \times \mathbb{Z}_1}$ and $b_{12}^N \in \pi^N\ell^2_{\mathbb{Z}_2 \times \mathbb{Z}_1} \to \R$. In fact, we have that $\mathbb{b}_{22} = \Gamma^\dagger_{\mathbb{Z}_2 \times \mathbb{Z}_1}(b_{22}^N + \pi_N)$. Using Lemma \ref{lem : gamma and Gamma properties}, this implies that 
\[
\|\mathbb{b}_{22}\|_2 = \max\{1,\|b_{22}^N\|_{2}\}.
\]
Similarly, we have that $\mathbb{b}_{12} = \gamma^\dagger_{\mathbb{Z}_2 \times \mathbb{Z}_1}(b_{12}^N)$ which implies that $\|\mathbb{b}_{12}\|_2 = \sqrt{|\om|} \|b_{12}^N\|_2$. 
We now introduce the following quantity:
\begin{align}
    \mathcal{C}_{B} \bydef (\|b_{12}^N\|_{2}^2 + \max\{1,\|b_{22}^N\|_{2}\}^2)^{\frac{1}{2}}.\label{def : C_B}
\end{align}
Observe that $C_{B} \geq \left\| \begin{bmatrix}
    0 & \mathbb{b}_{12}^* \\
    0 & \mathbb{b}_{22}
\end{bmatrix}\right\|_{2}$. Indeed, using Lemma 5.1 of \cite{unbounded_domain_cadiot}, we get
\begin{align}
    \left\| \begin{bmatrix}
    0 & \mathbb{b}_{12}^* \\
    0 & \mathbb{b}_{22}
\end{bmatrix}\right\|_{H_2}=  \left\| \begin{bmatrix}
    0 & ({b}_{12}^N)^* \\
    0 & {b}_{22}^N + \pi_N
\end{bmatrix}\right\|_{X_2} \leq 
(\|b_{12}^N\|_{2}^2 + \max\{1,\|b_{22}^N\|_{2}\}^2)^{\frac{1}{2}}.
\end{align}
Using $C_{B}$, we are able to compute the bound $\mathcal{Y}_s$.
\begin{lemma}
Let $\mathcal{Y}_s > 0$ be defined as 
\begin{align}
    \mathcal{Y}_s \bydef C_{B} (\mathcal{Y}_{s,1} + \mathcal{Y}_{s,2})
\end{align}
where $\mathcal{Y}_{s,1}$ and $\mathcal{Y}_{s,2}$ are defined as in Lemma \ref{lem : Ys Bound}. Then, it follows that $\left\|\mathbb{A} \begin{bmatrix}
    0 \\ \tilde{\mathbb{G}}(u_0)
\end{bmatrix}\right\|_{H_1} \leq \mathcal{Y}_s$.
\end{lemma}
\begin{proof}
Observe that
\begin{align}
    \left\|\mathbb{A} \begin{bmatrix}
    0 \\ \tilde{\mathbb{G}}(u_0)
\end{bmatrix}\right\|_{H_1} = \left\|\mathbb{B} \begin{bmatrix}
    0 \\ \tilde{\mathbb{G}}(u_0)
\end{bmatrix}\right\|_{H_2} = \left\|\begin{bmatrix}
0 & \mathbb{b}_{12}^* \\
0 & \mathbb{b}_{22}\end{bmatrix}\begin{bmatrix}
    0 \\ \tilde{\mathbb{G}}(u_0)
\end{bmatrix}\right\|_{H_2} \leq C_B \|\tilde{\mathbb{G}}(u_0)\|_{2} \leq C_B (\mathcal{Y}_{s,1} + \mathcal{Y}_{s,2}).
\end{align}
\end{proof}
Let us now compute the bounds $\mathcal{Z}_1, \mathcal{Z}_s,$ and $\mathcal{Z}_2$ with the extra equation present. 
\subsubsection{The Bound \texorpdfstring{$\mathcal{Z}_1$}{Z1}}\label{sec : Z1 nonisolate}
In this section, we study the bound $\mathcal{Z}_1$. We begin by introducing some notation. We define
\begin{align}
    &U_0^N = \pi^N U_0\text{,}~~u_0^N = \gamma_{\mathbb{Z}_2 \times \mathbb{Z}_1}(U_0^N)\text{,}~~\partial_{x_2} U_0 = \gamma_{\mathbb{Z}_2 \times \mathbb{Z}_1}(\partial_{x_2} u_0)\text{, and}~~\partial_{x_2}U_0^N = \gamma_{\mathbb{Z}_2 \times \mathbb{Z}_1}(\partial_{x_2} u_0^N).
\end{align}
Next, let $\mathbb{M}(0,u_0),M(0,U_0)$ be defined as
\begin{align}
    \mathbb{M}(0,u_0) \bydef \begin{bmatrix}
        -1 & \left(\frac{\partial_{x_2}u_0^N}{\rho}\right)^* \mathbb{L} \\
        \frac{\partial_{x_2} u_0^N}{\rho} & \mathbb{v}_0^N
    \end{bmatrix}~~\text{and}~~M(0,U_0) \bydef \begin{bmatrix}
        -1 & \left(\frac{\partial_{x_2}U_0^N}{\rho}\right)^* L \\
        \frac{\partial_{x_2} U_0^N}{\rho} & \mathbb{V}_0^N
    \end{bmatrix}
\end{align}
where $\mathbb{v}_0^N$ and $\mathbb{V}_0^N$ are the multiplication operators for $v_0$ and $V_0$ defined in \eqref{def : v_0} respectively.
 These notations allow to decompose both $D\mathbb{G}(0,u_0)$ and $DG(0,U_0)$ with a truncation of size $N$. 
Also, observe that for $x \in H_1, 
     \|x\|_{H_1} = \left\|\begin{bmatrix}
    1 & 0 \\
    0 & \mathbb{L}
\end{bmatrix} x\right\|_{H_2}.$
\begin{lemma}\label{lem : Z_1_extra}
Let $Z_1,\mathcal{Z}_u> 0$ be such that   \begin{align}
    &Z_1 \geq \left\|I_d - B\left(I_d + M(0,U_0)\begin{bmatrix}
    1 & 0 \\
    0 & L^{-1}
\end{bmatrix}\right)\right\|_{X_2},~\mathcal{Z}_u \geq \|(\Gamma_{\mathbb{Z}_2 \times \mathbb{Z}_1}^\dagger(L^{-1}) - \mathbb{L}^{-1})\mathbb{v}_0^N\|_{2}.
\end{align}
Then, defining 
{\small\begin{align}
    \mathcal{Z}_1 \bydef Z_1 + C_B \mathcal{Z}_u + \|\mathbb{B}\|_{H_2}\varphi\left(0,\frac{\sqrt{|\om|}}{\rho}\|\partial_{x_2} U_0^N - \partial_{x_2} U_0\|_{2}, \frac{\sqrt{|\om|}}{\rho}\|\partial_{x_2} U_0^N - \partial_{x_2} U_0\|_{2}, \frac{1}{\mu} \|V_0^N - V_0\|_{1}\right)
\end{align}}
we have $\|I_d - \mathbb{A}\mathbb{Q}(0,u_0)\|_{H_1} \leq \mathcal{Z}_1$.
\end{lemma}
\begin{proof}
The proof can be found in Appendix \ref{apen : Z1 derivation proof}.
\end{proof}
Now, we observe  that the bound $Z_1$ can be computed thanks to matrix norms. We provide such a result in the next lemma.
\begin{lemma}\label{lem : Z1_extra}
Let $M^N \bydef \pi^N + M(0,U_0)\begin{bmatrix}
    1 & 0 \\
    0 & L^{-1}
\end{bmatrix}$ and $M \bydef \pi_N + M^N$. Let $Z_1 > 0$ be such that   
\begin{align}
    Z_1 \geq \varphi(Z_{1,1},Z_{1,2},Z_{1,3},Z_{1,4})
\end{align}
where $\varphi$ is defined in Lemma 4.1 of \cite{gs_cadiot_blanco} and
{\small\begin{align}
    &Z_{1,1} \bydef \sqrt{\|(\pi^N - B^NM^N) (\pi^N - (M^N)^{\star} (B^N)^{\star})\|_{2}},~Z_{1,3} \bydef \sqrt{\left\|\pi^NL^{-1}\mathbb{V}_0^N\pi_N\mathbb{V}_0^N L^{-1} \pi^N\right\|_{2}}\\
    &Z_{1,2} \bydef \max_{n \in \mathbb{N}_0 \times \mathbb{Z} \setminus I^N} \frac{1}{|l(\tilde{n})|} \sqrt{\| B^NM(0,U_0)\pi_N M(0,U_0)^{\star}(B^N)^{\star}\|_{X_2}},~Z_{1,4} \bydef \max_{n \in \mathbb{N}_0 \times \mathbb{Z} \setminus I^N} \frac{1}{|l(\tilde{n})|} \|V_0^N\|_{1}.
\end{align}}
Then, $\left\|I - B\left(I_d + M(0,U_0)\begin{bmatrix}
    1 & 0 \\
    0 & L^{-1}
\end{bmatrix}\right)\right\|_{X_2} \leq Z_1$.
\end{lemma}
\begin{proof}
The proof can be found in Appendix \ref{apen : Z1 periodic proof}.
\end{proof}
\subsubsection{The Bound \texorpdfstring{$\mathcal{Z}_s$}{Zs}}
Before providing the result for $\mathcal{Z}_s$, we state the following lemma.
% \begin{lemma}\label{lem : kappa partial}
% Let $l_{\partial}(\xi) \bydef -2\pi i\xi_2$. In particular, $l_{\partial}$ is the Fourier transform of $\mathbb{L}_{\partial} \bydef \partial_{x_2}$. Let $\mathbb{L}_1 \bydef (I_d + \Delta)^2 + I_d$ with $l_1(\xi) = (1 - |2\pi\xi|^2)^2 + 1$ its Fourier transform. Let $\kappa_{\partial} \geq \left\|\frac{l_{\partial}}{l_1}\right\|_{2}$ be defined as
% \begin{align}
%     \kappa_{\partial} \bydef \frac{1}{64} \left(3 + \frac{4}{\pi}\right).\label{def : kappa_partial}
% \end{align}
% Then, for any $w$, we obtain
% \begin{align}
%     \|\partial_{x_2} w\|_{2} \leq \kappa_{\partial} \|w\|_{l_1}.
% \end{align}
% \end{lemma}
% \begin{proof}
% To begin, we introduce $l_1$.
% \begin{align}
%     \|\partial_{x_2} w\|_{2} = \|l_{\partial} \widehat{w}\|_{2} = \left\|l_{\partial} \frac{l_1}{l_1} \widehat{w}\right\|_{2} \leq \left\|\frac{l_{\partial}}{l_1}\right\|_{2} \|l_1 \widehat{w}\|_{2} \leq \kappa_{\partial} \|\mathbb{L}_1 w\|_{2} = \kappa_{\partial} \|w\|_{l_1}
% \end{align}
% as desired.
% Let us now compute $\kappa_{\partial}$.

% Observe that
% \begin{align}
%     \left\|\frac{l_{\partial}}{l_1}\right\|_{2}^2 = \int_{\mathbb{R}^2} \left(\frac{l_{\partial}(\xi)}{l_1(\xi)}\right)^2 d\xi &= \int_{-\infty}^{\infty} \int_{-\infty}^{\infty} \frac{4\pi^2 \xi_2^2}{((1 - |2\pi\xi|^2)^2 + 1)^2} d\xi_2 d\xi_1 \\
%     &= \frac{1}{64}\left(3 + \frac{4}{\pi}\right).
% \end{align}
% Hence, we can choose $\kappa_{\partial}$ as in \eqref{def : kappa_partial} to obtain a bound.
% \end{proof}
\begin{lemma}\label{lem : kappa partial}
Let $l_{\partial}(\xi) \bydef -2\pi i\xi_2$. In particular, $l_{\partial}$ is the Fourier transform of $\mathbb{L}_{\partial} \bydef \partial_{x_2}$. Let $\mathbb{L}_1 \bydef (I_d + \Delta)^2 + I_d$ and let $l_1$ be its symbol (Fourier transform) $l_1$. Let $\kappa_{\partial} \geq \left\|\frac{l_{\partial}}{l_1}\right\|_{2}$.
Then, 
\begin{align}
    \|\partial_{x_2} (w_0 - u_0)\|_{2} \leq \frac{\kappa_{\partial}}{j} \phi\left(\mathbb{L}_1 u_0, R_{\frac{2\pi}{j}}\right) \sum_{k = 0}^{j-1} \mathcal{R}_{k,j} .
\end{align}
\end{lemma}
\begin{proof}
To begin, we use Plancherel's identity and introduce $l_1$.
{\small\begin{align}
    \|\partial_{x_2} (w_0 - u_0)\|_{2} = \|l_{\partial} (\widehat{w}_0 - \widehat{u}_0) \|_{2}  \leq \left\|\frac{l_{\partial}}{l_1}\right\|_{2} \|l_1 (\widehat{w}_0 - \widehat{u}_0)\|_{2} 
    \leq \kappa_{\partial} \|w_0 - u_0\|_{l_1} 
    &\leq \frac{\kappa_{\partial}}{j} \phi\left(\mathbb{L}_1 u_0, R_{\frac{2\pi}{j}}\right) \sum_{k = 0}^{j-1} \mathcal{R}_{k,j}
\end{align}}
where the last step followed from Lemma \ref{lem : w_0 minus u_0} with $\mathbb{K} = \mathbb{L}_1$. This concludes the proof.
\end{proof}
With Lemma \ref{lem : kappa partial} available to us, we are ready to compute the $\mathcal{Z}_s$ bound. We summarize its result in a lemma.
\begin{lemma}\label{lem : Zs_extra}
Let $\kappa_{\partial}$ be computed using Lemma \ref{lem : kappa partial}. Let $\mathcal{Z}_s \geq 0$ be such that
\begin{align}
    \mathcal{Z}_s \bydef \|\mathbb{B}\|_{H_2} \varphi(0,\mathcal{Z}_{s,1},\mathcal{Z}_{s,1},\mathcal{Z}_{s,2})
\end{align}
where
\begin{align}
    \mathcal{Z}_{s,1} \bydef \frac{\kappa_{\partial}}{j\rho} \phi\left(\mathbb{L}_1u_0, R_{\frac{2\pi }{j}}\right)\sum_{k = 0}^{j-1} \mathcal{R}_{k,j},~~\mathcal{Z}_{s,2} \bydef 2\kappa (|\nu_1| + 3|\nu_2| \|U_0\|_{1}) \frac{1}{j} \phi\left(u_0, R_{\frac{2\pi }{m}}\right)\sum_{k = 0}^{j-1} \mathcal{R}_{k,j}.
\end{align}
Then, $\|\mathbb{A}(\mathbb{Q}(0,u_0) - D\mathbb{F}(0,w_0))\|_{H_1} \leq \mathcal{Z}_s$.
\end{lemma}
\begin{proof}
The proof can be found in Appendix \ref{apen : Zs nonisolate}.
\end{proof}
\begin{remark}
The choices of $l_0$ and $l_1$ were made for our specific case of Swift Hohenberg. It is possible that there are more optimal choices of these operators that will make the bounds smaller. It would be of interest to find such operators in order to minimize the quantities we compute.
\end{remark}
% \begin{remark}
% We could have computed the $\mathcal{Z}_s$ bound directy without using $\kappa_{\partial}$. This would require computing estimates of the form $\|\partial_{x_2}(w_0 - u_0)\|_{2}$, which is possible using Lemma \ref{lem : w_minus_rotation}. We chose this approach primarily due to computation efficiency. Indeed, we have already computed the terms $\|w_0 - u_0\|_{2}$ and $\|w_0 - u_0\|_{l_0}$ for the $\mathcal{Y}_s$ bound. Therefore, if we choose $\mathbb{L}_1 = \mathbb{L}_0$, we can simply re-use these lengthy computations at the cost of introducing $\kappa_{\partial}$. Furthermore, the computation of $\|\frac{1}{l_1}\|_{2}$ is already completed as $\mathbb{L}_1 = \mathbb{L}$ in the case $\mu = 1$.
% \end{remark}
\subsubsection{The Bound \texorpdfstring{$\mathcal{Z}_2$}{Z2}}
We are now ready to compute the $\mathcal{Z}_2$ bound in the case where we require an extra equation. We state the result as a lemma.
\begin{lemma}\label{lem : Z_2_extra}
Let $\kappa, W, \mathbb{W},\mathcal{Z}_{2,1}(r),$ and $\mathcal{Z}_{2,2}(r)$ be defined as in Lemma \ref{lem : bound Z_2 s}. Let $\mathcal{Z}_{2,3}(r): [0,\infty) \to (0,\infty), \mathcal{Z}_{2,4}(r) : [0,\infty) \to (0,\infty)$ be non-negative functions such that
\begin{align}
    \mathcal{Z}_{2,3}(r) \bydef  \frac{\kappa}{\mu} (\|\mathbb{W}\|_{2}^2 + \|W\|_{1}^2)^{\frac{1}{2}} r ,~~\mathcal{Z}_{2,4}(r) \bydef  \frac{6|\nu_2|  \kappa^2}{j} \phi\left(u_0, R_{\frac{2\pi }{j}}\right)\sum_{k=0}^{j-1} \mathcal{R}_{k,j}.
\end{align}
Let $\mathcal{Z}_2(r) : [0,\infty) \to (0,\infty)$ be a non-negative function such that
\begin{align}
    \mathcal{Z}_2(r) \bydef \mathcal{Z}_{2,1}(r) + \mathcal{Z}_{2,2}(r) + \max(1,\|b_{12}^N\|_{2})(\mathcal{Z}_{2,3}(r)+\mathcal{Z}_{2,4}(r)).
\end{align}  
Then, $\|\mathbb{A}(D\mathbb{F}(h_1,h_2) - D\mathbb{F}(0,w_0))\|_{H_1} \leq \mathcal{Z}_2(r)r$.
\end{lemma}
\begin{proof}
The proof can be found in Appendix \ref{apen : Z2 nonisolate}.
\end{proof}
\section{Constructive proofs of existence of dihedral localized patterns}\label{sec : proof of solutions}
In this section, we provide all of our results for existence of dihedral localized solutions of the 2D Swift-Hohenberg PDE \eqref{eq : sh localized 2D}. Note that we rely on the ability to construct sequences with $\mathbb{Z}_2 \times \mathbb{Z}_1, D_2,$ and $D_4$ symmetry. The former two can be constructed directly using RadiiPolynomial.jl \cite{julia_olivier}. For $D_4$, we use D4Fourier.jl \cite{dominic_D_4_julia}. This was previously used to prove the $D_4$ symmetry in \cite{gs_cadiot_blanco}.
\subsection{Proofs when \texorpdfstring{$j = 2,4$}{j24}}
In this section, our contribution is to prove the symmetry of the $D_4$ solution proven in \cite{sh_cadiot}. Furthermore, we provide a proof of existence and symmetry for a $D_2$ solution in \eqref{eq : swift_hohenberg}. We now present the following theorems.
\begin{theorem}[\bf The $D_2$ pattern]
Let $\mu = 0.24, \nu_1 = -1.6, \nu_2 = 1$. Moreover, let $r_0 \bydef 8 \times 10^{-5}$. Then there exists a unique solution $\tilde{u}$ to \eqref{eq : swift_hohenberg} in $\overline{B_{r_0}(u_{D_2})} \subset H^l_{D_2}$ and we have that $\|\tilde{u}-u_{D_2}\|_{l} \leq r_0$. That is, $\tilde{u}$ is at most $r_0$ away from the approximation shown in Figure \ref{fig : D2 pattern}.
\end{theorem}
\begin{proof}
Choose $N_0 = 150, N = 100, d = 85$. Then, we perform the full construction described in Section \ref{sec : construction of u0 gen} to build $u_0 = \gamma^\dagger(U_0)$. Then, we define $u_{D_2} \bydef u_0$. Next, we construct $B^N$ using the approach described in Section \ref{sec : operator A first}, compute $\kappa$ using \eqref{def : kappa}, and use the constants defined in Section 3.5.1 of \cite{sh_cadiot}. We find
\begin{align}
    \|B^N\|_{2} \leq 298.1,~~\kappa \bydef 2.56,~~C_0 \bydef 2.825.
 \end{align}
Finally, using \cite{julia_cadiot_blanco_symmetry}, we choose $r_0 \bydef 8 \times 10^{-5}$ and define
\begin{align}
    \mathcal{Y}_0 \bydef 1.63 \times 10^{-5} \text{,}~\mathcal{Z}_{2}(r_0) \bydef 4551   \text{,}~\mathcal{Z}_1 \bydef 0.097.
    \end{align}
and prove that these values satisfy Theorem \ref{th: radii polynomial}. 
\end{proof}
%  \begin{figure}[H]
% \centering
%  \begin{minipage}[H]{0.5\linewidth}
%   \centering\epsfig{figure=sh_D8.eps,width=\linewidth}
%  \end{minipage} 
%  \caption{Approximation of a $D_2$ pattern in the Swift Hohenberg equation.}
%  \label{fig : D2 pattern}
%  \end{figure}
\begin{theorem}[\bf The $D_4$ pattern]
Let $\mu = 0.28, \nu_1 = -1.6, \nu_2 = 1$. Moreover, let $r_0 \bydef 3.1 \times 10^{-5}$. Then there exists a unique solution $\tilde{u}$ to \eqref{eq : swift_hohenberg} in $\overline{B_{r_0}(u_{D_4})} \subset H^l_{D_4}$ and we have that $\|\tilde{u}-u_{D_4}\|_{l} \leq r_0$. That is, $\tilde{u}$ is at most $r_0$ away from the approximation shown in Figure \ref{fig : D4 pattern}.
\end{theorem}
\begin{proof}
Choose $N_0 = 130, N = 100, d = 70$. Then, we perform the full construction described in Section \ref{sec : construction of u0 gen} to build $u_0 = \gamma^\dagger(U_0)$. Then, we define $u_{D_4} \bydef u_0$. Next, we construct $B^N$ using the approach described in Section \ref{sec : operator A first} compute $\kappa$ using \eqref{def : kappa}, and use the constants defined in Section 3.5.1 of \cite{sh_cadiot}. We find
\begin{align}
    \|B^N\|_{2} \leq 41.07,~~\kappa \bydef 2.275,~~C_0 \bydef 2.6.
 \end{align}
Finally, using \cite{julia_cadiot_blanco_symmetry}, we choose $r_0 \bydef 3.1 \times 10^{-5}$ and define
\begin{align}
    \mathcal{Y}_0 \bydef 9.59 \times 10^{-6} \text{,}~\mathcal{Z}_{2}(r_0) \bydef  462.55 \text{,}~\mathcal{Z}_1 \bydef  0.108.
    \end{align}
and prove that these values satisfy Theorem \ref{th: radii polynomial}. 
\end{proof}
%  \begin{figure}[H]
% \centering
%  \begin{minipage}[H]{0.5\linewidth}
%   \centering\epsfig{figure=sh_D8.eps,width=\linewidth}
%  \end{minipage} 
%  \caption{Approximation of a $D_4$ pattern in the Swift Hohenberg equation.}
%  \label{fig : D4 pattern}
%  \end{figure}
 \begin{figure}[t]
    \centering
    \begin{subfigure}[b]{0.24\textwidth}
        \centering
        \epsfig{figure=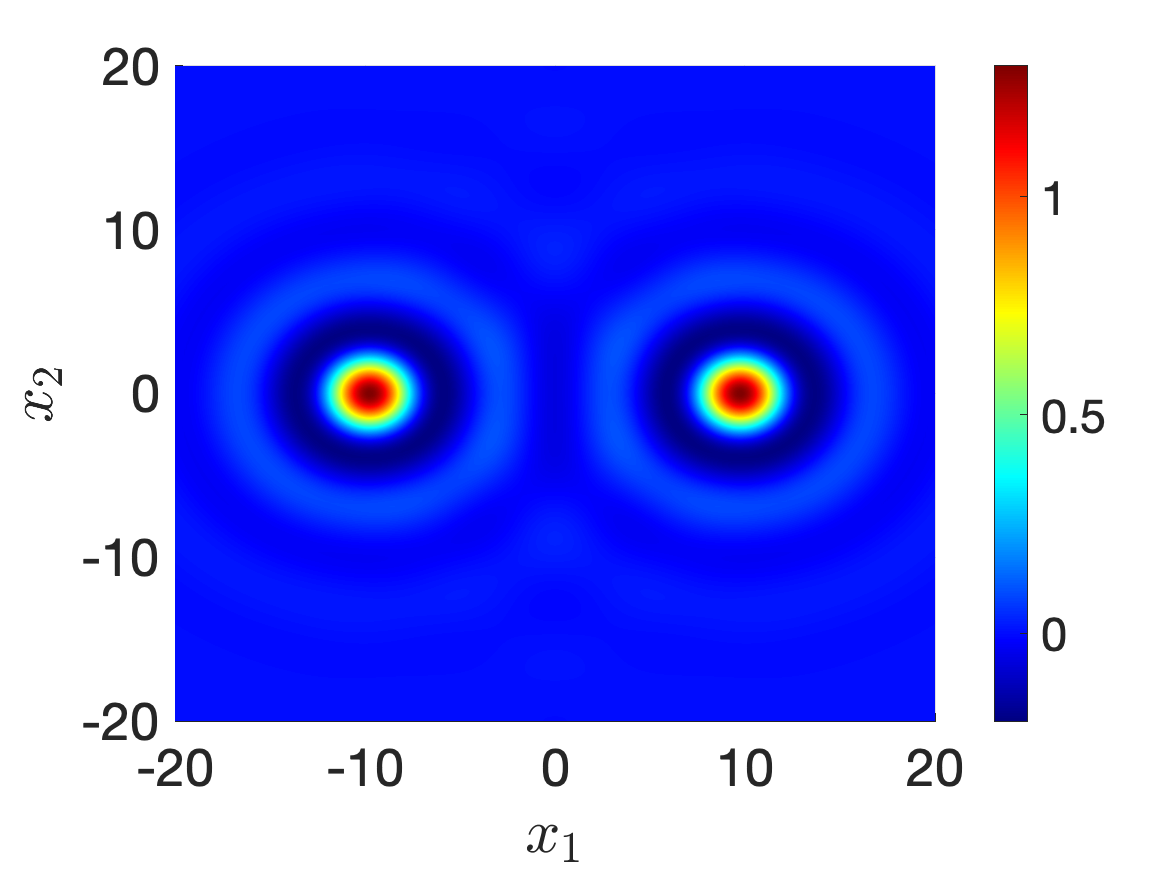, width=\textwidth}
        \caption{$D_2$ approximate solution in the Swift Hohenberg equation}\label{fig : D2 pattern}
    \end{subfigure}
    \hfill
    \begin{subfigure}[b]{0.24\textwidth}
    \centering
        \epsfig{figure=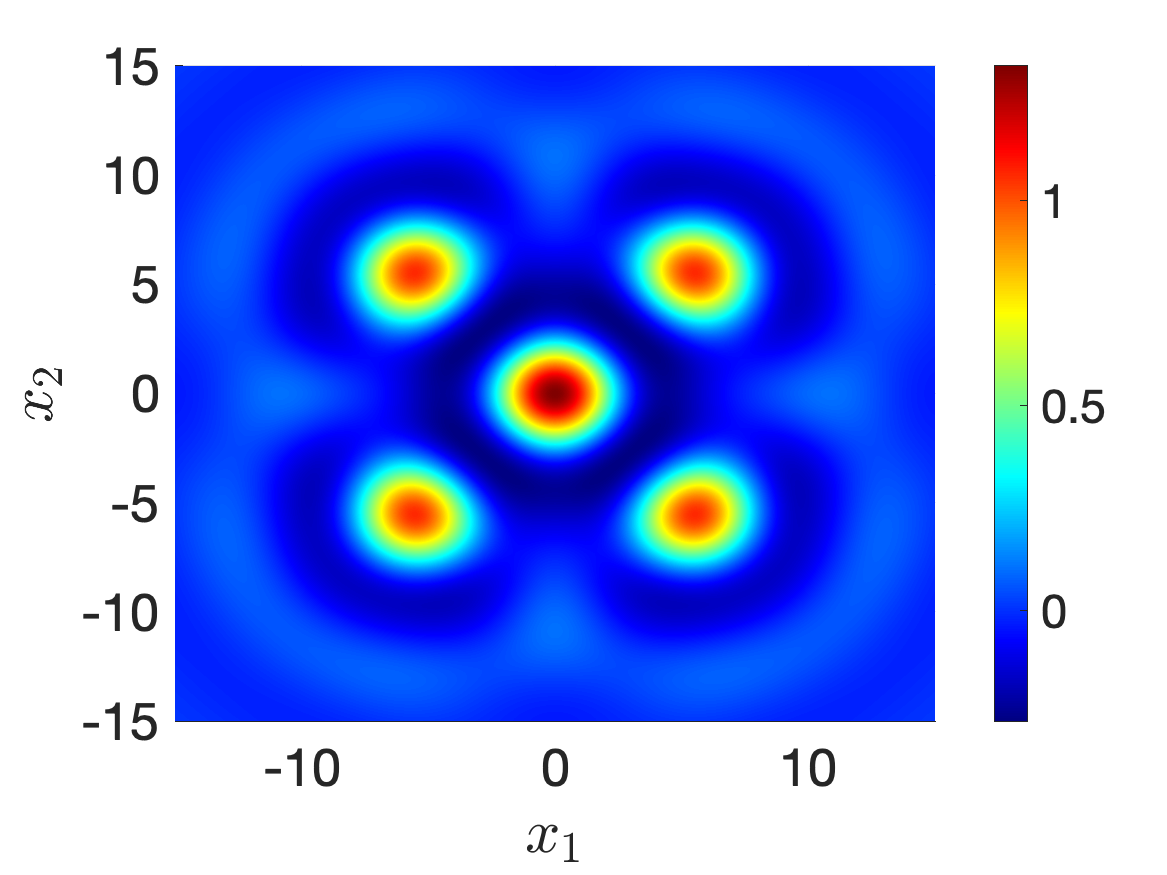, width=\textwidth}
        \caption{$D_4$ approximate solution in the Swift Hohenberg equation}\label{fig : D4 pattern}
    \end{subfigure}
    \hfill
    \begin{subfigure}[b]{0.24\textwidth}
    \centering
        \epsfig{figure=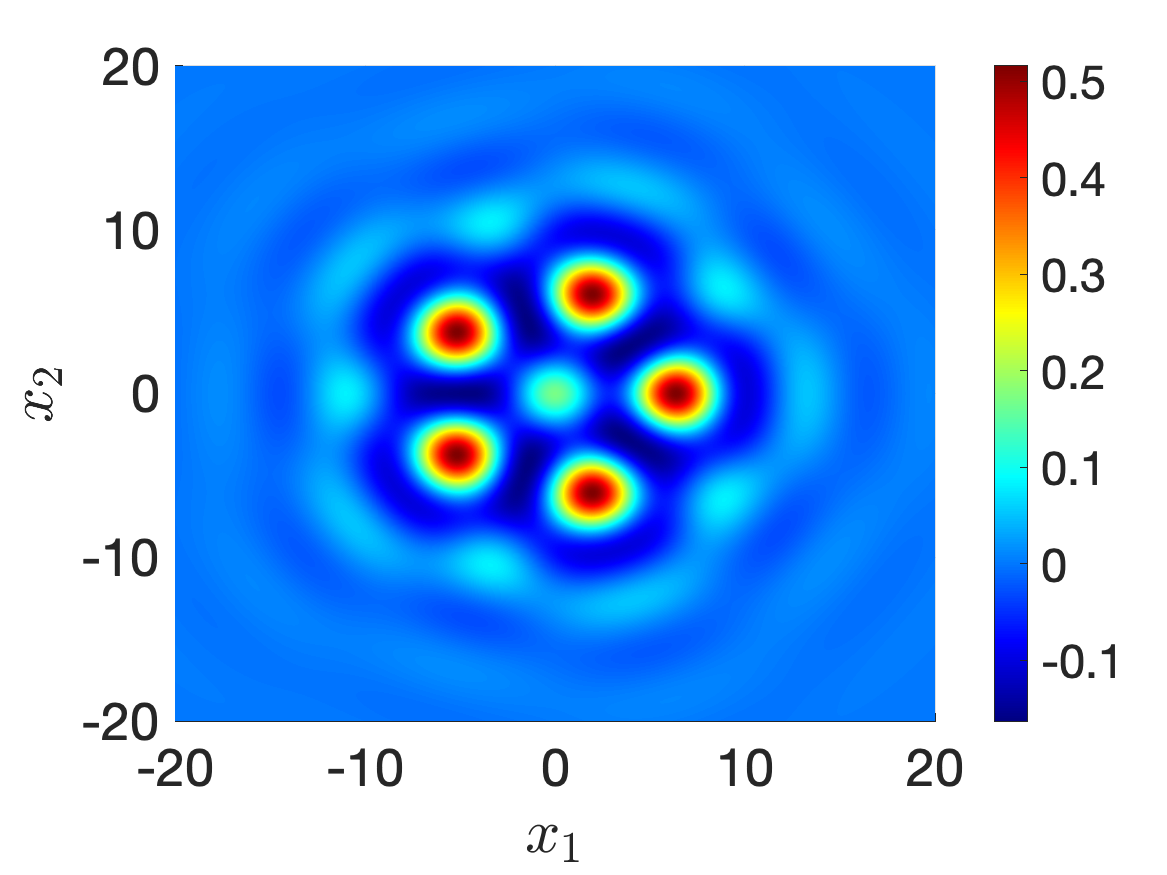, width=\textwidth}
        \caption{$D_5$ approximate solution in the Swift Hohenberg equation}\label{fig : D5 pattern}
    \end{subfigure}
    \\
    \begin{subfigure}[b]{0.24\textwidth}
    \centering
        \epsfig{figure=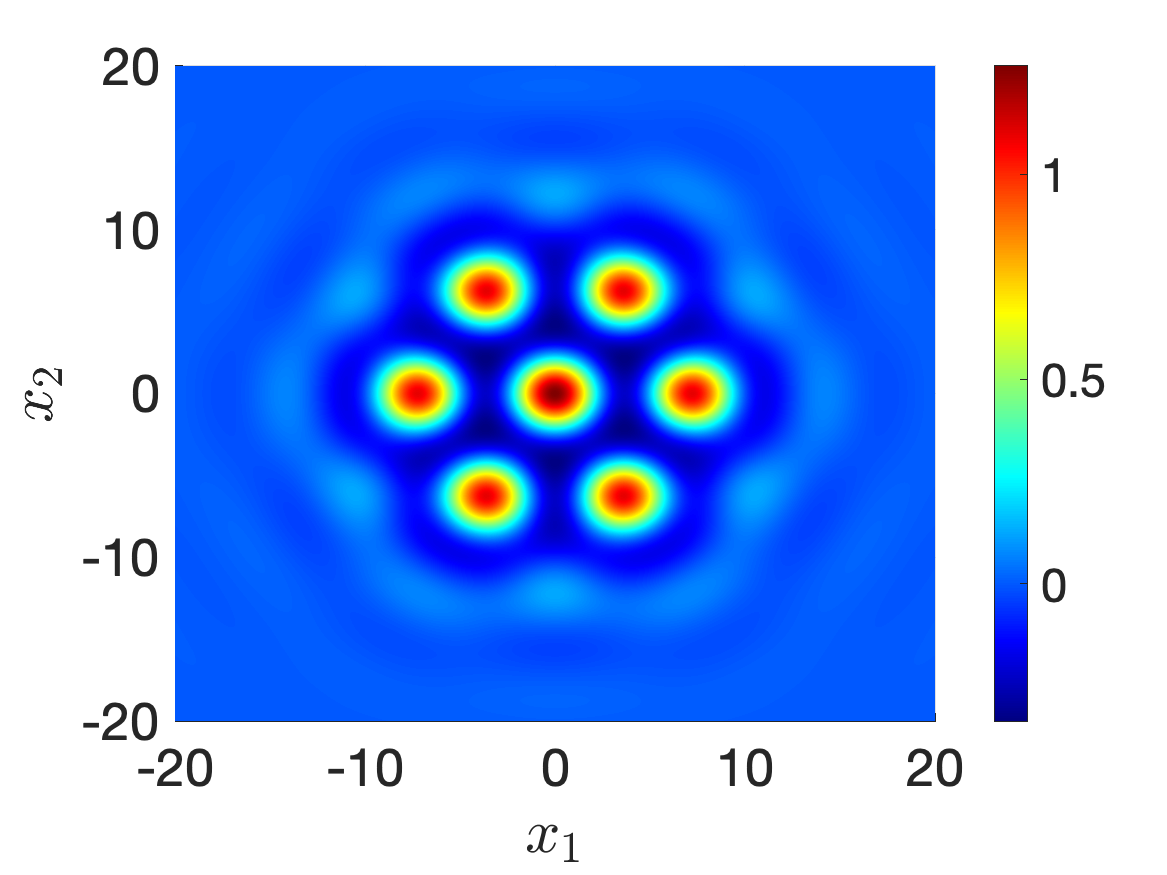, width=\textwidth}
        \caption{$D_6$ approximate solution in the Swift Hohenberg equation}\label{fig : D6 pattern}
    \end{subfigure}
    \hfill
    \begin{subfigure}[b]{0.24\textwidth}
    \centering
        \epsfig{figure=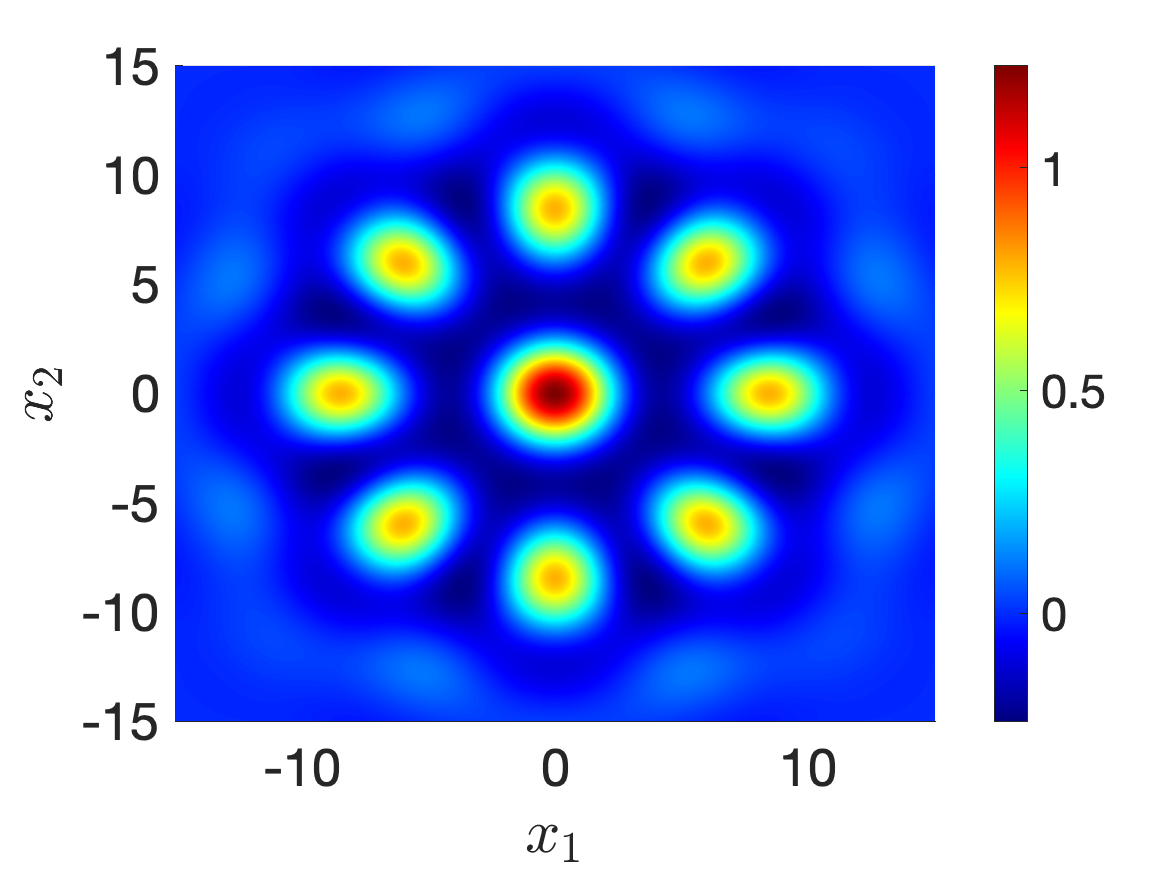, width=\textwidth}
        \caption{ $D_8$ approximate solution in the Swift Hohenberg equation}\label{fig : D8 pattern}
    \end{subfigure}
    \hfill
    \begin{subfigure}[b]{0.24\textwidth}
    \centering
        \epsfig{figure=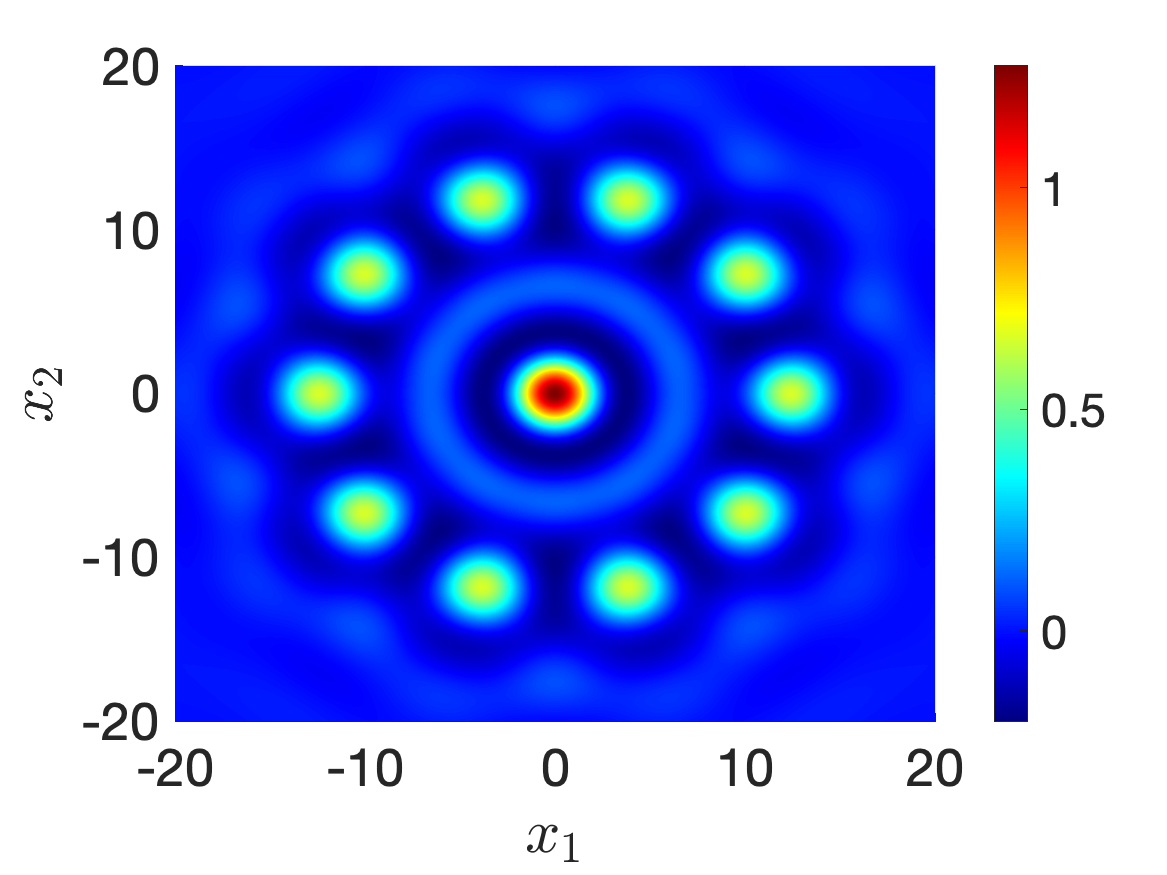, width=\textwidth}
        \caption{$D_{10}$ approximate solution in the Swift Hohenberg equation}\label{fig : D10 pattern}
    \end{subfigure}
    \hfill
    \begin{subfigure}[b]{0.24\textwidth}
    \centering
        \epsfig{figure=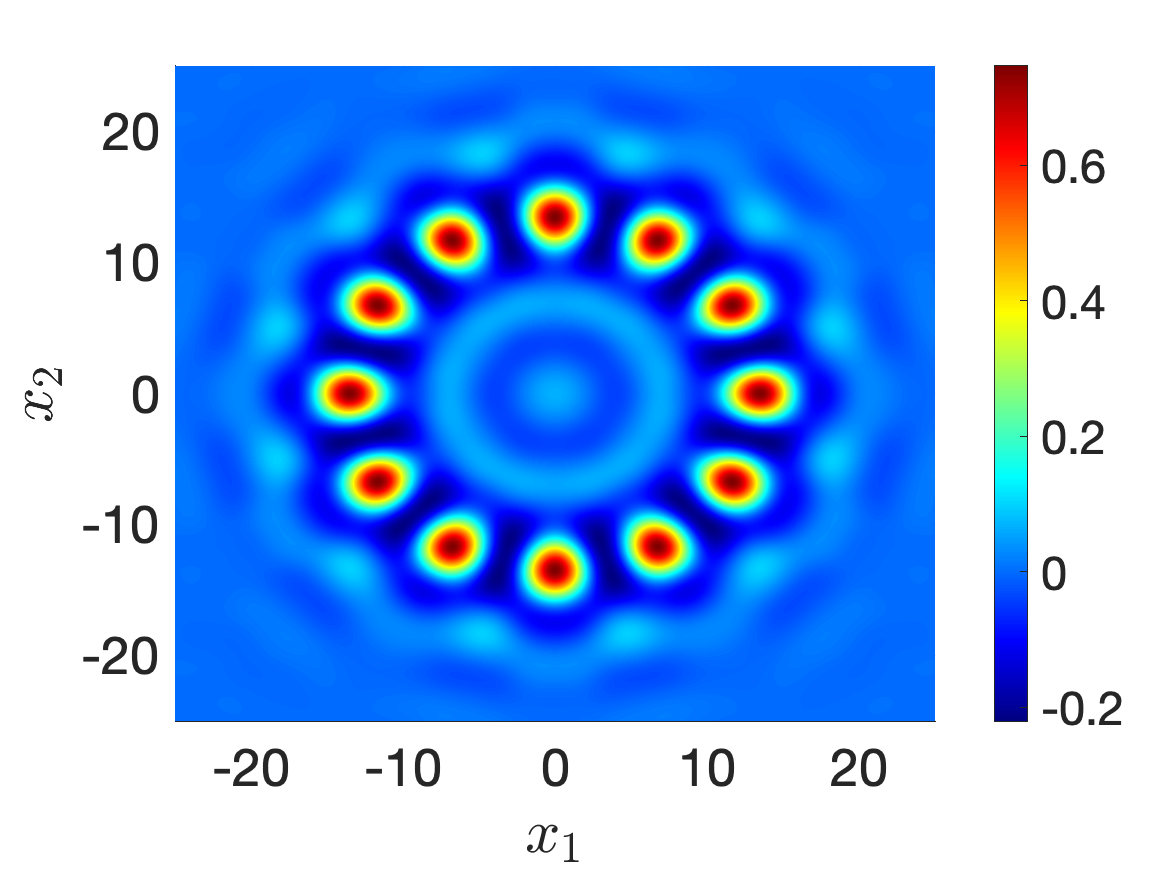, width=\textwidth}
        \caption{ $D_{12}$ approximate solution in the Swift Hohenberg equation}\label{fig : D12 pattern}
    \end{subfigure}
    \caption{Approximations of $D_j$ solutions for the Swift-Hohenberg equation }
\end{figure}
\subsection{Proofs when \texorpdfstring{$j$}{j} is not a prime number nor \texorpdfstring{$4$}{4}}
In this section, we consider the case where $\mathcal{H} = D_{j_0}, j_0 \in \{2,4\}$. We provide a proof of the $D_6$ and $D_8$ symmetry of the patterns proven in \cite{sh_cadiot}. We also provide proofs of patterns with $D_{10}$ and $D_{12}$ symmetry. These are, to the best of our knowledge, the first proofs of these patterns along with their symmetry in 2D for \eqref{eq : swift_hohenberg} without any restrictions on the value of $\mu$.
\begin{theorem}[\bf The $D_6$ pattern]
Let $\mu = 0.32, \nu_1 = -1.6, \nu_2 = 1$. Moreover, let $r_0 \bydef 8 \times 10^{-5}$. Then there exists a unique solution $\tilde{w}$ to \eqref{eq : swift_hohenberg} in $\overline{B_{r_0}(w_{D_6})} \subset H^l_{D_6}$ and we have that $\|\tilde{w}-w_{D_6}\|_{l} \leq r_0$. 
\end{theorem}
\begin{proof}
Choose $N_0 = 130, N = 70, N_1=80, d = 70$. Then, we perform the full construction described in Section \ref{sec : construction of w_0} to build $u_0 = \gamma^\dagger(U_0)$ and then $w_0$ using \eqref{v_D_m_fourier}. Then, we define $w_{D_6} \bydef w_0$. Next, we construct $B^N$ using the approach described in Section \ref{sec : operator A first} compute $\kappa$ using \eqref{def : kappa}, and use the constants defined in Section 3.5.1 of \cite{sh_cadiot}. We find
\begin{align}
    \|B^N\|_{2} \leq 24.12,~~\kappa \bydef 2.055
,~~C_0 \bydef 2.41.
 \end{align}
Finally, using \cite{julia_cadiot_blanco_symmetry}, we choose $r_0 \bydef 8 \times 10^{-5}$ and define
\begin{align}
    \mathcal{Y}_0 \bydef 7.6 \times 10^{-6} \text{,}~\mathcal{Z}_{2}(r_0) \bydef 263.71   \text{,}~\mathcal{Z}_1 \bydef 0.254,\mathcal{Y}_s \bydef 2.85 \times 10^{-6},\mathcal{Z}_{22} \bydef 0.1161, \mathcal{Z}_s \bydef 0.1469.
    \end{align}
and prove that these values satisfy Theorem \ref{th: radii polynomial s}. 
\end{proof}
\begin{theorem}[\bf The $D_8$ pattern]
Let $\mu = 0.28, \nu_1 = -1.6, \nu_2 = 1$. Moreover, let $r_0 \bydef 3 \times 10^{-4}$. Then there exists a unique solution $\tilde{w}$ to \eqref{eq : swift_hohenberg} in $\overline{B_{r_0}(w_{D_8})} \subset H^l_{D_8}$ and we have that $\|\tilde{w}-w_{D_8}\|_{l} \leq r_0$.
\end{theorem}
\begin{proof}
Choose $N_0 = 130, N = 100, N_1 = 115, d = 76$. Then, we perform the full construction described in Section \ref{sec : construction of w_0} to build $u_0 = \gamma^\dagger(U_0)$ and then $w_0$ using \eqref{v_D_m_fourier}. Then, we define $w_{D_8} \bydef w_0$. Next, we construct $B^N$ using the approach described in Section \ref{sec : operator A first} compute $\kappa$ using \eqref{def : kappa}, and use the constants defined in Section 3.5.1 of \cite{sh_cadiot}. We find
\begin{align}
    \|B^N\|_{2} \leq 141.87,~~\kappa \bydef 2.28,~~C_0 \bydef 2.6.
 \end{align}
Finally, using \cite{julia_cadiot_blanco_symmetry}, we choose $r_0 \bydef 3 \times 10^{-4}$ and define
\begin{align}
    \mathcal{Y}_0 \bydef 2.7 \times 10^{-5} \text{,}~\mathcal{Z}_{21}(r_0) \bydef  1298.3 \text{,}~\mathcal{Z}_1 \bydef 0.17 \text{,}~\mathcal{Y}_{s} \bydef 2.2 \times 10^{-4}\text{,}~\mathcal{Z}_{22} \bydef 0.062 \text{,}~\mathcal{Z}_s \bydef 0.07   \end{align}
and prove that these values satisfy Theorem \ref{th: radii polynomial s}. 
\end{proof}
\begin{theorem}[\bf The $D_{10}$ pattern]
   Let $\mu = 0.25, \nu_1 = -1.6, \nu_2 = 1$. Moreover, let $r_0 \bydef 8 \times 10^{-5}$. Then there exists a unique solution $\tilde{w}$ to \eqref{eq : swift_hohenberg} in $\overline{B_{r_0}(w_{D_{10}})} \subset H^l_{D_{10}}$ and we have that $\|\tilde{w}-w_{D_{10}}\|_{l} \leq r_0$.
\end{theorem}
\begin{proof}
Choose $N_0 = 150, N = 100, N_1 = 120, d = 85$. Then, we perform the full construction described in Section \ref{sec : construction of w_0} to build $u_0 = \gamma^\dagger(U_0)$ and then $w_0$ using \eqref{v_D_m_fourier}. Then, we define $w_{D_{10}} \bydef w_0$. Next, we construct $B^N$ using the approach described in Section \ref{sec : operator A first} compute $\kappa$ using \eqref{def : kappa}, and use the constants defined in Section 3.5.1 of \cite{sh_cadiot}. We find
\begin{align}
    \|B^N\|_{2} \leq 137.94,~~\kappa \bydef 2.49,~~C_0 \bydef 2.763.
 \end{align}
Finally, using \cite{julia_cadiot_blanco_symmetry}, we choose $r_0 \bydef 8 \times 10^{-5}$ and define
\begin{align}
    \mathcal{Y}_0 \bydef 2.18 \times 10^{-5} \text{,}~\mathcal{Z}_{21}(r_0) \bydef  1908 \text{,}~\mathcal{Z}_1 \bydef 0.167 \text{,}\mathcal{Y}_{s} \bydef 1.02 \times 10^{-9}\text{,}~\mathcal{Z}_{22} \bydef 0.124 \text{,}~\mathcal{Z}_s \bydef 9.2 \times 10^{-4}   \end{align}
and prove that these values satisfy Theorem \ref{th: radii polynomial s}. 
\end{proof}
\begin{theorem}[\bf The $D_{12}$ pattern]
       Let $\mu = 0.28, \nu_1 = -1.6, \nu_2 = 1$. Moreover, let $r_0 \bydef 7 \times 10^{-5}$. Then there exists a unique solution $\tilde{w}$ to \eqref{eq : swift_hohenberg} in $\overline{B_{r_0}(w_{D_{12}})} \subset H^l_{D_{12}}$ and we have that $\|\tilde{w}-w_{D_{12}}\|_{l} \leq r_0$.
\end{theorem}
\begin{proof}
Choose $N_0 = 180, N = 110, N_1 = 150, d = 100$. Then, we perform the full construction described in Section \ref{sec : construction of w_0} to build $u_0 = \gamma^\dagger(U_0)$ and then $w_0$ using \eqref{v_D_m_fourier}. Then, we define $w_{D_{12}} \bydef w_0$. Next, we construct $B^N$ using the approach described in Section \ref{sec : operator A first}, compute $\kappa$ using \eqref{def : kappa}, and use the constants defined in Section 3.5.1 of \cite{sh_cadiot}. We find
\begin{align}
    \|B^N\|_{2} \leq 1071.71,~~\kappa \bydef 2.28,~~C_0 \bydef 2.6.
 \end{align}
Finally, using \cite{julia_cadiot_blanco_symmetry}, we choose $r_0 \bydef 7 \times 10^{-5}$ and define
\begin{align}
    \mathcal{Y}_0 \bydef 8.31 \times 10^{-6} \text{,}~\mathcal{Z}_{21}(r_0) \bydef  11021 \text{,}~\mathcal{Z}_1 \bydef 0.1311  \text{,}\mathcal{Y}_{s} \bydef 4.9 \times 10^{-8} \text{,}~\mathcal{Z}_{22} \bydef 0.311  \text{,}~\mathcal{Z}_s \bydef 0.3535   \end{align}
and prove that these values satisfy Theorem \ref{th: radii polynomial s}. 
\end{proof}
\subsection{Proofs when \texorpdfstring{$j$}{j} is a prime number other than \texorpdfstring{$2$}{2}}
In this section, we provide an existence proof for a $D_5$-symmetric localized pattern. This symmetry is in fact of particular interest. Indeed, $D_2$, $D_4$, $D_6$ symmetries  are expected because their related $j$-gons tile the plane. In other words, $D_2$, $D_4$ and $D_6$ naturally lead to (spatially periodic) Turing patterns, and, taking the limit as the period tends to infinity, one can expect related localized patterns (similarly as what is achieved in \cite{sh_cadiot}).  

In the case of $D_5$-symmetry, since regular pentagons do not tile the plane, such a reasoning does not hold, and the existence of $D_5$-symmetric localized patterns is still an open question. Their existence was for instance conjectured by the authors of \cite{jason_review_paper,jason_spot_paper, jason_ring_paper}. In the theorem below, we establish the first existence proof of a $D_5$-symmetric localized solution in \eqref{eq : swift_hohenberg}. 
\begin{theorem}[\bf The $D_5$ pattern]
Let $\mu = 0.2, \nu_1 = -1.6, \nu_2 = 1$. Moreover, let $r_0 \bydef 7 \times 10^{-6}$. Then there exists a unique solution $\tilde{x} \bydef (0,\tilde{w})$ to \eqref{f_unfold_augmented} in $\overline{B_{r_0}(x_{D_5})} \subset H_1 \bydef \mathbb{R} \times H^l_{D_5}$ where $x_{D_5} = (0,w_{D_5})$ and we have that $\|\tilde{x}-x_{D_5}\|_{H_1} \leq r_0$. In particular, $\tilde{w}$ is a $D_5$ solution to \eqref{eq : swift_hohenberg}.
\end{theorem}
\begin{proof}
Choose $N_0 = 130, N = 73, N_1 = 100, d = 76$. Then, we perform the full construction described in Section \ref{sec : construction of w_0} to build $u_0 = \gamma^\dagger(U_0)$ and then $w_0$ using \eqref{v_D_m_fourier}. Then, we define $w_{D_5} \bydef w_0$. Following this, since we need to using the unfolding parameter setup, we define $x_{D_5} \bydef (0,w_{D_5})$. Next, we construct $B^N$ using the approach described in Section \ref{sec : operator A first}, compute $\kappa$ using \eqref{def : kappa}, and use the constants defined in Section 3.5.1 of \cite{sh_cadiot}. We find
\begin{align}
    \|B^N\|_{H_2} \leq 28.35,~ \kappa \bydef 2.941,~C_0 \bydef 3.124.
 \end{align}
Finally, using \cite{julia_cadiot_blanco_symmetry}, we choose $r_0 \bydef 7 \times 10^{-6}$ and define
{\small\begin{align}
    \mathcal{Y}_0 \bydef 1.31 \times 10^{-6} \text{,}~\mathcal{Z}_{21}(r_0) \bydef  838.8 \text{,}~\mathcal{Z}_1 \bydef 0.6621 \text{,}\mathcal{Y}_{s} \bydef 9.25 \times 10^{-8}\text{,}~\mathcal{Z}_{22} \bydef 0.012 \text{,}~\mathcal{Z}_s \bydef 0.0953   \end{align}}
and prove that these values satisfy Theorem \ref{th: radii polynomial s}. 
\end{proof}
\section{Conclusion}
In this manuscript, we have provided a method for proving the existence of $\mathcal{G}$-symmetric localized patterns on $\R^m$ when $\mathcal{G}$ is not necessarily a space group. Furthermore, it is applicable to dihedral symmetries, which we illustrated thanks to proof of dihedral patterns in the planar SH PDE. Specifically, we have provided the first proofs of existence of non-space group symmetric solutions.
In fact, this potential for verifying very general symmetries opens the door to further investigations on localized patterns.\\
 For instance, it gives the ability to prove branches of solutions with symmetry. As demonstrated in \cite{whitham_cadiot}, one can perform rigorous proof of a branch of solutions on unbounded domains. Combining the present work and \cite{whitham_cadiot},   one can validate the symmetry of a branch of localized solutions. This has the potential of detecting and validating symmetry breaking bifurcations, as well as the entering and leaving branches. We consider this a future work.

 While our approach provides novel results in the field of pattern formation, it can still be technically  improved. For now, it relies on Lemma \ref{lem : w_0 minus u_0} to compute quantities of the form $\|w - g \cdot w\|_{2}$ for at least one $g \in \mathcal{G}$ in each group orbit. Each of these quantities must be computed using continuous variables since the mixed term does not entertain the usual properties of Fourier series (i.e. Parseval's identity). As a result, each of these computations is far from efficient, and we found that a single one of them takes longer than everything else in the numerical work load. In the case of dihedral symmetry, this is not a huge burden since we can enforce all the reflections using $\mathbb{Z}_2$ symmetry meaning we only have the orbit generating the rotations. For more complicated groups with more orbits, this could become problematic. This will also pose a problem in higher dimensions where the time needed to compute this quantity will be even longer. This is, as of now, the computation limiting us to treat 3D problems. It would be of interest to improve the computation speed of Lemma \ref{lem : w_minus_rotation}, or provide an alternative estimation. 
 \par Another result of interest is the question of symmetries that can tile the plane on the hexagonal lattice. As mentioned in the introduction, we can represent $D_3$ and $D_6$ periodic functions using Fourier series. This is done by introducing a transformation $\mathcal{L}$ which takes us on the hexagonal lattice. However, due to the fact that $\mathcal{L}^{-T} \om$ is not invariant itself under the symmetry of $D_3$ or $D_6$, we were unable to use it as part of our approach. It would be of interest to make our approach compatible with space groups on the hexagonal lattice, making the proofs more memory efficient.

\section{Acknowledgements}
 Matthieu Cadiot was supported by the ANR project CAPPS: ANR-23-CE40-0004-01 and by the FMJH :  ANR-22-EXES-0013.

\appendix
\renewcommand{\theequation}{A.\arabic{equation}}
\setcounter{equation}{0}
\section{Computing the Bounds for Localized Patterns when \texorpdfstring{$j$}{j} is not a prime number nor \texorpdfstring{$4$}{4}}\label{apen : localized s computations}
In this section, we provide the computational details in order to compute the bounds stated in Theorem \ref{th: radii polynomial s} when $\mathcal{H} < \mathcal{G}$ and $\mathcal{H}$ isolates the solution. For dihedral groups, this would be when $\mathcal{H} \in \{D_2, D_4\}$. 
\subsection{Computation of \texorpdfstring{$\mathcal{Y}_S$}{Ys}}
We begin by computing the $\mathcal{Y}_s$ bound provided in Section \ref{sec : Y_s isolate}. More specifically, we prove Lemmas \ref{lem : w_minus_rotation} and \ref{lem : Ys Bound}. We will start with Lemma \ref{lem : w_minus_rotation}.
\subsubsection{Proof of Lemma \ref{lem : w_minus_rotation}}\label{apen : w minus rotation proof}
To begin, since $\mathrm{supp}(w) \subset  {\om}$, observe that.
\begin{align}
    \left\| w - R_{\theta} w\right\|_{2} &= \left\|\mathbb{1}_{\om \cup R_{\theta} \om} \left( w - R_{\theta} w\right)\right\|_{2}.
\end{align}
 Now, observe that
\begin{align}
    \nonumber \|\mathbb{1}_{\om \cup R_{\theta} \om}(w - R_{\theta} w)\|_{2}^2 &= \int_{\om \cup R_{\theta} \om} \left(w(x) - w\left(R_{\theta}x\right)\right)  {\left(w(x) - w\left(R_{\theta}x\right)\right)} dx \\ \nonumber
    &\hspace{-0.5cm}= \int_{\om} w(x) {w(x)} dx - 2\int_{\om \cap R_{\theta} \om} w(x)  {w\left(R_{\theta}x\right)} dx + \int_{ R_{\theta}\om} w\left(R_{\theta}x\right)  {w\left(R_{\theta}x\right)} dx. \end{align}
Observe that the first integral is $|\om|\|W\|_{2}^2$. In the third integral, we make the change of variable $y =  R_{\theta} x$ to undo the rotation. This results in
\begin{align}
    \hspace{-0.5cm}\int_{  R_{\theta}\om} w\left(R_{\theta}x\right)  {w\left(R_{\theta}x\right)} dx = \int_{ \om} w\left(y\right)  {w\left(y\right)} dy = |\om| \|W\|_{2}^2.\label{third_int_result w_minus_rotation}
\end{align}
Putting together \eqref{third_int_result w_minus_rotation} with the above, we obtain    \begin{align}\nonumber
    \|w - R_{\theta} w\|_{2} &= 2|\om| \|W\|_{2}^2 - 2 \int_{\om \cap R_{\theta}\om}  {w(y)} w\left(R_{\theta}y\right) dy \\
    &= \left(2|\om| \|W\|_{2}^2 - 2 \int_{\om \cap  R_{\theta} \om}  {w(y)} w\left(R_{\theta}y\right) dy\right).
\end{align}
\begin{remark}\label{rem : N_1 truncation}
In practice, for a given $\theta$, the quantity $\phi(w,\theta)$ can be evaluated thanks to rigorous numerics (cf. \cite{julia_cadiot_blanco_symmetry}).
In order to decrease the numerical complex quantity of its computation, one can introduce a further truncation. More specifically, let 
\begin{align}
    W^{N_1} \bydef \pi^{N_1} W \text{ and }  w^{N_1} \bydef \gamma_{\mathcal{H}}^\dagger(W^{N_1}).
\end{align}
Then, 
\begin{align}
    \phi(w,R_{\theta}) = \|w - R_{\theta} w\|_{2} &\leq \|w^{N_1} - R_{\theta} w^{N_1}\|_{2} + \|w^{N_1} - w\|_{2} + \|R_{\theta} (w^{N_1} - w)\|_{2} \\
    &= \|w^{N_1} - R_{\theta} w^{N_1}\|_{2} + 2\|w^{N_1} - w\|_{2} \\
    &= \|w^{N_1} - R_{\theta} w^{N_1}\|_{2} + 2\sqrt{|\om||}\|W^{N_1} - W\|_{2} \\
    &= \phi(w^{N_1},R_{\theta}) + 2\sqrt{|\om|}\|W^{N_1} - W\|_{2}.
\end{align}
The defect $\|W^{N_1} - W\|_{2}$ will be small if the decay of the coefficients of $W$ is good enough.
\end{remark}
\subsubsection{Proof of Lemma \ref{lem : Ys Bound}}\label{apen : Y_s isolate proof}
With Lemma \ref{lem : w_minus_rotation} now proven, we are able to prove Lemma \ref{lem : Ys Bound}. We now do so.
\begin{proof}
To begin, Lemma 3.4 of \cite{unbounded_domain_cadiot} gives that
\begin{align}
    \|\mathbb{A}\tilde{\mathbb{G}}(u_0)\|_{l} = \|\mathbb{B}\tilde{\mathbb{G}}(u_0)\|_{2} \leq \|\mathbb{B}\|_{2} \|\tilde{\mathbb{G}}(u_0)\|_{2} = \max\{1,\|B^N\|_{2}\} \|\tilde{\mathbb{G}}(u_0)\|_{2}.
\end{align}
In the case of \eqref{eq : swift_hohenberg}, $\mathbb{G}$ involves a quadratic and cubic term. With this in mind, observe that
\begin{align}
    w_0^2 = \frac{1}{j} \sum_{k = 0}^{j-1} R_{\frac{2\pi k}{j}} u_0^2 +\tilde{\mathbb{G}}_1(u_0),~~w_0^3 = \frac{1}{j} \sum_{k = 0}^{j-1} R_{\frac{2\pi k}{j}} u_0^3 + \tilde{\mathbb{G}}_2(u_0)
\end{align}
where
\begin{align}
    &\tilde{\mathbb{G}}_1(u_0) \bydef - \frac{1}{2j^2} \sum_{k,p = 0}^{j-1} \left(R_{\frac{2\pi k}{j}} u_0 - R_{\frac{2\pi p}{j}} u_0\right)^2 \\
    &\tilde{\mathbb{G}}_2(u_0) \bydef \frac{1}{j^3} \sum_{k,c,p = 0}^{j-1} \left(R_{\frac{2\pi k}{j}} u_0 - R_{\frac{2\pi c}{j}} u_0\right)\left(\left(R_{\frac{2\pi c}{j}} u_0\right) \left(R_{\frac{2\pi p}{j}}u_0\right) - R_{\frac{2\pi k}{j}} u_0^2\right).
\end{align}
Notice that we can bound
\begin{align}
 \|\tilde{\mathbb{G}}(u_0)\|_{2} \leq |\nu_1|\|\tilde{\mathbb{G}}_1(u_0)\|_{2} + |\nu_2|\|\tilde{\mathbb{G}}_2(u_0)\|_{2}.
\end{align}
Let us now compute $\|\tilde{\mathbb{G}}_1(u_0)\|_{2}$. 
{\footnotesize\begin{align} \|\tilde{\mathbb{G}}_1(u_0)\|_{2} &\leq \frac{1}{2j^2} \sum_{k,p = 0}^{j-1} \left\|\left(R_{\frac{2\pi k}{j}} u_0 - R_{\frac{2\pi p}{j}} u_0\right)^2\right\|_{2} \leq \frac{1}{2j^2} \sum_{k,p = 0}^{j-1} \left\|R_{\frac{2\pi k}{j}} u_0 - R_{\frac{2\pi p}{j}} u_0\right\|_{2}\left\|R_{\frac{2\pi k}{j}} u_0 - R_{\frac{2\pi p}{j}} u_0\right\|_{\infty}.\label{step_in_Ys_1}
\end{align}}
Now, we define 
\begin{align}
    \|u\|_{l_0} \bydef \|\mathbb{L}_0 u\|_{2} \text{ for all } u \in L^2(\R^2).\label{def : l0_norm}
\end{align}
Using \eqref{def : l0_norm},
\begin{align}
  \left\|R_{\frac{2\pi k}{j}} u_0 - R_{\frac{2\pi p}{j}} u_0\right\|_{\infty} 
    &\leq \left\|\frac{1}{l_0}\right\|_{2} \left\| \mathbb{L}_0 \left(R_{\frac{2\pi k}{j}} u_0 - R_{\frac{2\pi p}{j}} u_0\right)\right\|_{2} = \left\|\frac{1}{l_0}\right\|_{2} \left\| R_{\frac{2\pi k}{j}} u_0 - R_{\frac{2\pi p}{j}} u_0\right\|_{l_0}\label{estimate_by_l0_norm}
\end{align}
where we used Cauchy Schwarz on the second to last step. Using \eqref{estimate_by_l0_norm}, we return to \eqref{step_in_Ys_1}
\begin{align}
    \|\tilde{\mathbb{G}}_1(u_0)\|_{2} &\leq \left\|\frac{1}{l_0}\right\|_{2}\frac{1}{2j^2} \sum_{k,p=0}^{j-1} \left\| R_{\frac{2\pi k}{j}} u_0 - R_{\frac{2\pi p}{j}} u_0\right\|_{2}\left\| R_{\frac{2\pi k}{j}} u_0 - R_{\frac{2\pi p}{j}} u_0\right\|_{l_0} \\ \nonumber
    &= \left\|\frac{1}{l_0}\right\|_{2}\frac{1}{2j^2} \sum_{k,p=0}^{j-1} \left\| R_{\frac{2\pi k}{j}} \left(u_0 - R_{\frac{2\pi (p-k)}{j}} u_0\right)\right\|_{2}\left\| R_{\frac{2\pi k}{j}} \left(u_0 - R_{\frac{2\pi (p-k)}{j}} u_0\right)\right\|_{l_0} \\ \nonumber
    &= \left\|\frac{1}{l_0}\right\|_{2}\frac{1}{2j^2} \sum_{k,p=0}^{j-1} \left\|u_0 - R_{\frac{2\pi (p-k)}{j}} u_0\right\|_{2}\left\| u_0 - R_{\frac{2\pi (p-k)}{j}} u_0\right\|_{l_0}. 
\end{align}
Now, we must compute $\left\|u_0 - R_{\frac{2\pi (p-k)}{j}} u_0\right\|_{2}$ and $\|u_0 - R_{\frac{2\pi (p-k)}{j}} u_0\|_{l_0}$. We use Lemma \ref{lem : w_minus_rotation} to obtain
\begin{align}
    \|u_0 - R_{\frac{2\pi (p-k)}{j}} u_0\|_{2} \leq \phi\left(u_0,\frac{2\pi (p-k)}{j}\right),~~\|u_0 - R_{\frac{2\pi (p-k)}{j}} u_0\|_{l_0} \leq \phi\left(\mathbb{L}_0 u_0, \frac{2\pi (p-k)}{j}\right).\label{u_0_minus_norms_2_l0}
\end{align}
Using \eqref{u_0_minus_norms_2_l0}, we obtain
\begin{align}
      |\nu_1|\|\tilde{\mathbb{G}}_1(u_0)\|_{2} \leq \left\|\frac{1}{l_0}\right\|_{2}\frac{|\nu_1|}{2j^2 }\sum_{k,p = 0}^{j-1} \phi\left(u_0,R_{\frac{2\pi (p-k)}{j}}\right)\phi\left(\mathbb{L}_0u_0,R_{\frac{2\pi (p-k)}{j}}\right).\label{tilde_G_1_estimated}
\end{align}
\par Now, let us estimate $\|\tilde{\mathbb{G}}_2(u_0)\|_{2}$. We begin with the triangle inequality.
\begin{align}
    \|\tilde{\mathbb{G}}_2(u_0)\|_{2} \leq \frac{1}{j^3} \sum_{k,c,p = 0}^{j-1} \left\|\left(R_{\frac{2\pi k}{j}} u_0 - R_{\frac{2\pi c}{j}} u_0\right)\left(\left(R_{\frac{2\pi c}{m}} u_0\right) \left(R_{\frac{2\pi p}{j}}u_0\right) - R_{\frac{2\pi k}{j}} u_0^2\right)\right\|_{2}.
\end{align}
Let us now focus on the second term inside the norm. We factorize, and obtain 
{\footnotesize\begin{align}
    \nonumber &\left\|\left(R_{\frac{2\pi k}{j}} u_0 - R_{\frac{2\pi c}{j}} u_0\right)\left(\left(R_{\frac{2\pi c}{j}} u_0\right) \left(R_{\frac{2\pi p}{j}}u_0\right) - R_{\frac{2\pi k}{j}} u_0^2\right)\right\|_{2} \\ \nonumber
    &= \left\|\left(R_{\frac{2\pi k}{j}} u_0 - R_{\frac{2\pi c}{j}} u_0\right)\left[\left( R_{\frac{2\pi c}{j}} u_0\right) \left( R_{\frac{2\pi p}{j}} u_0 - R_{\frac{2\pi k}{j}} u_0\right) + \left(R_{\frac{2\pi k}{j}} u_0\right) \left(R_{\frac{2\pi c}{j}} u_0 - R_{\frac{2\pi k}{j}} u_0 \right)\right]\right\|_{2} \\ \nonumber
    &\leq \frac{1}{2\sqrt{\pi}} \left\| R_{\frac{2\pi k}{j}} u_0 - R_{\frac{2\pi c}{j}} u_0\right\|_{2}\left(\left\| R_{\frac{2\pi c}{j}} U_0\right\|_{1} \left\|R_{\frac{2\pi p}{j}} u_0 - R_{\frac{2\pi k}{j}} u_0 \right\|_{l_0} + \left\| R_{\frac{2\pi k}{j}} U_0\right\|_{1} \left\| R_{\frac{2\pi c}{j}} u_0 - R_{\frac{2\pi k}{j}} u_0\right\|_{l_0}\right)  
    \end{align}}
Since $\|R_{\theta} U_0\|_{1} = \|U_0\|_{1}$, we have    
\begin{align} \nonumber
    &= \left\|\frac{1}{l_0}\right\|_{2} \left\| U_0\right\|_{1}\left\| R_{\frac{2\pi k}{j}} u_0 - R_{\frac{2\pi c}{j}} u_0\right\|_{2}\left( \left\|R_{\frac{2\pi p}{j}} u_0 - R_{\frac{2\pi k}{j}} u_0 \right\|_{l_0} + \left\|R_{\frac{2\pi c}{j}} u_0 - R_{\frac{2\pi k}{j}} u_0\right\|_{l_0}\right) \\ \nonumber
    &= \left\|\frac{1}{l_0}\right\|_{2} \left\| U_0\right\|_{1}\left\| u_0 - R_{\frac{2\pi (c-k)}{j}} u_0\right\|_{2}\left( \left\|u_0 - R_{\frac{2\pi (k-p)}{j}} u_0 \right\|_{l_0} + \left\|u_0 - R_{\frac{2\pi (k-c)}{j}} u_0\right\|_{l_0}\right) \\ \nonumber
    &= \left\|\frac{1}{l_0}\right\|_{2} \|U_0\|_{1} \phi\left(u_0,R_{\frac{2\pi(c-k)}{j}}\right)\left(\phi\left(\mathbb{L}_0u_0,R_{\frac{2\pi (k-p)}{j}}\right) + \phi\left(\mathbb{L}_0u_0,R_{\frac{2\pi (k-c)}{j}}\right)\right).
\end{align}
Hence, we obtain
{\small\begin{align}
    \nonumber |\nu_2|\|\tilde{\mathbb{G}}_2(u_0)\|_{2} &\leq \left\|\frac{1}{l_0}\right\|_{2}\frac{|\nu_2|}{j^3}\|U_0\|_{1} \sum_{k,c,p = 0}^{j-1} \phi\left(u_0,R_{\frac{2\pi(c-k)}{j}}\right)\left(\phi\left(\mathbb{L}_0u_0,R_{\frac{2\pi (k-p)}{j}}\right) + \phi\left(\mathbb{L}_0u_0,R_{\frac{2\pi (k-c)}{j}}\right)\right).\label{tilde_G2_estimated}
\end{align}}
With \eqref{tilde_G_1_estimated} and \eqref{tilde_G2_estimated}, we have a result for the $\mathcal{Y}_s$ bound that we could compute directly.
\par The goal of the previous estimations was to re-write the necessary computations in such a way that Lemma \ref{lem : w_minus_rotation} applies. We now take this a step further and aim to only apply the aforementioned lemma once. More specifically, we wish to only compute $\phi(u_0, R_{\frac{2\pi}{j}})$ and $\phi(\mathbb{L}_0 u_0, R_{\frac{2\pi}{j}})$. Indeed, observe that for $v = u_0, \mathbb{L}_0 u_0$
\begin{align}
    \phi(v, R_{\frac{2\pi k}{j}}) = \|v - R_{\frac{2\pi k}{j}} v\|_{2} = \left\|\sum_{p = 0}^{k-1} R_{\frac{2\pi p}{j}} (v - R_{\frac{2\pi }{j}} v)\right\|_{2}   \leq  \left\|\sum_{p = 0}^{k-1} R_{\frac{2\pi p}{j}}\right\|_{2} \phi(v, R_{\frac{2\pi}{j}}).
\end{align}
Therefore, only $\phi(v, R_{\frac{2\pi}{j}})$ must be computed using Lemma \ref{lem : w_minus_rotation}. What remains is to compute $\left\|\sum_{c = 0}^{k-1} R_{\frac{2\pi c}{j}}\right\|_{2}$. This can be done using properties of matrices and linear algebra. This results in
% In particular,
% \begin{align}
%     \sum_{c = 0}^{k-1} R_{\frac{2\pi c}{j}} = I_d + \sum_{c = 1}^{k-1}\begin{bmatrix}
%         \cos(\frac{2\pi c}{j}) & -\sin(\frac{2\pi c}{j}) \\
%         \sin(\frac{2\pi c}{j}) & \cos(\frac{2\pi c}{j})
%     \end{bmatrix} \bydef  \begin{bmatrix}
%         1 + b & -o \\
%         o & 1 + b
%     \end{bmatrix} \bydef \mathcal{A}
% \end{align}
% where 
% \begin{align}
%     &b \bydef \sum_{c = 1}^{k-1} \cos(\frac{2\pi c}{j}) \\
%     &o \bydef \sum_{c = 1}^{k-1} \sin(\frac{2\pi c}{j}).
% \end{align}
% To compute the $2$-norm, observe that
% \begin{align}
%     \mathcal{A}\mathcal{A}^* = \begin{bmatrix}
%         1 + b & -o \\
%         o & 1 + b
%     \end{bmatrix}\begin{bmatrix}
%         1 + b & o \\
%         -o & 1 + b
%     \end{bmatrix} = \begin{bmatrix}
%         (1+b)^2 + o^2 & 0 \\
%         0 & (1+b)^2 + o^2
%     \end{bmatrix}.
% \end{align}
% Since the resulting matrix is diagonal, the eigenvalues are $(1 + b)^2 + o^2$. Hence, 
% {\small\begin{align}
%     \left\| \sum_{c = 0}^{k-1} R_{\frac{2\pi c}{j}}\right\|_{2} = \sqrt{\|\mathcal{A}\mathcal{A}^*\|_{2}} = \sqrt{(1 + b)^2 + o^2} = \sqrt{\left(\sum_{c = 0}^{k-1} \cos(\frac{2\pi c}{j})\right)^2 + \left(\sum_{c = 1}^{k-1} \sin(\frac{2\pi c}{j})\right)^2} \bydef \mathcal{R}_{k,j}.
% \end{align}}
{\small\begin{align}
    \left\| \sum_{p = 0}^{k-1} R_{\frac{2\pi p}{j}}\right\|_{2} = \sqrt{\left(\sum_{p = 0}^{k-1} \cos(\frac{2\pi p}{j})\right)^2 + \left(\sum_{p = 1}^{k-1} \sin(\frac{2\pi p}{j})\right)^2} \bydef \mathcal{R}_{k,j}.
\end{align}}
Therefore, we see that
$\phi(v, R_{\frac{2\pi k}{j}}) \leq \mathcal{R}_{k,j} \phi(v, R_{\frac{2\pi}{j}}).$
We now return to \eqref{tilde_G_1_estimated} and obtain
\begin{align}
    |\nu_1|\|\tilde{\mathbb{G}}_1(u_0)\|_{2} 
    &\leq \left\|\frac{1}{l_0}\right\|_{2}\frac{|\nu_1|}{2j^2 }\phi\left(u_0,R_{\frac{2\pi}{j}}\right)\phi\left(\mathbb{L}_0u_0,R_{\frac{2\pi}{j}}\right)\sum_{k,c = 0}^{j-1} \mathcal{R}_{c-k,j}^2
\end{align}
which is the $\mathcal{Y}_{s,1}$ bound defined in Lemma \ref{lem : Ys Bound}.
Additionally, we use it in \eqref{tilde_G2_estimated} to obtain
{\footnotesize\begin{align}
    &|\nu_2|\|\tilde{\mathbb{G}}_2(u_0)\|_{2} \leq \left\|\frac{1}{l_0}\right\|_{2}\frac{|\nu_2|}{j^3}\|U_0\|_{1} \phi\left(u_0,R_{\frac{2\pi}{j}}\right)\sum_{k,c,p = 0}^{j-1} \mathcal{R}_{c-k,j}\left(\mathcal{R}_{k-p,m}\phi\left(\mathbb{L}_0u_0,R_{\frac{2\pi}{j}}\right) + \mathcal{R}_{k-c,j}\phi\left(\mathbb{L}_0u_0,R_{\frac{2\pi }{j}}\right)\right)
\end{align}}
which is the $\mathcal{Y}_{s,2}$ bound defined in Lemma \ref{lem : Ys Bound}. This concludes the proof.
\end{proof}
\subsection{Computation of \texorpdfstring{$\mathcal{Z}_2$}{Z2}}\label{apen : Z2 isolate}
In this section, we aim to compute the $\mathcal{Z}_2(r)$ bound when $\mathcal{H}$ isolates the solution defined in Section \ref{sec : Z2 HleqG H isolates}. Note that this bound differs from the usual $\mathcal{Z}_2(r)$ when $\mathcal{H} = \mathcal{G}$ as we have $w_0$ in its statement. We must prove Lemma \ref{lem : bound Z_2 s}.
% \begin{lemma}
% Let $w_0 \in H^l_{D_j}$ be defined as in \eqref{v_D_m_fourier} with $u_0 \in H^l_{D_j}$. Denote $\kappa>0$ the constant defined in \eqref{def : kappa}. Moreover, let $r>0$ and let $\mathcal{Z}_2(r) >0$ be such that 
% \begin{equation}
%     \mathcal{Z}_2(r) \geq  \mathcal{Z}_{2,1}(r) + \mathcal{Z}_{2,2}(r) 
% \end{equation}
% where
% \begin{align}
%     &\mathcal{Z}_{2,1}(r) \bydef 3|\nu_2|\frac{\kappa^2}{\mu} \max \left\{1,\|B^N\|_{2}\right\} r +  \frac{\kappa}{\mu} \max\left\{ 2|\nu_1|, ~\left(\|\mathbb{W}(B^N)^{\star}\|_{2}^2+\|W\|_1^2\right)^{\frac{1}{2}}\right\} \\ 
%     &\mathcal{Z}_{2,2}(r) \bydef \frac{6\kappa^2 |\nu_2|}{j} \max\{1,\|B^N\|_{2}\} \phi\left(u_0, R_{\frac{2\pi}{j}}\right)\sum_{k = 0}^{j-1}\mathcal{R}_{k,j} 
% \end{align}
% where $W \bydef (w_k)_{k \in \mathbb{N}^2} = (2\nu_1\delta_{k}+6\nu_2u_k)_{k \in \mathbb{N}^2}$ and $\delta_k$ is the Kronecker symbol. Moreover, $\mathbb{W}$ is the discrete convolution operator associated to $W$.
% Then $ \|\mathbb{A}\left({D}\mathbb{F}(s) - D\mathbb{F}(w_0)\right)\|_l \leq \mathcal{Z}_2(r)r$  for all $s \in \overline{B_r(w_0)} \subset H^l_{D_j}.$
% \end{lemma}

\begin{proof}
Let $s \in \overline{B_r(w_0)}$. Then, $\|\mathbb{A}\left({D}\mathbb{F}(s) - D\mathbb{F}(w_0)\right)\|_l
   = \|\mathbb{B}\left({D}\mathbb{G}(s) - D\mathbb{G}(w_0)\right)\|_{l,2}$.
Now let  $h \bydef s-w_0 \in \overline{B_r(0)} \subset H^l_{D_j}$ (in particular $\|h\|_l \leq r$).  Then we have
\begin{align}
\nonumber D\mathbb{G}(s) - D\mathbb{G}(w_0) &= 2\nu_1(\mathbb{w}_0 + \mathbb{h}) + 3\nu_2(\mathbb{w}_0+\mathbb{h})^2 - 2\nu_1\mathbb{w}_0 -3\nu_2\mathbb{w}_0^2 = 2\nu_1\mathbb{h} + 6\nu_2\mathbb{w}_0\mathbb{h} + 3\nu_2\mathbb{h}^2.
\end{align}
Now, instead of proceeding as in Lemma 3.3 of \cite{sh_cadiot}, we only use Lemma 2.1 from \cite{sh_cadiot} on the quadratic term.
\begin{align}
    \|\mathbb{B}(2\nu_1 \mathbb{h} + 6\nu_2 \mathbb{w}_0 \mathbb{h} + 3\nu_2 \mathbb{h}^2)\|_{l,2} 
    &\leq 3|\nu_2|\|\mathbb{B}\|_{2}\|\mathbb{h}\|_{l,2}^2 + \|\mathbb{B}(2\nu_1I_d + 6\nu_2 \mathbb{w}_0)\mathbb{h}\|_{l,2} \\
    &\leq 3|\nu_2|\frac{\kappa^2}{\mu} \|\mathbb{B}\|_{2} r^2 + \|\mathbb{B}(2\nu_1I_d + 6\nu_2 \mathbb{w}_0)\mathbb{h}\|_{l,2}.\label{step in Z2}
\end{align}
On the second term in \eqref{step in Z2}, we introduce $\mathbb{u}_0$.
\begin{align}
    \nonumber \|\mathbb{B}(2\nu_1I_d + 6\nu_2 \mathbb{w}_0)\mathbb{h}\|_{l,2} &\leq  \|\mathbb{B}(2\nu_1I_d + 6\nu_2 \mathbb{u}_0)\mathbb{h}\|_{l,2} + 6|\nu_2|\|\mathbb{B}(\mathbb{w}_0 - \mathbb{u}_0)\mathbb{h}\|_{l,2} \\ \nonumber
    &\leq\frac{\kappa}{\mu}  \left(\|\mathbb{W}(B^N)^{\star}\|_{2}^2 + \|W\|_{1}^2\right)^{\frac{1}{2}} r+ 6|\nu_2|\|\mathbb{B}\|_{2}\|(\mathbb{w}_0 - \mathbb{u}_0)\mathbb{h}\|_{l,2}
\end{align}
where the last step followed from Lemma 3.3 of \cite{sh_cadiot}. At this step, observe that we have
{\footnotesize\begin{align}
    \nonumber \|\mathbb{A}(D\mathbb{F}(s) - D\mathbb{F}(w_0))\|_{l} 
    &\leq  3|\nu_2|\frac{\kappa^2}{\mu} \|\mathbb{B}\|_{2} r +  \frac{\kappa}{\mu} \max\left\{ 2|\nu_1|, ~\left(\|\mathbb{W}(B^N)^{\star}\|_{2}^2+\|W\|_1^2\right)^{\frac{1}{2}}\right\} + 6|\nu_2|\|\mathbb{B}\|_{2}\|(\mathbb{w}_0 - \mathbb{u}_0)\mathbb{h}\|_{l,2} \\ 
    &\bydef \mathcal{Z}_{2,1}(r)r + 6|\nu_2|\|\mathbb{B}\|_{2}\|(\mathbb{w}_0 - \mathbb{u}_0)\mathbb{h}\|_{l,2}.\label{step in Z2 2}
\end{align}}
Notice that $\mathcal{Z}_{2,1}(r)$ is the same quantity previously referred to as $\mathcal{Z}_2(r)$ in Corollary \ref{th: radii polynomial}. Indeed, the term $6|\nu_2| \|\mathbb{B}\|_{2} \|(\mathbb{w}_0 - \mathbb{u}_0)\mathbb{h}\|_{l,2}$ would be $0$ if $\mathcal{H} = \mathcal{G}$, meaning that it is the $\mathcal{Z}_2(r)$ bound that one would compute when in that case. Let us now examine $6|\nu_2| \|\mathbb{B}\|_{2}\|(\mathbb{w}_0 - \mathbb{u}_0)\mathbb{h}\|_{l,2}$. We use Lemmas \ref{lem : kappa_one} and \ref{lem : w_0 minus u_0} to obtain
\begin{align}
     \|(\mathbb{w}_0 - \mathbb{u}_0)\mathbb{h}\|_{l,2} = \sup_{w \in H^l} \frac{\| (\mathbb{w}_0 - \mathbb{u}_0)\mathbb{h} w\|_{2}}{\|w\|_{l}} &\leq  \kappa^2  \|w_0 - u_0\|_{2} \|h\|_{l} \leq \frac{\kappa^2}{j}\phi\left(u_0, R_{\frac{2\pi}{j}}\right) \sum_{k = 0}^{j-1} \mathcal{R}_{k,j} r.\label{step_in_Z2} 
\end{align}
Returning to \eqref{step in Z2 2}, we obtain
\begin{align}
    6|\nu_2| \|\mathbb{B}\|_{2} \| (\mathbb{w}_0 - \mathbb{u}_0)\mathbb{h}\|_{l,2} \leq \frac{6|\nu|_2 \kappa^2}{j} \max\{1,\|B^N\|_{2}\} \phi\left(u_0, R_{\frac{2\pi }{j}}\right)\sum_{k = 0}^{j-1} \mathcal{R}_{k,j}r \bydef \mathcal{Z}_{2,2}(r)r
\end{align}
as desired. This concludes the computation of the $\mathcal{Z}_2$ bound.
\end{proof}
\subsection{Computation of \texorpdfstring{$\mathcal{Z}_s$}{Zs}}\label{apen : Zs isolate}
We now compute the $\mathcal{Z}_s$ bound defined in Section \ref{sec : Zs isolate}. We present the proof of Lemma \ref{lem : Zs bound}.
% \begin{lemma}
% Let $\mathcal{Z}_s > 0$ be such that
% \begin{align}
%     \mathcal{Z}_s \geq 2\max\{1,\|B^N\|_{2}\}\kappa (\nu_1 + 3\nu_2 \|U_0\|_{1})\frac{1}{j} \phi\left(u_0, R_{\frac{2\pi }{j}}\right)\sum_{k = 0}^{j-1} \mathcal{R}_{k,j}
% \end{align}
% where $\phi$ is defined in Lemma \ref{lem : Ys Bound}
% Then, $\|\mathbb{A}(D\mathbb{G}(w_0) - D\mathbb{G}(u_0))\|_{l} \leq \mathcal{Z}_s$.
% \end{lemma}
\begin{proof}
To begin, observe that
\begin{align}
    \|\mathbb{A}(D\mathbb{G}(w_0) - D\mathbb{G}(u_0))\|_{l} = \|\mathbb{B}(D\mathbb{G}(w_0) - D\mathbb{G}(u_0))\|_{l,2} 
    &\leq \max\{1,\|B^N\|_{2}\}\|D\mathbb{G}(w_0) - D\mathbb{G}(u_0)\|_{l,2} 
\end{align}
Now, observe that
{\footnotesize\begin{align}
    \hspace{-1cm}\|D\mathbb{G}(w_0) - D\mathbb{G}(u_0)\|_{l,2} &=\sup_{h \in H^l} \frac{\|(2\nu_1 (\mathbb{w}_0 - \mathbb{u}_0) + 3\nu_2 (\mathbb{w}_0^2 - \mathbb{u}_0^2))h\|_{2}}{\|h\|_{l}} \leq \kappa \|2\nu_1 (w_0 - u_0) + 3\nu_2 (w_0^2 - u_0^2)\|_{2}\label{eq : first step Zs}
\end{align}}
where the last step followed by Lemma \ref{lem : kappa_one}. Now, using Young's Inequality, observe that
\begin{align}
    \|2\nu_1(w_0 - u_0) + 3\nu_2 (w_0^2 - u_0^2)\|_{2} 
    &\leq  2|\nu_1|\|w_0 - u_0\|_{2} + 3|\nu_2| \|w_0 - u_0\|_{2}\|w_0 + u_0\|_{\infty}.\label{step in Zs}
    \end{align}
Then, using the triangle inequality, the definition of $w_0$ from \eqref{v_D_m_fourier}, and Young's Inequality for sequences, we obtain 
    \begin{align}
    \|w_0 + u_0\|_{\infty} &\leq \|w_0\|_{\infty} + \|u_0\|_{\infty} \leq 2\|u_0\|_{\infty} \leq 2\|U_0\|_{1}\label{step in Zs 2}
    \end{align}
Using \eqref{step in Zs 2}, we return to \eqref{step in Zs} to obtain
{\small\begin{align}
    \hspace{-0.5cm}2|\nu_1|\|w_0 - u_0\|_{2} + 3|\nu_2| \|w_0 - u_0\|_{2}\|w_0 + u_0\|_{\infty}
    &\leq 2(|\nu_1| + 3|\nu_2| \|U_0\|_{1}) \left(\frac{1}{j} \phi\left(u_0, R_{\frac{2\pi}{j}}\right) \sum_{k = 0}^{j-1} \mathcal{R}_{k,j}\right)\label{step_in_Zs}
\end{align}}
where we used Lemma \ref{lem : w_0 minus u_0}. Combining \eqref{step_in_Zs} with \eqref{eq : first step Zs}, we get
\begin{align}
    \|D\mathbb{G}(w_0) - D\mathbb{G}(u_0)\|_{l,2} \leq 2\kappa (|\nu_1| + 3|\nu_2| \|U_1\|_{1}) \frac{1}{j}\phi\left(u_0, R_{\frac{2\pi}{j}}\right) \sum_{k = 0}^{j-1} \mathcal{R}_{k,j}.
\end{align}
Multiplying by $\max\{1,\|B^N\|_{2}\}$, we have completed the proof.
\end{proof}
\renewcommand{\theequation}{B.\arabic{equation}}
\setcounter{equation}{0}
\section{Computing the Bounds for Localized Patterns when \texorpdfstring{$j$}{j} is a prime number other than \texorpdfstring{$2$}{2}}\label{apen : localized s computations not isolate}
In this appendix, we aim to compute the bounds stated in Section \ref{sec : H nonisolate}. In particular, we must compute the $\mathcal{Z}_1, \mathcal{Z}_s,$ and $\mathcal{Z}_2$ bounds. We will begin with the $\mathcal{Z}_1$ bound.
\subsection{Computation of \texorpdfstring{$\mathcal{Z}_1$}{Z1}}\label{apen : Z1 nonisolate}
In this appendix, we will discuss the computation of the $\mathcal{Z}_1$ bound. While this bound is unaffected by our goal to prove the symmetry, it is different from the previous computations presented as we now have an extra equation due to $\mathcal{H}$ not isolating the solution. We provided Lemmas \ref{lem : Z_1_extra} and \ref{lem : Z1_extra} to present these results. We now prove these lemmas. We will begin with Lemma \ref{lem : Z_1_extra}. 
\subsubsection{Proof of Lemma \ref{lem : Z_1_extra}}\label{apen : Z1 derivation proof}
In this appendix, we must prove Lemma \ref{lem : Z_1_extra}. This lemma does not compute a bound in particular, but it derives the $Z_1$ and $\mathcal{Z}_u$ formulas. We now perform the proof.
\begin{proof}
First, let $\overline{\mathbb{L}} \bydef \begin{bmatrix}
    1 & 0 \\
    0 & \mathbb{L}
\end{bmatrix}, \overline{L} \bydef \begin{bmatrix}
    1 & 0 \\
    0 & L
\end{bmatrix},$ and $\mathcal{M}(0,u_0) \bydef \begin{bmatrix}
        -1 & \left(\frac{ \partial_{x_2} u_0}{\rho}\right)^*\mathbb{L} \\
        \frac{ \partial_{x_2} u_0}{\rho} & \mathbb{v}_0^N - \mathbb{v}_0
    \end{bmatrix}$. 
Then,
\begin{align}
    \|I_d - \mathbb{A}\mathbb{Q}(0,u_0)\|_{H_1} &\leq \left\|I_d - \mathbb{A}(\overline{\mathbb{L}} + \mathbb{M}(0,u_0))\right\|_{H_1} + \left\|\mathbb{A}(\mathbb{M}(0,u_0) - \mathcal{M}(0,u_0))\right\|_{H_1} \\
    &\leq \left\|I_d - \mathbb{B}(I_d + \mathbb{M}(0,u_0)\overline{\mathbb{L}^{-1}})\right\|_{H_2} + \|\mathbb{B}\|_{H_2} \|(\mathbb{M}(0,u_0) - \mathcal{M}(0,u_0))\overline{\mathbb{L}}^{-1}\|_{H_2}. 
    \end{align}
For the first term, observe that
{\small\begin{align}
    \left\|I_d - \mathbb{B}(I_d + \mathbb{M}(0,u_0)\overline{\mathbb{L}}^{-1})\right\|_{H_2}\leq \left\|I_d - B(I_d + M(0,U_0)\overline{L}^{-1})\right\|_{X_2} + C_{B} \|(\Gamma_{\mathbb{Z}_2 \times \mathbb{Z}_1}^\dagger(L^{-1}) - \mathbb{L}^{-1})\mathbb{v}_0^N\|_{2},
\end{align}}
which leads to the bounds $Z_1$ and $\mathcal{Z}_u$.
For the second term, observe that
\begin{align}
    \|(\mathbb{M}(0,u_0) - \mathcal{M}(0,u_0))\overline{\mathbb{L}}^{-1}\|_{H_2} \leq \varphi\left(0,\frac{\|\partial_{x_2} u_0^N - \partial_{x_2} u_0\|_2}{\rho}, \frac{\|\partial_{x_2} u_0^N - \partial_{x_2} u_0\|_2}{\rho}, \|(\mathbb{v}_0^N - \mathbb{v}_0)\mathbb{L}^{-1}\|_2\right).
\end{align}
Now, we examine each term individually. Observe that
{\small\begin{align}
    \|\partial_{x_2} u_0^N - \partial_{x_2} u_0\|_{2} =  \sqrt{|\om|}\|\partial_{x_2} U_0^N - \partial_{x_2} U_0\|_{2},~~\|(\mathbb{v}_0^N - \mathbb{v}_0)\mathbb{L}^{-1}\|_{2} \leq \|\mathbb{L}^{-1}\|_{2} \|\mathbb{V}_0^N - \mathbb{V}_0\|_{2} \leq \frac{1}{\mu} \|V_0^N - V_0\|_{1}.
\end{align}}
Hence, we define
{\small\begin{align}
    \mathcal{Z}_1 \bydef Z_1 + C_B\mathcal{Z}_u + \|\mathbb{B}\|_{H_2}\varphi\left(0,\frac{\sqrt{|\om|}}{\rho}\|\partial_{x_2} U_0^N - \partial_{x_2} U_0\|_{2}, \frac{\sqrt{|\om|}}{\rho}\|\partial_{x_2} U_0^N - \partial_{x_2} U_0\|_{2}, \frac{1}{\mu} \|V_0^N - V_0\|_{1}\right).
\end{align}}
 
\end{proof}
\subsubsection{Proof of Lemma \ref{lem : Z1_extra}}\label{apen : Z1 periodic proof}
In this appendix, we prove Lemma \ref{lem : Z1_extra}. In particular, we are computing the $Z_1$ bound when we have an extra equation. We now perform the proof.
% \begin{lemma}
% Let $M^N \bydef \pi^N + M(0,U_0)\begin{bmatrix}
%     1 & 0 \\
%     0 & L^{-1}
% \end{bmatrix}$and $M \bydef \pi_N + M^N$. Let $Z_1 > 0$ be such that   
% \begin{align}
%     Z_1 \geq \varphi(Z_{1,1},Z_{1,2},Z_{1,3},Z_{1,4})
% \end{align}
% where $\varphi$ is defined in Lemma 4.1 of \cite{gs_cadiot_blanco} and
% \begin{align}
%     &Z_{1,1} \bydef \sqrt{\|(\pi^N - B^NM^N) (\pi^N - (M^N)^{\star} (B^N)^{\star})\|_{2}} \\
%     &Z_{1,2} \bydef \max_{n \in \mathbb{Z}^2 \times \mathbb{Z}^2 \setminus I^N} \frac{1}{|l(\tilde{n})|} \sqrt{\| B^NM(0,U_0)\pi_N M(0,U_0)^{\star}(B^N)^{\star}\|_{X_2}} \\
%     &Z_{1,3} \bydef \sqrt{\left\|\pi^NL^{-1}\mathbb{V}_0^N\pi_N\mathbb{V}_0^N L^{-1} \pi^N\right\|_{2}} \\
%     &Z_{1,4} \bydef \max_{n \in \mathbb{Z}^2 \times \mathbb{Z}^2 \setminus I^N} \frac{1}{|l(\tilde{n})|} \|V_0^N\|_{1}.
% \end{align}
% Then, $\left\|I - B(I_d + M(0,U_0)\begin{bmatrix}
%     1 & 0 \\
%     0 & L^{-1}
% \end{bmatrix}\right\|_{X_2} \leq Z_1$.
% \end{lemma}
\begin{proof}
By choosing $M^N$ as in the statement of Lemma \ref{lem : Z1_extra}, we are in the setup of \cite{gs_cadiot_blanco} for computing $Z_1$ for a system. Observe that
{\small\begin{align}
    I - BM 
    &= \begin{bmatrix}
        \pi^N - B^NM^N & - B^NM\pi_N \\
        -\pi_N M\pi^N & \pi_N - \pi_N M\pi_N
    \end{bmatrix} = \begin{bmatrix}
         \pi^N - B^NM^N & -B^N M(0,U_0) \overline{L}^{-1} \pi_N \\ 
        -\pi_N M(0,U_0) \overline{L}^{-1}\pi^N & -\pi_N M(0,U_0) \overline{L}^{-1} \pi_N
    \end{bmatrix}\label{step_in_Z1}
\end{align}}
where 
$\overline{L}^{-1} \bydef \begin{bmatrix}
    1 & 0 \\
    0 & L^{-1} \end{bmatrix}$.
Now, we will use Lemma 4.4 from \cite{gs_cadiot_blanco} on \eqref{step_in_Z1}. This means we must compute the norm of each of the blocks of \eqref{step_in_Z1}. We begin with $\pi^N - B^NM^N$.
\begin{align}
     \|\pi^N - B^NM^N\|_{X_2}^2 &= \|(\pi^N - B^NM^N) (\pi^N - (M^N)^{\star} (B^N)^{\star})\|_{X_2} \bydef Z_{1,1}^2.\label{definition_of_Z11}
\end{align}
Next, we compute $\|-B^NM(0,U_0)\overline{L}^{-1}\pi_N\|_{X_2}$.
\begin{align}
    \nonumber \|- B^N M(0,U_0) \overline{L}^{-1} \pi_N\|_{X_2}^2 
    &\leq \|\pi_N \overline{L}^{-1}\|_{X_2}^2 \|\pi_N M(0,U_0)^{\star} (B^N)^{\star}\|_{X_2}^2 \\ 
    &= \|\pi_N L^{-1}\|_{2}^2 \| B^NM(0,U_0)\pi_N M(0,U_0)^{\star}(B^N)^{\star}\|_{X_2}\bydef Z_{1,2}^2\label{step_in_Z1_2} \end{align}
% Now, observe that
% \begin{align}
%     \pi_N M(0,U_0)^{\star} = \begin{bmatrix}
%         0 & 0 \\
%         0 & \pi_N
%     \end{bmatrix}\begin{bmatrix}
%         0 & 
%          \left(\frac{\partial_{x_2} U_0^N}{\rho}\right)^* \\
%         (\left(\frac{\partial_{x_2} U_0^N}{\rho}\right)^* L)^* &  (\mathbb{V}_0^N)^{\star}
%     \end{bmatrix} &= \begin{bmatrix}
%         0 & 
%          0 \\
%         \pi_N \frac{L \partial_{x_2} U_0^N}{\rho} &  \pi_N \mathbb{V}_0^N
%     \end{bmatrix} = \begin{bmatrix}
%         0 & 
%          0 \\
%         0 &  \pi_N \mathbb{V}_0^N
%     \end{bmatrix}.
% \end{align}
% Hence,
% \begin{align}
%     &M(0,U_0) \pi_N M(0,U_0)^{*}
%     = \begin{bmatrix}
%         0 & 0 \\
%         0 &  \mathbb{V}_0^N\pi_N \mathbb{V}_0^N 
%     \end{bmatrix}
% \end{align}
% which can be analyzed in the same way as the usual $Z_{1,2}$ bound in \eqref{step_in_Z1_2_isolated}.
\par Next, we will consider $\|-\pi_N M(0,U_0) \overline{L}^{-1} \pi^N\|_{X_2}$.
\begin{align}
    \|-\pi_N M(0,U_0) \overline{L}^{-1} \pi^N\|_{X_2}   
    &= \|\pi_N \mathbb{V}_0^NL^{-1}\pi^N\|_{2} = \sqrt{\left\|\pi^NL^{-1}\mathbb{V}_0^N\pi_N\mathbb{V}_0^N L^{-1} \pi^N\right\|_{2}} \bydef Z_{1,3}^2\label{step_in_Z1_3}
\end{align}
Observe that this is the same $Z_{1,3}$ bound as in Lemma \ref{lem : Z_1 periodic}, particularly in \eqref{step_in_Z1_3_isolated}. Lastly, we will consider $\|-\pi_N M(0,U_0) \overline{L}^{-1} \pi_N\|_{X_2}$.
{\footnotesize\begin{align}
    \|-\pi_N M(0,U_0) \overline{L}^{-1} \pi_N\|_{X_2} 
    \leq \|\pi_N L^{-1}\|_{2} \left\|\pi_N \mathbb{V}_0^N\pi_N\right\|_{2} \leq \|\pi_N L^{-1}\|_{2} \|V_0^N\|_{1} \bydef Z_{1,4}\label{step_in_Z1_4}
\end{align}}
where the last step followed by Young's inequality. Observe that this is the same $Z_{1,4}$ bound as in Lemma \ref{lem : Z_1 periodic}, particularly in \eqref{step_in_Z1_4_isolated}. Combining \eqref{step_in_Z1}, \eqref{step_in_Z1_2}, \eqref{step_in_Z1_3}, and \eqref{step_in_Z1_4}, we conclude.
\end{proof}
\subsection{Computation of \texorpdfstring{$\mathcal{Z}_s$}{Zs}}\label{apen : Zs nonisolate}
In this appendix, we wish to prove Lemma \ref{lem : Zs_extra}. This will provide us with the extra term required to compute the $\mathcal{Z}_s$ bound due to the extra equation. We now prove Lemma \ref{lem : Zs_extra}.
% \begin{lemma}
% Let $\mathcal{Z}_s \geq 0$ be such that
% \begin{align}
%     \mathcal{Z}_s \bydef \|\mathbb{B}\|_{H_2} \varphi(0,\mathcal{Z}_{s,1},\mathcal{Z}_{s,1},\mathcal{Z}_{s,2})
% \end{align}
% where
% \begin{align}
%     &\mathcal{Z}_{s,1} \bydef \frac{\kappa_{\partial}}{j\rho} \phi\left(\mathbb{L}_1u_0, R_{\frac{2\pi }{j}}\right) \sum_{k = 0}^{j-1} \mathcal{R}_{k,j}\\
%     &\mathcal{Z}_{s,2} \bydef 2\kappa (|\nu_1| + 3|\nu_2| \|U_0\|_{1}) \frac{1}{j} \phi\left(u_0, R_{\frac{2\pi }{m}}\right)\sum_{k = 0}^{j-1} \mathcal{R}_{k,j}.
% \end{align}
% Then, $\|\mathbb{A}(\mathbb{Q}(0,u_0) - D\mathbb{F}(0,w_0))\|_{H_1} \leq \mathcal{Z}_s$.
% \end{lemma}
\begin{proof}
We begin from the definition,
% {\small\begin{align}
%     \|\mathbb{A}(\mathbb{Q}(0,u_0) - D\mathbb{F}(0,w_0))\|_{H_1} &= \left\|\mathbb{B}\left(\begin{bmatrix}
%         0 & \left(\frac{\partial_{x_2} u_0}{\rho}\right)^* \mathbb{L} \\
%         \frac{\partial_{x_2} u_0}{\rho} & D\mathbb{f}(0,u_0)
%     \end{bmatrix} - \begin{bmatrix}
%         0 & \left(\frac{\partial_{x_2} w_0}{\rho}\right)^* \mathbb{L} \\
%         \frac{\partial_{x_2} w_0}{\rho} & D\mathbb{f}(0,w_0)
%     \end{bmatrix}\right)\begin{bmatrix}
%         1 & 0 \\
%         0 & \mathbb{L}^{-1}
%     \end{bmatrix}\right\|_{H_2} \\
%     &\leq \|\mathbb{B}\|_{H_2} \left\| \begin{bmatrix}
%         0 & \left(\frac{\partial_{x_2} (u_0 - w_0)}{\rho}\right)^* \\
%         \frac{\partial_{x_2} (u_0 - w_0)}{\rho} & (D\mathbb{G}(0,u_0) - D\mathbb{G}(0,w_0))\mathbb{L}^{-1}
%     \end{bmatrix}\right\|_{H_2} \\
%     &\hspace{-2cm}\leq \|\mathbb{B}\|_{H_2} \varphi\left(0,\frac{\|\partial_{x_2}(u_0 - w_0)\|_2}{\rho}, \frac{\|\partial_{x_2}(u_0 - w_0)\|_2}{\rho},\|(D\mathbb{G}(0,u_0) - D\mathbb{G}(0,w_0))\mathbb{L}^{-1}\|_2\right)
% \end{align}}
{\small\begin{align}
    \|\mathbb{A}(\mathbb{Q}(0,u_0) - D\mathbb{F}(0,w_0))\|_{H_1} 
    &\leq \|\mathbb{B}\|_{H_2} \left\| \begin{bmatrix}
        0 & \left(\frac{\partial_{x_2} (u_0 - w_0)}{\rho}\right)^* \\
        \frac{\partial_{x_2} (u_0 - w_0)}{\rho} & (D\mathbb{G}(0,u_0) - D\mathbb{G}(0,w_0))\mathbb{L}^{-1}
    \end{bmatrix}\right\|_{H_2} \\
    &\hspace{-2cm}\leq \|\mathbb{B}\|_{H_2} \varphi\left(0,\frac{\|\partial_{x_2}(u_0 - w_0)\|_2}{\rho}, \frac{\|\partial_{x_2}(u_0 - w_0)\|_2}{\rho},\|(D\mathbb{G}(0,u_0) - D\mathbb{G}(0,w_0))\mathbb{L}^{-1}\|_2\right),
\end{align}}
where $\varphi$ is defined in Lemma 4.2 of \cite{gs_cadiot_blanco}. The terms $\|\partial_{x_2} (u_0 - w_0)\|_{2},\|(D\mathbb{G}(0,u_0) - D\mathbb{G}(0,w_0))\mathbb{L}^{-1}\|_{2}$ can be bounded using Lemmas \ref{lem : kappa partial} and \ref{lem : Zs bound} respectively. Hence, we obtain
\begin{align}
    &\frac{\|\partial_{x_2} (u_0 - w_0)\|_{2}}{\rho} \leq \frac{\kappa_{\partial}}{j\rho} \phi\left(\mathbb{L}_1u_0, R_{\frac{2\pi }{j}}\right)\sum_{k = 0}^{j-1} \mathcal{R}_{k,j} \bydef \mathcal{Z}_{s,1}\\
    &\|(D\mathbb{G}(0,u_0) - D\mathbb{G}(0,w_0))\mathbb{L}^{-1}\|_{2} \leq 2\kappa (\nu_1 + 3\nu_2 \|U_0\|_{1}) \frac{1}{j} \phi\left(u_0, R_{\frac{2\pi }{j}}\right)\sum_{k = 0}^{j-1} \mathcal{R}_{k,j} \bydef \mathcal{Z}_{s,2}.
\end{align}
This provides the desired result.
\end{proof}
\subsection{Computation of \texorpdfstring{$\mathcal{Z}_2$}{Z2}}\label{apen : Z2 nonisolate}
In this appendix, we will examine the $\mathcal{Z}_2$ bound when $\mathcal{H}$ does not isolate the solution. Here, the primary difference stems from the structure of $\mathbb{B}$. Since $\mathbb{B}$ is now an operator from $\mathbb{R} \times L^2_{\mathbb{Z}_2 \times \mathbb{Z}_1}$ to itself, we must take this into account when performing the computation. This will result in two quantities, both similar to those defined in Lemma \ref{lem : bound Z_2 s}, but slightly different. We now prove Lemma \ref{lem : Z_2_extra}.
% \begin{lemma}
% Let $\kappa, W, \mathbb{W},\mathcal{Z}_{2,1}(r),$ and $\mathcal{Z}_{2,2}(r)$ be defined as in Lemma \ref{lem : bound Z_2 s}. Let $\mathcal{Z}_{2,3}(r): [0,\infty) \to (0,\infty)$ be a non-negative function such that
% \begin{align}
%     &\mathcal{Z}_{2,3}(r) \bydef  \frac{\kappa}{\mu} (\|\mathbb{W}\|_{2}^2 + \|W\|_{1}^2)^{\frac{1}{2}} r\\
%     &\mathcal{Z}_{2,4}(r) \bydef  \frac{6|\nu_2|  \kappa^2}{j} \phi\left(u_0, R_{\frac{2\pi }{j}}\right)\sum_{k=0}^{j-1} \mathcal{R}_{k,j}.
% \end{align}
% Let $\mathcal{Z}_2(r) : [0,\infty) \to (0,\infty)$ be a non-negative function such that
% \begin{align}
%     \mathcal{Z}_2(r) \bydef \mathcal{Z}_{2,1}(r) + \mathcal{Z}_{2,2}(r) + \max(1,\|b_{12}^N\|_{2})(\mathcal{Z}_{2,3}(r) + \mathcal{Z}_{2,4}(r)).
% \end{align} 
% Then, $\|\mathbb{A}(D\mathbb{F}(h_1,h_2) - D\mathbb{F}(0,w_0))\|_{H_1} \leq \mathcal{Z}_2(r)r$.
% \end{lemma}
\begin{proof}
We start from the definition,
\begin{align}
    \|\mathbb{B}(D\mathbb{F}(h_1,h_2) - D\mathbb{F}(0,w_0))\|_{H_1,H_2}
    &= \left\| \begin{bmatrix}
        0 & \mathbb{b}_{12}^*(D\mathbb{f}(h_1,h_2) - D_u\mathbb{f}(0,w_0))\mathbb{L}^{-1} \\
        0 & \mathbb{b}_{22}(D\mathbb{f}(h_1,h_2) - D_u\mathbb{f}(0,w_0))\mathbb{L}^{-1}
    \end{bmatrix}\right\|_{H_2} \\ 
    &\hspace{-3cm}\leq \|\mathbb{b}_{12}\|_{2} \|D \mathbb{f}(h_1,h_2) - D \mathbb{f}(0,w_0)\|_{l,2} + \|\mathbb{b}_{22}(D \mathbb{f}(h_1,h_2) - D\mathbb{f}(0,w_0))\|_{l,2}.\label{refer step Z2 full}
\end{align}
Observe that, the second term in \eqref{refer step Z2 full} is the same bound that was previously referred to as $\mathcal{Z}_2$ in Lemma \ref{lem : bound Z_2 s}.
Therefore, we can conclude that
{\small\begin{align}
    \|\mathbb{B}(D\mathbb{F}(h_1,h_2) - D\mathbb{F}(0,w_0))\|_{H_1,H_2} \leq \|\mathbb{b}_{12}\|_{2} \|(D\mathbb{f}(h_1,h_2) - D\mathbb{f}(0,w_0))\mathbb{L}^{-1}\|_{2} + (\mathcal{Z}_{2,1}(r) + \mathcal{Z}_{2,2}(r))r.
\end{align}}
For the first term in \eqref{refer step Z2 full}, we follow the steps of Lemma \ref{lem : bound Z_2 s} but replacing $\mathbb{B}$ with $I_d$. 
{\footnotesize\begin{align}
    \|(D\mathbb{f}(h_1,h_2) - D\mathbb{f}(0,w_0))\mathbb{L}^{-1}\|_{2} \leq  \frac{\kappa}{\mu} (\|\mathbb{W}\|_{2}^2 + \|W\|_{1}^2)^{\frac{1}{2}} r + \frac{6|\nu_2|  \kappa^2}{j} \sum_{k=0}^{j-1} \mathcal{R}_{k,j}\phi\left(u_0, R_{\frac{2\pi }{j}}\right) r \bydef (\mathcal{Z}_{2,3}(r) + \mathcal{Z}_{2,4}(r))r.
\end{align}}
Hence, arrive at the desired result
{\small\begin{align}
    \|\mathbb{B}(D\mathbb{F}(h_1,h_2) - D\mathbb{F}(0,w_0))\|_{H_1,H_2} \leq (\mathcal{Z}_{2,1}(r) + \mathcal{Z}_{2,2}(r) + \max(1,\|b_{12}^N\|_{2})(\mathcal{Z}_{2,3}(r))+\mathcal{Z}_{2,4}(r))r \bydef \mathcal{Z}_2(r) r.
\end{align}}
\end{proof}
\nocite{julia_cadiot_Kawahara}
\nocite{julia_cadiot_sh}
\nocite{julia_cadiot_blanco_GS}
\nocite{julia_interval}
\bibliographystyle{abbrv}
\bibliography{biblio}
\end{document}